\documentclass[10pt]{amsart}
\usepackage{epic,eepic}
\usepackage{graphicx} 
\usepackage[mathscr]{eucal} 
\numberwithin{equation}{section}
\usepackage{amsthm}
\theoremstyle{plain}
\newtheorem{theorem}{Theorem}[section]
\newtheorem{lemma}[theorem]{Lemma}
\newtheorem{proposition}[theorem]{Proposition}

\theoremstyle{definition}
\newtheorem{definition}[theorem]{Definition}

\title[Periodicities of T and Y-systems II: Types $C_r$, $F_4$, and $G_2$]
{Periodicities of T and Y-systems,\\
dilogarithm identities,
and cluster algebras II:\\ Types $C_r$, $F_4$, and $G_2$}

\author[R.\  Inoue]{Rei Inoue}
\address{ R.\  Inoue:
Faculty of Pharmaceutical Sciences, Suzuka University of Medical Science,
Suzuka, 513-8670, Japan}

\author[O.\ Iyama]{Osamu Iyama}
\address{ O.\ Iyama:
 Graduate School of Mathematics, Nagoya University,
Nagoya, 464-8604, Japan}

\author[B.\ Keller]{Bernhard Keller}
\address{B.\ Keller:
Universit\'e Paris Diderot -- Paris 7,
UFR de Math\'ematiques,
Institut de Math\'ematiques de Jussieu, UMR 7586 du CNRS,
Case 7012,
2, place Jussieu,
75251 Paris Cedex 05,
France}

\author[A.\ Kuniba]{\\Atsuo Kuniba}
\address{ A.\ Kuniba:
Institute of Physics,
University of Tokyo,
Tokyo, 153-8902, Japan}

\author[T.\ Nakanishi]{Tomoki Nakanishi}
\address{ T.\ Nakanishi:
 Graduate School of Mathematics, Nagoya University,
Nagoya, 464-8604, Japan}

\begin{document}

\footnote[0]{2010 {\em Mathematics Subject Classification}.
Primary 13F60; Secondary 17B37.}

\begin{abstract}
We prove the
 periodicities
of the restricted T and Y-systems associated with
the quantum affine algebra of type $C_r$, $F_4$,
and $G_2$ at any level.
We also prove the dilogarithm identities
for these Y-systems  at
any level.
Our proof is based on
the tropical Y-systems
and the categorification of
the cluster algebra associated with 
any skew-symmetric matrix by Plamondon.
\end{abstract}

\maketitle

\tableofcontents

\section{Introduction}
This is the continuation of the paper \cite{IIKKN}.
In \cite{IIKKN}, we proved the periodicities of the restricted
T and Y-systems associated with the quantum affine algebra of
type $B_r$ at any level.
We also proved the dilogarithm  identities for these
Y-systems at any level.
Our proof was based on
the tropical Y-systems
and the categorification of
the cluster algebra associated with 
any skew-symmetric matrix by Plamondon \cite{P1,P2}.
In this paper, using the same method,
we prove the corresponding statements for
type $C_r$, $F_4$, and $G_2$,
thereby completing all the nonsimply laced types.

The results are basically parallel to type $B_r$.
Since the common method and the proofs of the statements
for type $B_r$ were described in \cite{IIKKN} in detail,
in this paper, we skip the proofs of  most statements,
and concentrate on presenting the results with emphasis
on the special features of each case.
Notably, the tropical Y-system at level 2,
which is the core part in the entire method,
is quite specific to each case.

While we try to make the paper as self-contained as
possible, we also try to minimize  duplication with
\cite{IIKKN}.
Therefore, we have to ask the reader's patience to refer to
the companion paper  \cite{IIKKN} for the things which are omitted.
In particular, basic  definitions for cluster algebras
are summarized in \cite[Section 2.1]{IIKKN}.

The organization of the paper is as follows.
In section 2 we present the main results
as well as the T and Y-systems for each type.
In Section 3 the results for type $C_r$ are presented.
The tropical Y-system at level 2 is the key 
and described in detail in Section 3.6.
In Section 4 the results for type $F_4$ are presented.
In Section 5 the results for type $G_2$ are presented.
In Section 6 we list the known mutation equivalences
of quivers corresponding to the T and Y-systems.


\section{Main results}

\subsection{Restricted T and Y-systems of types $C_r$,
$F_4$, and $G_2$}

\label{subsect:restrictedT}

Let $X_r$ be the Dynkin diagram of type $C_r$,
$F_4$, or $G_2$ with rank $r$,
and $I=\{1,\dots, r\}$ be the enumeration
of the vertices of $X_r$ as below.

%

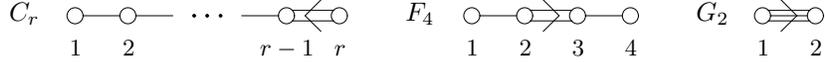
\begin{align*}
\begin{picture}(270,25)(0,-15)
%
%
\put(0,0){
\put(0,0){\circle{6}}
\put(20,0){\circle{6}}
\put(80,0){\circle{6}}
\put(100,0){\circle{6}}
\put(45,0){\circle*{1}}
\put(50,0){\circle*{1}}
\put(55,0){\circle*{1}}
\drawline(3,0)(17,0)
\drawline(23,0)(37,0)
\drawline(63,0)(77,0)
\drawline(82,-2)(98,-2)
\drawline(82,2)(98,2)
\drawline(87,0)(93,-6)
\drawline(87,0)(93,6)
\put(-25,-2){$C_r$}
\put(-2,-15){\small $1$}
\put(18,-15){\small $2$}
\put(70,-15){\small $ r-1$}
\put(98,-15){\small $r$}
}
%
\put(150,0){
\put(0,0){\circle{6}}
\put(20,0){\circle{6}}
\put(40,0){\circle{6}}
\put(60,0){\circle{6}}
\drawline(3,0)(17,0)
\drawline(43,0)(57,0)
\drawline(22,-2)(38,-2)
\drawline(22,2)(38,2)
\drawline(27,6)(33,0)
\drawline(27,-6)(33,0)
\put(-25,-2){$F_4$}
\put(-2,-15){\small $1$}
\put(18,-15){\small $2$}
\put(38,-15){\small $3$}
\put(58,-15){\small $4$}
}
%
\put(260,0){
\put(0,0){\circle{6}}
\put(20,0){\circle{6}}
\drawline(3,0)(17,0)
\drawline(2,-2)(18,-2)
\drawline(2,2)(18,2)
\drawline(7,6)(13,0)
\drawline(7,-6)(13,0)
\put(-25,-2){$G_2$}
\put(-2,-15){\small $1$}
\put(18,-15){\small $2$}
}
\end{picture}
\end{align*}
Let $h$ and $h^{\vee}$ be 
the Coxeter number and the dual Coxeter number of $X_r$,
respectively.
\begin{align}
\begin{tabular}{c|ccc}
  $X_r$ &   $C_r$ & $F_4$ & $G_2$ \\
  \hline
 $h$& $2r$&$12$&$6$\\
 $h^{\vee}$& $r+1$&$9$&$4$\\
\end{tabular}
\end{align}

We set  numbers $t$ and $t_a$ ($a\in I$) by
\begin{align}
\label{eq:t1}
t=
\begin{cases}
2 & \text{$X_r=C_r$, $F_4$}\\
3 & \text{$X_r=G_2$},
\end{cases}
\quad t_a=
\begin{cases}
1 & \text{$\alpha_a$: long root}\\
t & \text{$\alpha_a$: short root}.
\end{cases}
\end{align}

For a given integer $\ell \geq 2$,
we introduce a set of triplets
\begin{align}
\mathcal{I}_{\ell}=\mathcal{I}_{\ell}(X_r):=
\{(a,m,u)\mid
a\in I ; m=1,\dots,t_a\ell-1;
u\in \frac{1}{t}\mathbb{Z}
\}.
\end{align}

\begin{definition}[\cite{KNS}]
\label{defn:RT}
Fix an integer $\ell \geq 2$.
The {\it level $\ell$ restricted T-system $\mathbb{T}_{\ell}(X_r)$
of type $X_r$
(with the unit boundary condition)}
is the following system of relations
for
a family of variables $T_{\ell}=\{T^{(a)}_m(u)
\mid (a,m,u)\in \mathcal{I}_{\ell}
\}$,
where 
$T^{(0)}_m (u)=T^{(a)}_0 (u)=1$,
and furthermore,  $T^{(a)}_{t_a\ell}(u)=1$
(the {\em unit boundary condition\/}) if they occur
in the right hand sides in the relations:

(Here and throughout the paper,
$2m$ (resp.\ $2m+1$) in the left hand sides, for example,
represents elements  $ 2,4,\dots$
(resp.\ $1,3,\dots$).)

For $X_r=C_r$,
\begin{align}
\begin{split}
\label{eq:TC1}
T^{(a)}_m\left(u-\textstyle\frac{1}{2}\right)
T^{(a)}_m\left(u+\textstyle\frac{1}{2}\right)
&=
T^{(a)}_{m-1}(u)T^{(a)}_{m+1}(u)
+T^{(a-1)}_{m}(u)T^{(a+1)}_{m}(u)\\
&\hskip120pt
 (1\leq a\leq r-2),\\
T^{(r-1)}_{2m}\left(u-\textstyle\frac{1}{2}\right)
T^{(r-1)}_{2m}\left(u+\textstyle\frac{1}{2}\right)
&=
T^{(r-1)}_{2m-1}(u)T^{(r-1)}_{2m+1}(u)\\
&\qquad
+
T^{(r-2)}_{2m}(u)
T^{(r)}_{m}\left(u-\textstyle\frac{1}{2}\right)
T^{(r)}_{m}\left(u+\textstyle\frac{1}{2}\right),\\
T^{(r-1)}_{2m+1}\left(u-\textstyle\frac{1}{2}\right)
T^{(r-1)}_{2m+1}\left(u+\textstyle\frac{1}{2}\right)
&=
T^{(r-1)}_{2m}(u)T^{(r-1)}_{2m+2}(u)\\
&\qquad
+
T^{(r-2)}_{2m+1}(u)
T^{(r)}_{m}(u)T^{(r)}_{m+1}(u),\\
T^{(r)}_{m}(u-1)
T^{(r)}_{m}(u+1)
&=
T^{(r)}_{m-1}(u)T^{(r)}_{m+1}(u)
+
T^{(r-1)}_{2m}(u).
\end{split}
\end{align}

For $X_r=F_4$,

\begin{align}
\begin{split}
\label{eq:TF1}
T^{(1)}_m(u-1)T^{(1)}_m(u+1)
&=
T^{(1)}_{m-1}(u)T^{(1)}_{m+1}(u)
+T^{(2)}_{m}(u),\\
T^{(2)}_m(u-1)T^{(2)}_m(u+1)
&=
T^{(2)}_{m-1}(u)T^{(2)}_{m+1}(u)
+
T^{(1)}_{m}(u)T^{(3)}_{2m}(u),\\
T^{(3)}_{2m}\left(u-\textstyle\frac{1}{2}\right)
T^{(3)}_{2m}\left(u+\textstyle\frac{1}{2}\right)
&=
T^{(3)}_{2m-1}(u)T^{(3)}_{2m+1}(u)\\
&\qquad
+
T^{(2)}_{m}\left(u-\textstyle\frac{1}{2}\right)
T^{(2)}_{m}\left(u+\textstyle\frac{1}{2}\right)
T^{(4)}_{2m}(u),\\
T^{(3)}_{2m+1}\left(u-\textstyle\frac{1}{2}\right)
T^{(3)}_{2m+1}\left(u+\textstyle\frac{1}{2}\right)
&=
T^{(3)}_{2m}(u)T^{(3)}_{2m+2}(u)
+
T^{(2)}_{m}(u)T^{(2)}_{m+1}(u)
T^{(4)}_{2m+1}(u),\\
T^{(4)}_{m}\left(u-\textstyle\frac{1}{2}\right)
T^{(4)}_{m}\left(u+\textstyle\frac{1}{2}\right)
&=
T^{(4)}_{m-1}(u)T^{(4)}_{m+1}(u)
+T^{(3)}_m(u).
\end{split}
\end{align}

For $X_r=G_2$,
\begin{align}
\begin{split}
\label{eq:TG1}
T^{(1)}_m(u-1)T^{(1)}_m(u+1)
&=
T^{(1)}_{m-1}(u)T^{(1)}_{m+1}(u)
+
T^{(2)}_{3m}(u),\\
T^{(2)}_{3m}\left(u-\textstyle\frac{1}{3}\right)
T^{(2)}_{3m}\left(u+\textstyle\frac{1}{3}\right)
&=
T^{(2)}_{3m-1}(u)T^{(2)}_{3m+1}(u)\\
&\qquad
+
T^{(1)}_{m}\left(u-\textstyle\frac{2}{3}\right)
T^{(1)}_m(u)
T^{(1)}_{m}\left(u+\textstyle\frac{2}{3}\right),\\
T^{(2)}_{3m+1}\left(u-\textstyle\frac{1}{3}\right)
T^{(2)}_{3m+1}\left(u+\textstyle\frac{1}{3}\right)
&=
T^{(2)}_{3m}(u)T^{(2)}_{3m+2}(u)\\
&\qquad+
T^{(1)}_{m}\left(u-\textstyle\frac{1}{3}\right)
T^{(1)}_{m}\left(u+\textstyle\frac{1}{3}\right)
T^{(1)}_{m+1}(u),\\
T^{(2)}_{3m+2}\left(u-\textstyle\frac{1}{3}\right)
T^{(2)}_{3m+2}\left(u+\textstyle\frac{1}{3}\right)
&=
T^{(2)}_{3m+1}(u)T^{(2)}_{3m+3}(u)\\
&\qquad+
T^{(1)}_{m}(u)
T^{(1)}_{m+1}\left(u-\textstyle\frac{1}{3}\right)
T^{(1)}_{m+1}\left(u+\textstyle\frac{1}{3}\right).
\end{split}
\end{align}
\end{definition}

\begin{definition}[\cite{KN}]
\label{defn:RY}
Fix an integer $\ell \geq 2$.
The {\it level $\ell$ restricted Y-system $\mathbb{Y}_{\ell}(X_r)$
of type $X_r$}
is the following system of relations
for a family of variables $Y_{\ell}=\{Y^{(a)}_m(u)
\mid
(a,m,u)\in \mathcal{I}_{\ell}
\}$,
where 
$Y^{(0)}_m (u)=Y^{(a)}_0 (u)^{-1}
=Y^{(a)}_{t_a\ell}(u)^{-1}=0$
 if they occur
in the right hand sides in the relations:

For $X_r=C_r$,
\begin{align}
\begin{split}
\label{eq:YC1}
Y^{(a)}_m\left(u-\textstyle\frac{1}{2}\right)
Y^{(a)}_m\left(u+\textstyle\frac{1}{2}\right)
&=
\frac{
(1+Y^{(a-1)}_{m}(u))(1+Y^{(a+1)}_{m}(u))
}{
(1+Y^{(a)}_{m-1}(u)^{-1})(1+Y^{(a)}_{m+1}(u)^{-1})
}\\
&\hskip130pt
 (1\leq a\leq r-2),\\
Y^{(r-1)}_{2m}\left(u-\textstyle\frac{1}{2}\right)
Y^{(r-1)}_{2m}\left(u+\textstyle\frac{1}{2}\right)
&=
\frac{
(1+Y^{(r-2)}_{2m}(u))(1+Y^{(r)}_{m}(u))
}{
(1+Y^{(r-1)}_{2m-1}(u)^{-1})(1+Y^{(r-1)}_{2m+1}(u)^{-1})
},\\
Y^{(r-1)}_{2m+1}\left(u-\textstyle\frac{1}{2}\right)
Y^{(r-1)}_{2m+1}\left(u+\textstyle\frac{1}{2}\right)
&=
\frac{
1+Y^{(r-2)}_{2m+1}(u)
}{
(1+Y^{(r-1)}_{2m}(u)^{-1})(1+Y^{(r-1)}_{2m+2}(u)^{-1})
},\\
Y^{(r)}_{m}(u-1)
Y^{(r)}_{m}(u+1)
&=
\frac{
\begin{array}{l}
\textstyle
(1+Y^{(r-1)}_{2m+1}(u))
(1+Y^{(r-1)}_{2m-1}(u))\\
\textstyle
\quad\times(1+Y^{(r-1)}_{2m}\left(u-\frac{1}{2}\right))
(1+Y^{(r-1)}_{2m}\left(u+\frac{1}{2}\right))
\end{array}
}
{
(1+Y^{(r)}_{m-1}(u)^{-1})(1+Y^{(r)}_{m+1}(u)^{-1})
}.
\end{split}
\end{align}

For $X_r=F_4$,
\begin{align}
\begin{split}
Y^{(1)}_m(u-1)Y^{(1)}_m(u+1)
&=
\frac{
1+Y^{(2)}_{m}(u)
}{
(1+Y^{(1)}_{m-1}(u)^{-1})(1+Y^{(1)}_{m+1}(u)^{-1})
},\\
Y^{(2)}_m(u-1)Y^{(2)}_m(u+1)
&=
\frac{
{
\begin{array}{l}
\textstyle
(1+Y^{(1)}_{m}(u))
(1+Y^{(3)}_{2m-1}(u))
(1+Y^{(3)}_{2m+1}(u))\\
\textstyle
\quad\times(1+Y^{(3)}_{2m}\left(u-\frac{1}{2}\right))
(1+Y^{(3)}_{2m}\left(u+\frac{1}{2}\right))
\end{array}
}
}
{
(1+Y^{(2)}_{m-1}(u)^{-1})(1+Y^{(2)}_{m+1}(u)^{-1})
},
\\
Y^{(3)}_{2m}\left(u-\textstyle\frac{1}{2}\right)
Y^{(3)}_{2m}\left(u+\textstyle\frac{1}{2}\right)
&=
\frac{
(1+Y^{(2)}_{m}(u))(1+Y^{(4)}_{2m}(u))
}{
(1+Y^{(3)}_{2m-1}(u)^{-1})(1+Y^{(3)}_{2m+1}(u)^{-1})
},\\
Y^{(3)}_{2m+1}\left(u-\textstyle\frac{1}{2}\right)
Y^{(3)}_{2m+1}\left(u+\textstyle\frac{1}{2}\right)
&=
\frac{
1+Y^{(4)}_{2m+1}(u)
}{
(1+Y^{(3)}_{2m}(u)^{-1})(1+Y^{(3)}_{2m+2}(u)^{-1})
},\\
Y^{(4)}_{m}\left(u-\textstyle\frac{1}{2}\right)
Y^{(4)}_{m}\left(u+\textstyle\frac{1}{2}\right)
&=
\frac{
1+Y^{(3)}_{m}(u)
}{
(1+Y^{(4)}_{m-1}(u)^{-1})(1+Y^{(4)}_{m+1}(u)^{-1})
}.
\end{split}
\end{align}

For $X_r=G_2$,
\begin{align}
\begin{split}
\label{eq:YG1}
Y^{(1)}_m(u-1)Y^{(1)}_m(u+1)
&=
\frac{
{
\begin{array}{l}
\textstyle
(1+Y^{(2)}_{3m-2}(u))
(1+Y^{(2)}_{3m+2}(u))\\
\textstyle
\times(1+Y^{(2)}_{3m-1}\left(u-\frac{1}{3}\right))
(1+Y^{(2)}_{3m-1}\left(u+\frac{1}{3}\right))\\
\textstyle
\times(1+Y^{(2)}_{3m+1}\left(u-\frac{1}{3}\right))
(1+Y^{(2)}_{3m+1}\left(u+\frac{1}{3}\right))\\
\textstyle
\times(1+Y^{(2)}_{3m}\left(u-\frac{2}{3}\right))
(1+Y^{(2)}_{3m}\left(u+\frac{2}{3}\right))\\
\times (1+Y^{(2)}_{3m}\left(u\right))
\end{array}
}
}
{
(1+Y^{(1)}_{m-1}(u)^{-1})(1+Y^{(1)}_{m+1}(u)^{-1})
},
\\
Y^{(2)}_{3m}\left(u-\textstyle\frac{1}{3}\right)
Y^{(2)}_{3m}\left(u+\textstyle\frac{1}{3}\right)
&=
\frac{1+Y^{(1)}_m(u)}
{
(1+Y^{(2)}_{3m-1}(u)^{-1})(1+Y^{(2)}_{3m+1}(u)^{-1})
},
\\
Y^{(2)}_{3m+1}\left(u-\textstyle\frac{1}{3}\right)
Y^{(2)}_{3m+1}\left(u+\textstyle\frac{1}{3}\right)
&=
\frac{1}
{
(1+Y^{(2)}_{3m}(u)^{-1})(1+Y^{(2)}_{3m+2}(u)^{-1})
},
\\
Y^{(2)}_{3m+2}\left(u-\textstyle\frac{1}{3}\right)
Y^{(2)}_{3m+2}\left(u+\textstyle\frac{1}{3}\right)
&=
\frac{1}
{
(1+Y^{(2)}_{3m+1}(u)^{-1})(1+Y^{(2)}_{3m+3}(u)^{-1})
}.
\end{split}
\end{align}
\end{definition}

Let us write \eqref{eq:TC1}--\eqref{eq:TG1} in a unified manner
\begin{align}
\label{eq:Tu}
T^{(a)}_{m}\left(u-\textstyle\frac{1}{t_a}\right)
T^{(a)}_{m}\left(u+\textstyle\frac{1}{t_a}\right)
&=
T^{(a)}_{m-1}(u)T^{(a)}_{m+1}(u)
+
\prod_{(b,k,v)\in \mathcal{I}_{\ell}}
T^{(b)}_{k}(v)^{G(b,k,v;a,m,u)}.
\end{align}
Define the transposition
$^{t}G(b,k,v;a,m,u)=G(a,m,u;b,k,v)$.
Then, \eqref{eq:YC1}--\eqref{eq:YG1} are written as
\begin{align}
\label{eq:Yu}
Y^{(a)}_{m}\left(u-\textstyle\frac{1}{t_a}\right)
Y^{(a)}_{m}\left(u+\textstyle\frac{1}{t_a}\right)
&=
\frac{
\displaystyle
\prod_{(b,k,v)\in \mathcal{I}_{\ell}}
(1+Y^{(b)}_{k}(v))^{{}^t\!G(b,k,v;a,m,u)}
}
{
(1+Y^{(a)}_{m-1}(u)^{-1})(1+Y^{(a)}_{m+1}(u)^{-1})
}.
\end{align}

\subsection{Periodicities}


\begin{definition}
\label{defn:RT2}
Let  $\EuScript{T}_{\ell}(X_r)$
be the commutative ring over $\mathbb{Z}$ with identity
element, with generators
$T^{(a)}_m(u)^{\pm 1}$
 ($(a,m,u)\in \mathcal{I}_{\ell}$)
and relations $\mathbb{T}_{\ell}(X_r)$
together with $T^{(a)}_m(u)T^{(a)}_m(u)^{-1}=1$.
Let $\EuScript{T}^{\circ}_{\ell}(X_r)$ be
 the subring of $\EuScript{T}_{\ell}(X_r)$
generated by 
$T^{(a)}_m(u)$ 
 ($(a,m,u)\in \mathcal{I}_{\ell}$).
\end{definition}

\begin{definition}
\label{def:YB}
Let $\EuScript{Y}_{\ell}(X_r)$
be the semifield with generators
$Y^{(a)}_m(u)$
 ($(a,m,u)\in \mathcal{I}_{\ell}$
and relations $\mathbb{Y}_{\ell}(X_r)$.
Let $\EuScript{Y}^{\circ}_{\ell}(X_r)$
be the multiplicative subgroup
of $\EuScript{Y}_{\ell}(X_r)$
generated by
$Y^{(a)}_m(u)$, $1+Y^{(a)}_m(u)$
 ($(a,m,u)\in \mathcal{I}_{\ell}$).
(Here we use the symbol $+$ instead of $\oplus$ 
for simplicity.)
\end{definition}

The first main result of the paper is the periodicities
of the T and Y-systems.

\begin{theorem}[Conjectured by \cite{IIKNS}]
\label{thm:Tperiod}
The following relations hold in
$\EuScript{T}^{\circ}_{\ell}(X_r)$.
\par
(i) Half periodicity: $T^{(a)}_m(u+h^{\vee}+\ell)
=T^{(a)}_{t_a \ell -m}(u)$.
\par
(ii) Full periodicity: $T^{(a)}_m(u+2(h^{\vee}+\ell))
=T^{(a)}_m(u)$.
\end{theorem}

\begin{theorem}[Conjectured by \cite{KNS}]
\label{thm:Yperiod}
The following relations hold in
$\EuScript{Y}^{\circ}_{\ell}(X_r)$.
\par
(i) Half periodicity: $Y^{(a)}_m(u+h^{\vee}+\ell)
=Y^{(a)}_{t_a \ell -m}(u)$.
\par
(ii) Full periodicity: $Y^{(a)}_m(u+2(h^{\vee}+\ell))
=Y^{(a)}_m(u)$.
\end{theorem}

\subsection{Dilogarithm identities}
Let $L(x)$ be the {\em Rogers dilogarithm function\/}
\begin{align}
\label{eq:L0}
L(x)=-\frac{1}{2}\int_{0}^x 
\left\{ \frac{\log(1-y)}{y}+
\frac{\log y}{1-y}
\right\} dy
\quad (0\leq x\leq 1).
\end{align}

We introduce the {\em constant version\/} of the Y-system.

\begin{definition}
\label{defn:RYc}
Fix an integer $\ell \geq 2$.
The {\it level $\ell$ restricted constant Y-system
 $\mathbb{Y}^{\mathrm{c}}_{\ell}(X_r)$
of type $X_r$}
is the following system of relations
for a family of variables $Y^{\mathrm{c}}_{\ell}=\{Y^{(a)}_m
\mid
a\in I; m=1,\dots,t_a\ell-1 \}$,
where 
$Y^{(0)}_m =Y^{(a)}_0{}^{-1}=
Y^{(a)}_{t_a\ell}{}^{-1}=0$
 if they occur
in the right hand sides in the relations:

For $X_r=C_r$,
\begin{align}
\begin{split}
\label{eq:cYC1}
(Y^{(a)}_m)^2
&=
\frac{
(1+Y^{(a-1)}_{m})(1+Y^{(a+1)}_{m})
}{
(1+Y^{(a)}_{m-1}{}^{-1})(1+Y^{(a)}_{m+1}{}^{-1})
}\quad
 (1\leq a\leq r-2),\\
(Y^{(r-1)}_{2m})^2
&=
\frac{
(1+Y^{(r-2)}_{2m})(1+Y^{(r)}_{m})
}{
(1+Y^{(r-1)}_{2m-1}{}^{-1})(1+Y^{(r-1)}_{2m+1}{}^{-1})
},\\
(Y^{(r-1)}_{2m+1})^2
&=
\frac{
1+Y^{(r-2)}_{2m+1}
}{
(1+Y^{(r-1)}_{2m}{}^{-1})(1+Y^{(r-1)}_{2m+2}{}^{-1})
},\\
(Y^{(r)}_{m})^2
&=
\frac{
(1+Y^{(r-1)}_{2m-1})
(1+Y^{(r-1)}_{2m})^2
(1+Y^{(r-1)}_{2m+1})
}
{
(1+Y^{(r)}_{m-1}{}^{-1})(1+Y^{(r)}_{m+1}{}^{-1})
}.
\end{split}
\end{align}

For $X_r=F_4$,
\begin{align}
\begin{split}
(Y^{(1)}_m)^2
&=
\frac{
1+Y^{(2)}_{m}
}{
(1+Y^{(1)}_{m-1}{}^{-1})(1+Y^{(1)}_{m+1}{}^{-1})
},\\
(Y^{(2)}_m)^2
&=
\frac{
(1+Y^{(1)}_{m})
(1+Y^{(3)}_{2m-1})
(1+Y^{(3)}_{2m})^2
(1+Y^{(3)}_{2m+1})
}
{
(1+Y^{(2)}_{m-1}(u)^{-1})(1+Y^{(2)}_{m+1}(u)^{-1})
},
\\
(Y^{(3)}_{2m})^2
&=
\frac{
(1+Y^{(2)}_{m})(1+Y^{(4)}_{2m})
}{
(1+Y^{(3)}_{2m-1}{}^{-1})(1+Y^{(3)}_{2m+1}{}^{-1})
},\\
(Y^{(3)}_{2m+1})^2
&=
\frac{
1+Y^{(4)}_{2m+1}
}{
(1+Y^{(3)}_{2m}{}^{-1})(1+Y^{(3)}_{2m+2}{}^{-1})
},\\
(Y^{(4)}_{m})^2
&=
\frac{
1+Y^{(3)}_{m}
}{
(1+Y^{(4)}_{m-1}{}^{-1})(1+Y^{(4)}_{m+1}{}^{-1})
}.
\end{split}
\end{align}

For $X_r=G_2$,
\begin{align}
\begin{split}
\label{eq:cYG1}
(Y^{(1)}_m)^2
&=
\frac{
{
\begin{array}{l}
\textstyle
(1+Y^{(2)}_{3m-2})
(1+Y^{(2)}_{3m-1})^2
(1+Y^{(2)}_{3m})^3
(1+Y^{(2)}_{3m+1})^2
(1+Y^{(2)}_{3m+2})
\end{array}
}
}
{
(1+Y^{(1)}_{m-1}{}^{-1})(1+Y^{(1)}_{m+1}{}^{-1})
},
\\
(Y^{(2)}_{3m})^2
&=
\frac{1+Y^{(1)}_m}
{
(1+Y^{(2)}_{3m-1}{}^{-1})(1+Y^{(2)}_{3m+1}{}^{-1})
},
\\
(Y^{(2)}_{3m+1})^2
&=
\frac{1}
{
(1+Y^{(2)}_{3m}{}^{-1})(1+Y^{(2)}_{3m+2}{}^{-1})
},
\\
(Y^{(2)}_{3m+2})^2
&=
\frac{1}
{
(1+Y^{(2)}_{3m+1}{}^{-1})(1+Y^{(2)}_{3m+3}{}^{-1})
}.
\end{split}
\end{align}
\end{definition}

\begin{proposition}
There exists a unique positive real solution
of  $\mathbb{Y}^{\mathrm{c}}_{\ell}(X_r)$.
\end{proposition}
\begin{proof}
The same proof of \cite[Proposition 1.8]{IIKKN} is applicable.
\end{proof}

The second main result of the paper is the dilogarithm
identities conjectured
by Kirillov \cite[Eq.~(7)]{Ki}, properly corrected by
Kuniba \cite[Eqs.~(A.1a), (A.1c)]{Ku},

\begin{theorem}[Dilogarithm identities]
\label{thm:DI}
Suppose that a family of positive real numbers
$\{Y^{(a)}_m \mid a\in I; m=1,\dots,t_a\ell-1\}$
satisfies  $\mathbb{Y}^{\mathrm{c}}_{\ell}(X_r)$.
Then, we have the identity
\begin{align}
\label{eq:DI}
\frac{6}{\pi^2}\sum_{a\in I}
\sum_{m=1}^{t_a\ell-1}
L\left(\frac{Y^{(a)}_m}{1+Y^{(a)}_m}\right)
=
\frac{\ell \dim \mathfrak{g}}{h^{\vee}+\ell} - r,
\end{align}
where $\mathfrak{g}$ is
 the simple Lie algebra of type $X_r$.
\end{theorem}

The right hand side of \eqref{eq:DI} is
 equal to the number
\begin{align}
\label{eq:c1}
\frac{ r(\ell h - h^{\vee})}{h^{\vee}+\ell}.
\end{align}

In fact, we prove a functional generalization of
Theorem \ref{thm:DI}.

\begin{theorem}[Functional dilogarithm identities]
\label{thm:DI2}
Suppose that 
a family of positive real numbers
 $\{Y^{(a)}_m (u)\mid (a,m,u)\in \mathcal{I}_{\ell} \}$
satisfies $\mathbb{Y}_{\ell}(X_r)$.
Then, we have the identities
\begin{align}
\begin{split}
\label{eq:DI2}
\frac{6}{\pi^2}
\sum_{
\genfrac{}{}{0pt}{1}
{
(a,m,u)\in \mathcal{I}_{\ell}
}
{
0\leq u < 2(h^{\vee}+\ell)
}
}
L\left(
\frac{Y^{(a)}_m(u)}{1+Y^{(a)}_m(u)}
\right)
&=
2tr(\ell h-h^{\vee})\\
&=
\begin{cases}
4r(2r\ell -r-1)& C_r\\
48(4\ell-3)& F_4\\
24(3\ell - 2)& G_2,
\end{cases}
\end{split}
\\
\label{eq:DI3}
\frac{6}{\pi^2}
\sum_{
\genfrac{}{}{0pt}{1}
{
(a,m,u)\in \mathcal{I}_{\ell}
}
{
0\leq u < 2(h^{\vee}+\ell)
}
}
L\left(
\frac{1}{1+Y^{(a)}_m(u)}
\right)
&
=
\begin{cases}
4\ell(2r\ell -\ell-1)& C_r\\
8\ell(3\ell+1)& F_4\\
12\ell(2\ell +1)& G_2.
\end{cases}
\end{align}
\end{theorem}
The two identities \eqref{eq:DI2}
and \eqref{eq:DI3} are equivalent to each other,
since the sum of the right hand sides
is equal to $2t(h^{\vee}+\ell)((\sum_{a\in I}t_a)\ell-r)$,
which is the total number of $(a,m,u)\in I_{\ell}$
in the region $0\leq u < 2(h^{\vee}+\ell)$.

It is clear that Theorem \ref{thm:DI}
follows form Theorem \ref{thm:DI2}.

\section{Type $C_r$}

The $C_r$ case is quite parallel to the $B_r$ case.
For the reader's convenience, we repeat most of 
the basic definitions and results in \cite{IIKKN}.
Most propositions are proved in a parallel manner
to the $B_r$ case, so that proofs are omitted.
The properties of the tropical Y-system at level 2
(Proposition \ref{prop:Clev2}) are crucial and specific to $C_r$.
Since its derivation is a little more complicated
than the  $B_r$ case, the outline of the proof is provided.

\begin{figure}
\begin{picture}(270,120)(-80,-40)
\put(0,-15)
{
\put(-30,0){\circle*{5}}
\put(0,0){\circle*{5}}
\put(30,0){\circle*{5}}
\put(60,0){\circle*{5}}
\put(-27,0){\vector(1,0){24}}
\put(27,0){\vector(-1,0){24}}
\put(33,0){\vector(1,0){24}}
}
\put(0,0)
{
\put(-30,0){\circle*{5}}
\put(0,0){\circle*{5}}
\put(30,0){\circle*{5}}
\put(60,0){\circle*{5}}
\put(120,0){\circle{5}}
\put(150,0){\circle*{5}}
\put(150,15){\circle*{5}}
\put(180,0){\circle{5}}
\put(-3,0){\vector(-1,0){24}}
\put(3,0){\vector(1,0){24}}
\put(57,0){\vector(-1,0){24}}
\put(147,0){\vector(-1,0){24}}
\put(153,0){\vector(1,0){24}}
}
\put(0,15)
{
\put(-30,0){\circle*{5}}
\put(0,0){\circle*{5}}
\put(30,0){\circle*{5}}
\put(60,0){\circle*{5}}
\put(-27,0){\vector(1,0){24}}
\put(27,0){\vector(-1,0){24}}
\put(33,0){\vector(1,0){24}}
}
\put(0,30)
{
\put(-30,0){\circle*{5}}
\put(0,0){\circle*{5}}
\put(30,0){\circle*{5}}
\put(60,0){\circle*{5}}
\put(120,0){\circle{5}}
\put(150,0){\circle*{5}}
\put(150,15){\circle*{5}}
\put(180,0){\circle{5}}
\put(-3,0){\vector(-1,0){24}}
\put(3,0){\vector(1,0){24}}
\put(57,0){\vector(-1,0){24}}
\put(147,0){\vector(-1,0){24}}
\put(153,0){\vector(1,0){24}}
}
\put(0,45)
{
\put(-30,0){\circle*{5}}
\put(0,0){\circle*{5}}
\put(30,0){\circle*{5}}
\put(60,0){\circle*{5}}
\put(-27,0){\vector(1,0){24}}
\put(27,0){\vector(-1,0){24}}
\put(33,0){\vector(1,0){24}}
}
\put(0,60)
{
\put(-30,0){\circle*{5}}
\put(0,0){\circle*{5}}
\put(30,0){\circle*{5}}
\put(60,0){\circle*{5}}
\put(120,0){\circle{5}}
\put(150,0){\circle*{5}}
\put(180,0){\circle{5}}
\put(150,15){\circle*{5}}
\put(-3,0){\vector(-1,0){24}}
\put(3,0){\vector(1,0){24}}
\put(57,0){\vector(-1,0){24}}
\put(147,0){\vector(-1,0){24}}
\put(153,0){\vector(1,0){24}}
}
\put(0,75)
{
\put(-30,0){\circle*{5}}
\put(0,0){\circle*{5}}
\put(30,0){\circle*{5}}
\put(60,0){\circle*{5}}
\put(-27,0){\vector(1,0){24}}
\put(27,0){\vector(-1,0){24}}
\put(33,0){\vector(1,0){24}}
}
\put(150,-15){\circle*{5}}
%
\put(-30,0)
{
\put(0,-3){\vector(0,-1){9}}
\put(0,3){\vector(0,1){9}}
\put(0,27){\vector(0,-1){9}}
\put(0,33){\vector(0,1){9}}
\put(0,57){\vector(0,-1){9}}
\put(0,63){\vector(0,1){9}}
}
\put(0,0)
{
\put(0,-12){\vector(0,1){9}}
\put(0,12){\vector(0,-1){9}}
\put(0,18){\vector(0,1){9}}
\put(0,42){\vector(0,-1){9}}
\put(0,48){\vector(0,1){9}}
\put(0,72){\vector(0,-1){9}}
}
\put(30,0)
{
\put(0,-3){\vector(0,-1){9}}
\put(0,3){\vector(0,1){9}}
\put(0,27){\vector(0,-1){9}}
\put(0,33){\vector(0,1){9}}
\put(0,57){\vector(0,-1){9}}
\put(0,63){\vector(0,1){9}}
}
\put(60,0)
{
\put(0,-12){\vector(0,1){9}}
\put(0,12){\vector(0,-1){9}}
\put(0,18){\vector(0,1){9}}
\put(0,42){\vector(0,-1){9}}
\put(0,48){\vector(0,1){9}}
\put(0,72){\vector(0,-1){9}}
}
\put(150,0)
{
\put(0,-12){\vector(0,1){9}}
\put(0,12){\vector(0,-1){9}}
\put(0,18){\vector(0,1){9}}
\put(0,42){\vector(0,-1){9}}
\put(0,48){\vector(0,1){9}}
\put(0,72){\vector(0,-1){9}}
}
\put(123,-2){\vector(2,-1){24}}
\put(123,2){\vector(2,1){24}}
\put(123,58){\vector(2,-1){24}}
\put(123,62){\vector(2,1){24}}
\put(177,28){\vector(-2,-1){24}}
\put(177,32){\vector(-2,1){24}}
\put(120,27){\vector(0,-1){24}}
\put(180,3){\vector(0,1){24}}
\put(120,33){\vector(0,1){24}}
\put(180,57){\vector(0,-1){24}}
\put(-50,-17){$\cdots$}
\put(-50,-2){$\cdots$}
\put(-50,13){$\cdots$}
\put(-50,28){$\cdots$}
\put(-50,43){$\cdots$}
\put(-50,58){$\cdots$}
\put(-50,73){$\cdots$}
\put(0,-14)
{
\put(-28,2){\small $ -$}
\put(2,2){\small $+$}
\put(32,2){\small $-$}
\put(62,2){\small $+$}
}
\put(0,1)
{
\put(-28,2){\small $ +$}
\put(2,2){\small $-$}
\put(32,2){\small $+$}
\put(62,2){\small $-$}
}
\put(0,16)
{
\put(-28,2){\small $ -$}
\put(2,2){\small $+$}
\put(32,2){\small $-$}
\put(62,2){\small $+$}
}
\put(0,31)
{
\put(-28,2){\small $ +$}
\put(2,2){\small $-$}
\put(32,2){\small $+$}
\put(62,2){\small $-$}
}
\put(0,46)
{
\put(-28,2){\small $ -$}
\put(2,2){\small $+$}
\put(32,2){\small $-$}
\put(62,2){\small $+$}
}
\put(0,61)
{
\put(-28,2){\small $ +$}
\put(2,2){\small $-$}
\put(32,2){\small $+$}
\put(62,2){\small $-$}
}
\put(0,76)
{
\put(-28,2){\small $ -$}
\put(2,2){\small $+$}
\put(32,2){\small $-$}
\put(62,2){\small $+$}
}
\put(0,1)
{
\put(122,5){\small $-$}
\put(122,32){\small $+$}
\put(122,65){\small $-$}
\put(152,-13){\small $+$}
\put(152,2){\small $-$}
\put(152,20){\small $+$}
\put(152,32){\small $-$}
\put(152,47){\small $+$}
\put(152,62){\small $-$}
\put(152,77){\small $+$}
\put(182,2){\small $+$}
\put(182,32){\small $-$}
\put(182,62){\small $+$}
}
%
%
%
\put(-60,-25){$\underbrace{\hskip120pt}_{r-1}$}
\put(85,28){${\scriptstyle\ell -1}\left\{ \makebox(0,37){}\right.$}
\put(-85,28){${\scriptstyle2\ell -1}\left\{ \makebox(0,50){}\right.$}
\end{picture}
\begin{picture}(270,160)(-80,-40)
\put(0,-15)
{
\put(-30,0){\circle*{5}}
\put(0,0){\circle*{5}}
\put(30,0){\circle*{5}}
\put(60,0){\circle*{5}}
\put(-27,0){\vector(1,0){24}}
\put(27,0){\vector(-1,0){24}}
\put(33,0){\vector(1,0){24}}
}
\put(0,0)
{
\put(-30,0){\circle*{5}}
\put(0,0){\circle*{5}}
\put(30,0){\circle*{5}}
\put(60,0){\circle*{5}}
\put(120,0){\circle{5}}
\put(150,0){\circle*{5}}
\put(150,15){\circle*{5}}
\put(180,0){\circle{5}}
\put(-3,0){\vector(-1,0){24}}
\put(3,0){\vector(1,0){24}}
\put(57,0){\vector(-1,0){24}}
\put(147,0){\vector(-1,0){24}}
\put(153,0){\vector(1,0){24}}
}
\put(0,15)
{
\put(-30,0){\circle*{5}}
\put(0,0){\circle*{5}}
\put(30,0){\circle*{5}}
\put(60,0){\circle*{5}}
\put(-27,0){\vector(1,0){24}}
\put(27,0){\vector(-1,0){24}}
\put(33,0){\vector(1,0){24}}
}
\put(0,30)
{
\put(-30,0){\circle*{5}}
\put(0,0){\circle*{5}}
\put(30,0){\circle*{5}}
\put(60,0){\circle*{5}}
\put(120,0){\circle{5}}
\put(150,0){\circle*{5}}
\put(150,15){\circle*{5}}
\put(180,0){\circle{5}}
\put(-3,0){\vector(-1,0){24}}
\put(3,0){\vector(1,0){24}}
\put(57,0){\vector(-1,0){24}}
\put(147,0){\vector(-1,0){24}}
\put(153,0){\vector(1,0){24}}
}
\put(0,45)
{
\put(-30,0){\circle*{5}}
\put(0,0){\circle*{5}}
\put(30,0){\circle*{5}}
\put(60,0){\circle*{5}}
\put(-27,0){\vector(1,0){24}}
\put(27,0){\vector(-1,0){24}}
\put(33,0){\vector(1,0){24}}
}
\put(0,60)
{
\put(-30,0){\circle*{5}}
\put(0,0){\circle*{5}}
\put(30,0){\circle*{5}}
\put(60,0){\circle*{5}}
\put(120,0){\circle{5}}
\put(150,0){\circle*{5}}
\put(180,0){\circle{5}}
\put(150,15){\circle*{5}}
\put(-3,0){\vector(-1,0){24}}
\put(3,0){\vector(1,0){24}}
\put(57,0){\vector(-1,0){24}}
\put(147,0){\vector(-1,0){24}}
\put(153,0){\vector(1,0){24}}
}
\put(0,75)
{
\put(-30,0){\circle*{5}}
\put(0,0){\circle*{5}}
\put(30,0){\circle*{5}}
\put(60,0){\circle*{5}}
\put(-27,0){\vector(1,0){24}}
\put(27,0){\vector(-1,0){24}}
\put(33,0){\vector(1,0){24}}
}
\put(0,90)
{
\put(-30,0){\circle*{5}}
\put(0,0){\circle*{5}}
\put(30,0){\circle*{5}}
\put(60,0){\circle*{5}}
\put(120,0){\circle{5}}
\put(150,0){\circle*{5}}
\put(180,0){\circle{5}}
\put(150,15){\circle*{5}}
\put(-3,0){\vector(-1,0){24}}
\put(3,0){\vector(1,0){24}}
\put(57,0){\vector(-1,0){24}}
\put(147,0){\vector(-1,0){24}}
\put(153,0){\vector(1,0){24}}
}
\put(0,105)
{
\put(-30,0){\circle*{5}}
\put(0,0){\circle*{5}}
\put(30,0){\circle*{5}}
\put(60,0){\circle*{5}}
\put(-27,0){\vector(1,0){24}}
\put(27,0){\vector(-1,0){24}}
\put(33,0){\vector(1,0){24}}
}
\put(150,-15){\circle*{5}}
%
\put(-30,0)
{
\put(0,-3){\vector(0,-1){9}}
\put(0,3){\vector(0,1){9}}
\put(0,27){\vector(0,-1){9}}
\put(0,33){\vector(0,1){9}}
\put(0,57){\vector(0,-1){9}}
\put(0,63){\vector(0,1){9}}
\put(0,87){\vector(0,-1){9}}
\put(0,93){\vector(0,1){9}}
}
\put(0,0)
{
\put(0,-12){\vector(0,1){9}}
\put(0,12){\vector(0,-1){9}}
\put(0,18){\vector(0,1){9}}
\put(0,42){\vector(0,-1){9}}
\put(0,48){\vector(0,1){9}}
\put(0,72){\vector(0,-1){9}}
\put(0,78){\vector(0,1){9}}
\put(0,102){\vector(0,-1){9}}
}
\put(30,0)
{
\put(0,-3){\vector(0,-1){9}}
\put(0,3){\vector(0,1){9}}
\put(0,27){\vector(0,-1){9}}
\put(0,33){\vector(0,1){9}}
\put(0,57){\vector(0,-1){9}}
\put(0,63){\vector(0,1){9}}
\put(0,87){\vector(0,-1){9}}
\put(0,93){\vector(0,1){9}}
}
\put(60,0)
{
\put(0,-12){\vector(0,1){9}}
\put(0,12){\vector(0,-1){9}}
\put(0,18){\vector(0,1){9}}
\put(0,42){\vector(0,-1){9}}
\put(0,48){\vector(0,1){9}}
\put(0,72){\vector(0,-1){9}}
\put(0,78){\vector(0,1){9}}
\put(0,102){\vector(0,-1){9}}
}
\put(150,0)
{
\put(0,-12){\vector(0,1){9}}
\put(0,12){\vector(0,-1){9}}
\put(0,18){\vector(0,1){9}}
\put(0,42){\vector(0,-1){9}}
\put(0,48){\vector(0,1){9}}
\put(0,72){\vector(0,-1){9}}
\put(0,78){\vector(0,1){9}}
\put(0,102){\vector(0,-1){9}}
}
\put(123,-2){\vector(2,-1){24}}
\put(123,2){\vector(2,1){24}}
\put(123,58){\vector(2,-1){24}}
\put(123,62){\vector(2,1){24}}
\put(177,28){\vector(-2,-1){24}}
\put(177,32){\vector(-2,1){24}}
\put(177,88){\vector(-2,-1){24}}
\put(177,92){\vector(-2,1){24}}
\put(120,27){\vector(0,-1){24}}
\put(180,3){\vector(0,1){24}}
\put(120,33){\vector(0,1){24}}
\put(180,57){\vector(0,-1){24}}
\put(120,87){\vector(0,-1){24}}
\put(180,63){\vector(0,1){24}}
\put(-50,-17){$\cdots$}
\put(-50,-2){$\cdots$}
\put(-50,13){$\cdots$}
\put(-50,28){$\cdots$}
\put(-50,43){$\cdots$}
\put(-50,58){$\cdots$}
\put(-50,73){$\cdots$}
\put(-50,88){$\cdots$}
\put(-50,103){$\cdots$}
\put(0,-14)
{
\put(-28,2){\small $ -$}
\put(2,2){\small $+$}
\put(32,2){\small $-$}
\put(62,2){\small $+$}
}
\put(0,1)
{
\put(-28,2){\small $ +$}
\put(2,2){\small $-$}
\put(32,2){\small $+$}
\put(62,2){\small $-$}
}
\put(0,16)
{
\put(-28,2){\small $ -$}
\put(2,2){\small $+$}
\put(32,2){\small $-$}
\put(62,2){\small $+$}
}
\put(0,31)
{
\put(-28,2){\small $ +$}
\put(2,2){\small $-$}
\put(32,2){\small $+$}
\put(62,2){\small $-$}
}
\put(0,46)
{
\put(-28,2){\small $ -$}
\put(2,2){\small $+$}
\put(32,2){\small $-$}
\put(62,2){\small $+$}
}
\put(0,61)
{
\put(-28,2){\small $ +$}
\put(2,2){\small $-$}
\put(32,2){\small $+$}
\put(62,2){\small $-$}
}
\put(0,76)
{
\put(-28,2){\small $ -$}
\put(2,2){\small $+$}
\put(32,2){\small $-$}
\put(62,2){\small $+$}
}
\put(0,91)
{
\put(-28,2){\small $ +$}
\put(2,2){\small $-$}
\put(32,2){\small $+$}
\put(62,2){\small $-$}
}
\put(0,106)
{
\put(-28,2){\small $ -$}
\put(2,2){\small $+$}
\put(32,2){\small $-$}
\put(62,2){\small $+$}
}
\put(0,1)
{
\put(122,5){\small $-$}
\put(122,32){\small $+$}
\put(122,65){\small $-$}
\put(122,92){\small $+$}
\put(152,-13){\small $+$}
\put(152,2){\small $-$}
\put(152,20){\small $+$}
\put(152,32){\small $-$}
\put(152,47){\small $+$}
\put(152,62){\small $-$}
\put(152,80){\small $+$}
\put(152,92){\small $-$}
\put(152,107){\small $+$}
\put(182,2){\small $+$}
\put(182,32){\small $-$}
\put(182,62){\small $+$}
\put(182,92){\small $-$}
}
%
%
%
\put(-60,-25){$\underbrace{\hskip120pt}_{r-1}$}
\put(85,43){${\scriptstyle\ell -1}\left\{ \makebox(0,50){}\right.$}
\put(-85,43){${\scriptstyle2\ell -1}\left\{ \makebox(0,65){}\right.$}
\end{picture}
\caption{The quiver $Q_{\ell}(C_r)$  for even $\ell$
(upper) and for odd $\ell$ (lower),
where we identify the rightmost column in the left quiver
with the middle column in the right quiver.}
\label{fig:quiverC}
\end{figure}
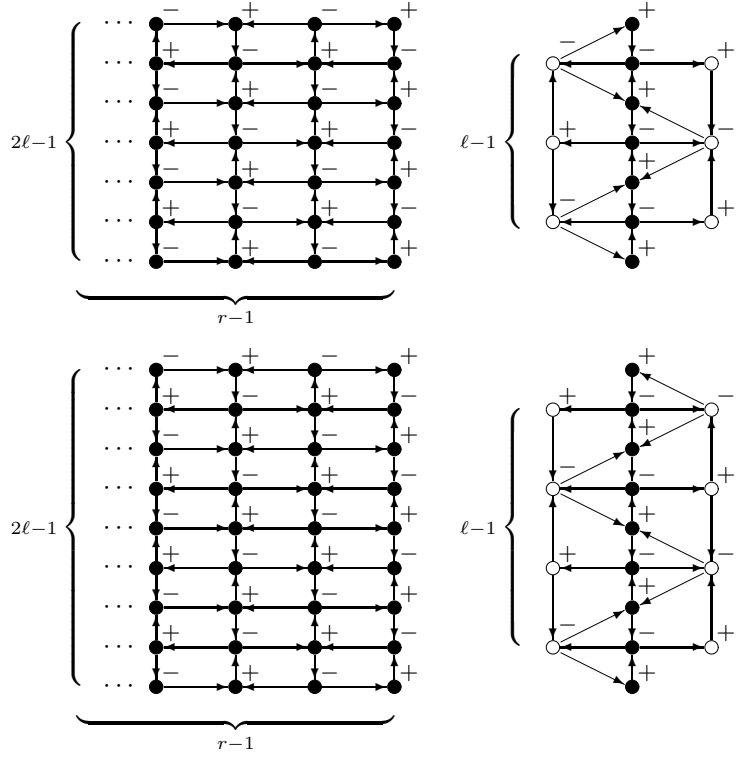

\subsection{Parity decompositions of T and Y-systems}

For a triplet $(a,m,u)\in \mathcal{I}_{\ell}$,
we set the `parity conditions' $\mathbf{P}_{+}$ and
$\mathbf{P}_{-}$ by
\begin{align}
\label{eq:CPcond}
\mathbf{P}_{+}:& \ \mbox{
 $r+a+m+2u$ is odd if $a\neq r$;
$2u$ is even if
$a=r$,}\\
\mathbf{P}_{-}:& \ \mbox{
 $r+a+m+2u$ is even if $a\neq r$;
$2u$ is odd if
$a=r $.}
\end{align}
We write, for example, $(a,m,u):\mathbf{P}_{+}$ if $(a,m,u)$ satisfies
$\mathbf{P}_{+}$.
We have $\mathcal{I}_{\ell}=
\mathcal{I}_{\ell+}\sqcup \mathcal{I}_{\ell-}$,
where $\mathcal{I}_{\ell\varepsilon}$ 
is the set of all $(a,m,u):\mathbf{P}_{\varepsilon}$.

Define $\EuScript{T}^{\circ}_{\ell}(C_r)_{\varepsilon}$
($\varepsilon=\pm$)
to be the subring of $\EuScript{T}^{\circ}_{\ell}(C_r)$
generated by
 $T^{(a)}_m(u)$
$((a,m,u)\in \mathcal{I}_{\ell\varepsilon})$.
Then, we have
$\EuScript{T}^{\circ}_{\ell}(C_r)_+
\simeq
\EuScript{T}^{\circ}_{\ell}(C_r)_-
$
by $T^{(a)}_m(u)\mapsto T^{(a)}_m(u+\frac{1}{2})$ and
\begin{align}
\EuScript{T}^{\circ}_{\ell}(C_r)
\simeq
\EuScript{T}^{\circ}_{\ell}(C_r)_+
\otimes_{\mathbb{Z}}
\EuScript{T}^{\circ}_{\ell}(C_r)_-.
\end{align}

For a triplet $(a,m,u)\in \mathcal{I}_{\ell}$ ,
we set another `parity conditions' $\mathbf{P}'_{+}$ and
$\mathbf{P}'_{-}$ by
\begin{align}
\label{eq:CPcond2}
\mathbf{P}'_{+}:& \ \mbox{
 $r+a+m+2u$ is even if $a\neq r$;
$2u$ is even if
$a=r$,}\\
\mathbf{P}'_{-}:& \ \mbox{
 $r+a+m+2u$ is odd if $a\neq r$;
$2u$ is odd if
$a=r $.}
\end{align}
We have $\mathcal{I}_{\ell}=
\mathcal{I}'_{\ell+}\sqcup \mathcal{I}'_{\ell-}$,
where $\mathcal{I}'_{\ell\varepsilon}$ 
is the set of all $(a,m,u):\mathbf{P}'_{\varepsilon}$.
We also have
\begin{align}
(a,m,u):\mathbf{P}'_+
\ \Longleftrightarrow \
\textstyle (a,m,u\pm \frac{1}{t_a}):\mathbf{P}_+.
\end{align}

Define $\EuScript{Y}^{\circ}_{\ell}(C_r)_{\varepsilon}$
($\varepsilon=\pm$)
to be the subgroup of $\EuScript{Y}^{\circ}_{\ell}(C_r)$
generated by
$Y^{(a)}_m(u)$, $1+Y^{(a)}_m(u)$
$((a,m,u)\in \mathcal{I}'_{\ell\varepsilon})$.
Then, we have
$\EuScript{Y}^{\circ}_{\ell}(C_r)_+
\simeq
\EuScript{Y}^{\circ}_{\ell}(C_r)_-
$
by $Y^{(a)}_m(u)\mapsto Y^{(a)}_m(u+\frac{1}{2})$,
$1+Y^{(a)}_m(u)\mapsto 1+Y^{(a)}_m(u+\frac{1}{2})$,
 and
\begin{align}
\EuScript{Y}^{\circ}_{\ell}(C_r)
\simeq
\EuScript{Y}^{\circ}_{\ell}(C_r)_+
\times
\EuScript{Y}^{\circ}_{\ell}(C_r)_-.
\end{align}

\subsection{Quiver $Q_{\ell}(C_r)$}

With type $C_r$ and $\ell\geq 2$ we associate
 the quiver $Q_{\ell}(C_r)$
by Figure
\ref{fig:quiverC},
where the rightmost column in the left quiver
and the middle column in the right quiver are
identified.
Also, we assign the empty or filled circle $\circ$/$\bullet$ and
the sign +/$-$ to each vertex.

Let us choose  the index set $\mathbf{I}$
of the vertices of $Q_{\ell}(C_r)$
so that $\mathbf{i}=(i,i')\in \mathbf{I}$ represents
the vertex 
at the $i'$th row (from the bottom)
of the $i$th column (from the left)
in the left quiver for $i=1,\dots,r-1$,
the one
of the right column 
in the right quiver
for $i=r$,
and 
the one
of the left column 
in the right quiver 
for $i=r+1$.
Thus, $i=1,\dots,r+1$, and $i'=1,\dots,\ell-1$ if $i\neq r,r+1$
and $i'=1,\dots,2\ell-1$ if $i=r,r+1$.
We use a natural notation $\mathbf{I}^{\circ}$
(resp.\ $\mathbf{I}^{\circ}_+$)
for the set of the vertices $\mathbf{i}$ with
property $\circ$ (resp.\ $\circ$ and +), and so on.
We have $\mathbf{I}=\mathbf{I}^{\circ}\sqcup \mathbf{I}^{\bullet}
=\mathbf{I}^{\circ}_+\sqcup
\mathbf{I}^{\circ}_-\sqcup
\mathbf{I}^{\bullet}_+\sqcup
\mathbf{I}^{\bullet}_-$.

We define composite mutations,
\begin{align}
\label{eq:Cmupm2}
\mu^{\circ}_+=\prod_{\mathbf{i}\in\mathbf{I}^{\circ}_+}
\mu_{\mathbf{i}},
\quad
\mu^{\circ}_-=\prod_{\mathbf{i}\in\mathbf{I}^{\circ}_-}
\mu_{\mathbf{i}},
\quad
\mu^{\bullet}_+=\prod_{\mathbf{i}\in\mathbf{I}^{\bullet}_+}
\mu_{\mathbf{i}},
\quad
\mu^{\bullet}_-=\prod_{\mathbf{i}\in\mathbf{I}^{\bullet}_-}
\mu_{\mathbf{i}}.
\end{align}
Note that they do not depend on the order of the product.

Let $\boldsymbol{r}$ be the involution acting on $\mathbf{I}$
by the left-right reflection of the right quiver.
Let $\boldsymbol{\omega}$ be the involution acting on $\mathbf{I}$
defined by,
for even $r$,
 the up-down reflection of the left quiver
and the $180^{\circ}$ rotation of the right quiver;
and for odd $r$,
 the up-down reflection of the left and right quivers.
Let $\boldsymbol{r}(Q_{\ell}(C_r))$
and $\boldsymbol{\omega}(Q_{\ell}(C_r))$ denote the quivers
 induced from $Q_{\ell}(C_r)$
 by
$\boldsymbol{r}$ and $\boldsymbol{\omega}$,
respectively.
For example, 
if there is an arrow
 $\mathbf{i}\rightarrow
\mathbf{j}$ in $Q_{\ell}(C_r)$,
then, there is an arrow
$\boldsymbol{r}(\mathbf{i})
\rightarrow
\boldsymbol{r}(\mathbf{j})
$
in  $\boldsymbol{r}(Q_{\ell}(C_r))$.
For a quiver $Q$, $Q^{\mathrm{op}}$ denotes the opposite quiver.

\begin{lemma}
\label{lem:CQmut}
Let $Q=Q_{\ell}(C_r)$.
\par
(i)
We have a periodic sequence of mutations of quivers
\begin{align}
\label{eq:CB2}
Q\ 
\mathop{\longleftrightarrow}^{\mu^{\bullet}_+
\mu^{\circ}_+}
\
Q^{\mathrm{op}}
\
\mathop{\longleftrightarrow}^{\mu^{\bullet}_-}
\
\boldsymbol{r}(Q)
\
\mathop{\longleftrightarrow}^{\mu^{\bullet}_+
\mu^{\circ}_-}
\
\boldsymbol{r}(Q){}^{\mathrm{op}}
\
\mathop{\longleftrightarrow}^{\mu^{\bullet}_-}
\
Q.
\end{align}
\par
(ii)
 $\boldsymbol{\omega}(Q)=Q$ if $h^{\vee}+\ell$ is even,
and  $\boldsymbol{\omega}(Q)=\boldsymbol{r}(Q)$ if $h^{\vee}+\ell$ is odd.
\end{lemma}

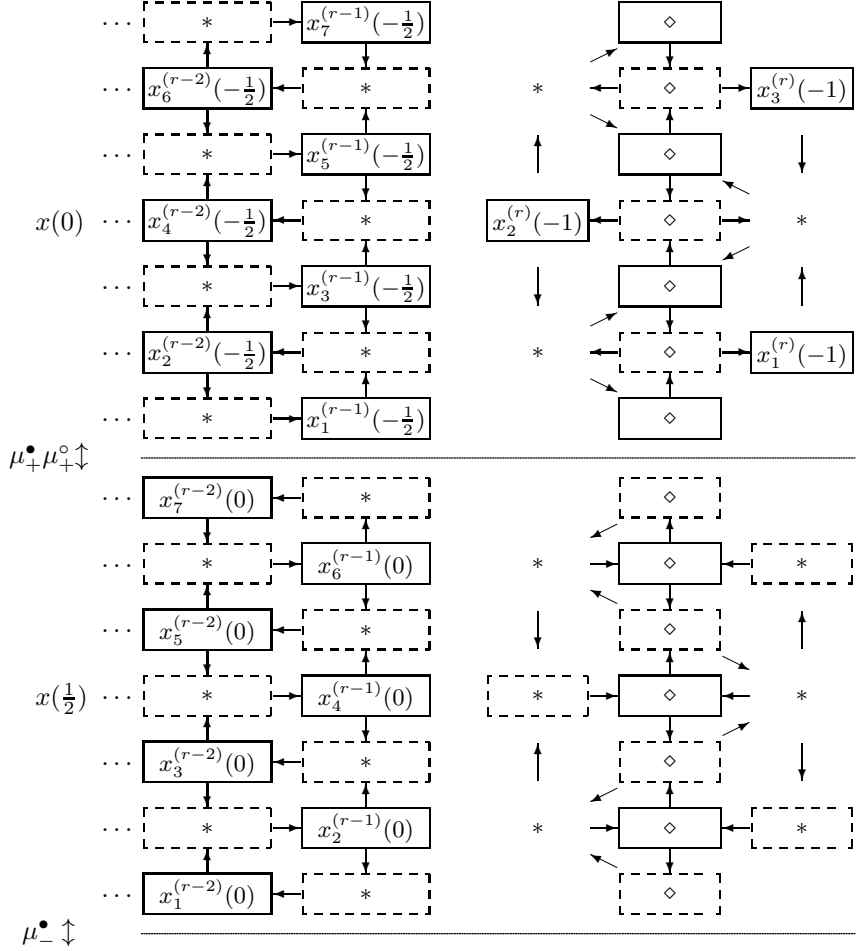
\begin{figure}
\begin{picture}(280,370)(-50,-230)
\put(80,0)
{
\put(-140,46.5){$x(0)$}
\put(-19,-7.5)
{
\put(50,0){\makebox(38,15){\small $*$}}
\put(150,0){\framebox(38,15){\small $x^{(r)}_1(-1)$}}
\put(50,50){\framebox(38,15){\small $x^{(r)}_2(-1)$}}
\put(150,50){\makebox(38,15){\small $*$}}
\put(50,100){\makebox(38,15){\small $*$}}
\put(150,100){\framebox(38,15){\small $x^{(r)}_3(-1)$}}
\put(100,-25){\framebox(38,15){\small $\diamond$}}
\put(100,0){\dashbox{3}(38,15){\small $\diamond$}}
\put(100,25){\framebox(38,15){\small $\diamond$}}
\put(100,50){\dashbox{3}(38,15){\small $\diamond$}}
\put(100,75){\framebox(38,15){\small $\diamond$}}
\put(100,100){\dashbox{3}(38,15){\small $\diamond$}}
\put(100,125){\framebox(38,15){\small $\diamond$}}
}
%
\put(80,0){$\vector(-1,0){10}$}
\put(120,0){$\vector(1,0){10}$}
\put(80,50){$\vector(-1,0){10}$}
\put(120,50){$\vector(1,0){10}$}
\put(80,100){$\vector(-1,0){10}$}
\put(120,100){$\vector(1,0){10}$}
\put(50,32){\vector(0,-1){14}}
\put(50,68){\vector(0,1){14}}
\put(100,-17){\vector(0,1){9}}
\put(100,17){\vector(0,-1){9}}
\put(100,33){\vector(0,1){9}}
\put(100,67){\vector(0,-1){9}}
\put(100,83){\vector(0,1){9}}
\put(100,117){\vector(0,-1){9}}
\put(150,18){\vector(0,1){14}}
\put(150,82){\vector(0,-1){14}}

\put(70,-10){\vector(2,-1){10}}
\put(70,10){\vector(2,1){10}}
\put(70,90){\vector(2,-1){10}}
\put(70,110){\vector(2,1){10}}
\put(130,40){\vector(-2,-1){10}}
\put(130,60){\vector(-2,1){10}}
%
\put(30,0)
{
\put(-169,-7.5)
{
\put(40,-25){\dashbox{3}(48,15){\small $*$}}
\put(40,0){\framebox(48,15){\small $x^{(r-2)}_2(-\frac{1}{2})$}}
\put(40,25){\dashbox{3}(48,15){\small $*$}}
\put(40,50){\framebox(48,15){\small $x^{(r-2)}_4(-\frac{1}{2})$}}
\put(40,75){\dashbox{3}(48,15){\small $*$}}
\put(40,100){\framebox(48,15){\small $x^{(r-2)}_6(-\frac{1}{2})$}}
\put(40,125){\dashbox{3}(48,15){\small $*$}}
\put(100,-25){\framebox(48,15){\small $x^{(r-1)}_1(-\frac{1}{2})$}}
\put(100,0){\dashbox{3}(48,15){\small $*$}}
\put(100,25){\framebox(48,15){\small $x^{(r-1)}_3(-\frac{1}{2})$}}
\put(100,50){\dashbox{3}(48,15){\small $*$}}
\put(100,75){\framebox(48,15){\small $x^{(r-1)}_5(-\frac{1}{2})$}}
\put(100,100){\dashbox{3}(48,15){\small $*$}}
\put(100,125){\framebox(48,15){\small $x^{(r-1)}_7(-\frac{1}{2})$}}
}
\put(-45,-17){\vector(0,1){9}}
\put(-45,17){\vector(0,-1){9}}
\put(-45,33){\vector(0,1){9}}
\put(-45,67){\vector(0,-1){9}}
\put(-45,83){\vector(0,1){9}}
\put(-45,117){\vector(0,-1){9}}
\put(-105,-8){\vector(0,-1){9}}
\put(-105,8){\vector(0,1){9}}
\put(-105,42){\vector(0,-1){9}}
\put(-105,58){\vector(0,1){9}}
\put(-105,92){\vector(0,-1){9}}
\put(-105,108){\vector(0,1){9}}
\put(-80,-25){$\vector(1,0){10}$}
\put(-70,0){$\vector(-1,0){10}$}
\put(-80,25){$\vector(1,0){10}$}
\put(-70,50){$\vector(-1,0){10}$}
\put(-80,75){$\vector(1,0){10}$}
\put(-70,100){$\vector(-1,0){10}$}
\put(-80,125){$\vector(1,0){10}$}
\put(-145,-26){$\dots$}
\put(-145,-1){$\dots$}
\put(-145,24){$\dots$}
\put(-145,49){$\dots$}
\put(-145,74){$\dots$}
\put(-145,99){$\dots$}
\put(-145,124){$\dots$}
} 
\dottedline(-100,-40)(170,-40)
\put(-125,-42){$\updownarrow$}
\put(-150,-42){$\mu^{\bullet}_+\mu^{\circ}_+$}
}
\put(80,-180)
{
\put(-140,46.5){$x(\frac{1}{2})$}
\put(-19,-7.5)
{
\put(50,0){\makebox(38,15){\small $*$}}
\put(150,0){\dashbox{3}(38,15){\small $*$}}
\put(50,50){\dashbox{3}(38,15){\small $*$}}
\put(150,50){\makebox(38,15){\small $*$}}
\put(50,100){\makebox(38,15){\small $*$}}
\put(150,100){\dashbox{3}(38,15){\small $*$}}
\put(100,-25){\dashbox{3}(38,15){\small $\diamond$}}
\put(100,0){\framebox(38,15){\small $\diamond$}}
\put(100,25){\dashbox{3}(38,15){\small $\diamond$}}
\put(100,50){\framebox(38,15){\small $\diamond$}}
\put(100,75){\dashbox{3}(38,15){\small $\diamond$}}
\put(100,100){\framebox(38,15){\small $\diamond$}}
\put(100,125){\dashbox{3}(38,15){\small $\diamond$}}
}
%
\put(70,0){$\vector(1,0){10}$}
\put(130,0){$\vector(-1,0){10}$}
\put(70,50){$\vector(1,0){10}$}
\put(130,50){$\vector(-1,0){10}$}
\put(70,100){$\vector(1,0){10}$}
\put(130,100){$\vector(-1,0){10}$}
\put(50,18){\vector(0,1){14}}
\put(50,82){\vector(0,-1){14}}
\put(100,-8){\vector(0,-1){9}}
\put(100,8){\vector(0,1){9}}
\put(100,42){\vector(0,-1){9}}
\put(100,58){\vector(0,1){9}}
\put(100,92){\vector(0,-1){9}}
\put(100,108){\vector(0,1){9}}
\put(150,32){\vector(0,-1){14}}
\put(150,68){\vector(0,1){14}}

\put(80,-15){\vector(-2,1){10}}
\put(80,15){\vector(-2,-1){10}}
\put(80,85){\vector(-2,1){10}}
\put(80,115){\vector(-2,-1){10}}
\put(120,35){\vector(2,1){10}}
\put(120,65){\vector(2,-1){10}}
%
\put(30,0)
{
\put(-169,-7.5)
{
\put(40,-25){\framebox(48,15){\small $x^{(r-2)}_1(0)$}}
\put(40,0){\dashbox{3}(48,15){\small $*$}}
\put(40,25){\framebox(48,15){\small $x^{(r-2)}_3(0)$}}
\put(40,50){\dashbox{3}(48,15){\small $*$}}
\put(40,75){\framebox(48,15){\small $x^{(r-2)}_5(0)$}}
\put(40,100){\dashbox{3}(48,15){\small $*$}}
\put(40,125){\framebox(48,15){\small $x^{(r-2)}_7(0)$}}
\put(100,-25){\dashbox{3}(48,15){\small $*$}}
\put(100,0){\framebox(48,15){\small $x^{(r-1)}_2(0)$}}
\put(100,25){\dashbox{3}(48,15){\small $*$}}
\put(100,50){\framebox(48,15){\small $x^{(r-1)}_4(0)$}}
\put(100,75){\dashbox{3}(48,15){\small $*$}}
\put(100,100){\framebox(48,15){\small $x^{(r-1)}_6(0)$}}
\put(100,125){\dashbox{3}(48,15){\small $*$}}
}
\put(-45,-8){\vector(0,-1){9}}
\put(-45,8){\vector(0,1){9}}
\put(-45,42){\vector(0,-1){9}}
\put(-45,58){\vector(0,1){9}}
\put(-45,92){\vector(0,-1){9}}
\put(-45,108){\vector(0,1){9}}
\put(-105,-17){\vector(0,1){9}}
\put(-105,17){\vector(0,-1){9}}
\put(-105,33){\vector(0,1){9}}
\put(-105,67){\vector(0,-1){9}}
\put(-105,83){\vector(0,1){9}}
\put(-105,117){\vector(0,-1){9}}
\put(-70,-25){$\vector(-1,0){10}$}
\put(-80,0){$\vector(1,0){10}$}
\put(-70,25){$\vector(-1,0){10}$}
\put(-80,50){$\vector(1,0){10}$}
\put(-70,75){$\vector(-1,0){10}$}
\put(-80,100){$\vector(1,0){10}$}
\put(-70,125){$\vector(-1,0){10}$}
\put(-145,-26){$\dots$}
\put(-145,-1){$\dots$}
\put(-145,24){$\dots$}
\put(-145,49){$\dots$}
\put(-145,74){$\dots$}
\put(-145,99){$\dots$}
\put(-145,124){$\dots$}
} 
\dottedline(-100,-40)(170,-40)
\put(-130,-42){$\updownarrow$}
\put(-145,-42){$\mu^{\bullet}_-$}
}
\end{picture}
\caption{(Continues to Figure \ref{fig:labelxC2})
Label of cluster variables $x_{\mathbf{i}}(u)$
by $\mathcal{I}_{\ell+}$  for
$C_r$, $\ell=4$.
The variables framed by solid/dashed lines
satisfy the condition $\mathbf{p}_+$/$\mathbf{p}_-$,
respectively.
The middle column in the right quiver (marked by $\diamond$)
is identified with the rightmost column in the left quiver.
}
\label{fig:labelxC1}
\end{figure}

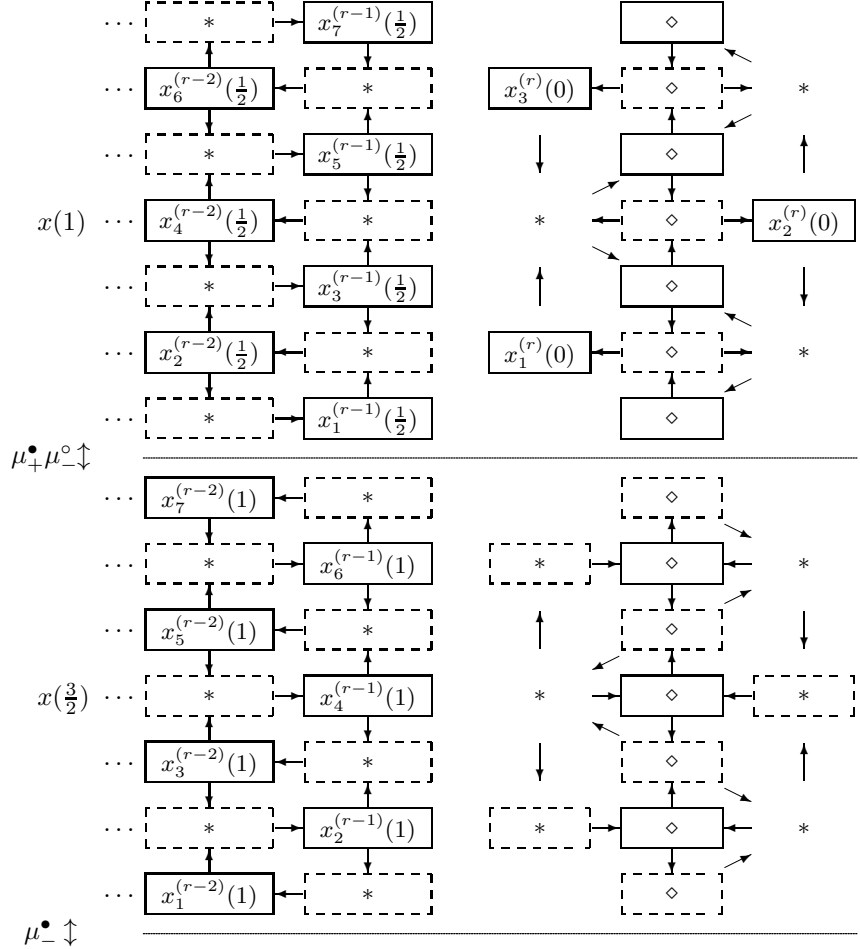
\begin{figure}
\begin{picture}(280,370)(-50,-230)
\put(80,0)
{
\put(-140,46.5){$x(1)$}
\put(-19,-7.5)
{
\put(50,0){\framebox(38,15){\small $x^{(r)}_1(0)$}}
\put(150,0){\makebox(38,15){\small $*$}}
\put(50,50){\makebox(38,15){\small $*$}}
\put(150,50){\framebox(38,15){\small $x^{(r)}_2(0)$}}
\put(50,100){\framebox(38,15){\small $x^{(r)}_3(0)$}}
\put(150,100){\makebox(38,15){\small $*$}}
\put(100,-25){\framebox(38,15){\small $\diamond$}}
\put(100,0){\dashbox{3}(38,15){\small $\diamond$}}
\put(100,25){\framebox(38,15){\small $\diamond$}}
\put(100,50){\dashbox{3}(38,15){\small $\diamond$}}
\put(100,75){\framebox(38,15){\small $\diamond$}}
\put(100,100){\dashbox{3}(38,15){\small $\diamond$}}
\put(100,125){\framebox(38,15){\small $\diamond$}}
}
%
\put(80,0){$\vector(-1,0){10}$}
\put(120,0){$\vector(1,0){10}$}
\put(80,50){$\vector(-1,0){10}$}
\put(120,50){$\vector(1,0){10}$}
\put(80,100){$\vector(-1,0){10}$}
\put(120,100){$\vector(1,0){10}$}
\put(50,18){\vector(0,1){14}}
\put(50,82){\vector(0,-1){14}}
\put(100,-17){\vector(0,1){9}}
\put(100,17){\vector(0,-1){9}}
\put(100,33){\vector(0,1){9}}
\put(100,67){\vector(0,-1){9}}
\put(100,83){\vector(0,1){9}}
\put(100,117){\vector(0,-1){9}}
\put(150,32){\vector(0,-1){14}}
\put(150,68){\vector(0,1){14}}

\put(130,-10){\vector(-2,-1){10}}
\put(130,10){\vector(-2,1){10}}
\put(130,90){\vector(-2,-1){10}}
\put(130,110){\vector(-2,1){10}}
\put(70,40){\vector(2,-1){10}}
\put(70,60){\vector(2,1){10}}
%
\put(30,0)
{
\put(-169,-7.5)
{
\put(40,-25){\dashbox{3}(48,15){\small $*$}}
\put(40,0){\framebox(48,15){\small $x^{(r-2)}_2(\frac{1}{2})$}}
\put(40,25){\dashbox{3}(48,15){\small $*$}}
\put(40,50){\framebox(48,15){\small $x^{(r-2)}_4(\frac{1}{2})$}}
\put(40,75){\dashbox{3}(48,15){\small $*$}}
\put(40,100){\framebox(48,15){\small $x^{(r-2)}_6(\frac{1}{2})$}}
\put(40,125){\dashbox{3}(48,15){\small $*$}}
\put(100,-25){\framebox(48,15){\small $x^{(r-1)}_1(\frac{1}{2})$}}
\put(100,0){\dashbox{3}(48,15){\small $*$}}
\put(100,25){\framebox(48,15){\small $x^{(r-1)}_3(\frac{1}{2})$}}
\put(100,50){\dashbox{3}(48,15){\small $*$}}
\put(100,75){\framebox(48,15){\small $x^{(r-1)}_5(\frac{1}{2})$}}
\put(100,100){\dashbox{3}(48,15){\small $*$}}
\put(100,125){\framebox(48,15){\small $x^{(r-1)}_7(\frac{1}{2})$}}
}
\put(-45,-17){\vector(0,1){9}}
\put(-45,17){\vector(0,-1){9}}
\put(-45,33){\vector(0,1){9}}
\put(-45,67){\vector(0,-1){9}}
\put(-45,83){\vector(0,1){9}}
\put(-45,117){\vector(0,-1){9}}
\put(-105,-8){\vector(0,-1){9}}
\put(-105,8){\vector(0,1){9}}
\put(-105,42){\vector(0,-1){9}}
\put(-105,58){\vector(0,1){9}}
\put(-105,92){\vector(0,-1){9}}
\put(-105,108){\vector(0,1){9}}
\put(-80,-25){$\vector(1,0){10}$}
\put(-70,0){$\vector(-1,0){10}$}
\put(-80,25){$\vector(1,0){10}$}
\put(-70,50){$\vector(-1,0){10}$}
\put(-80,75){$\vector(1,0){10}$}
\put(-70,100){$\vector(-1,0){10}$}
\put(-80,125){$\vector(1,0){10}$}
\put(-145,-26){$\dots$}
\put(-145,-1){$\dots$}
\put(-145,24){$\dots$}
\put(-145,49){$\dots$}
\put(-145,74){$\dots$}
\put(-145,99){$\dots$}
\put(-145,124){$\dots$}
} 
\dottedline(-100,-40)(170,-40)
\put(-125,-42){$\updownarrow$}
\put(-150,-42){$\mu^{\bullet}_+\mu^{\circ}_-$}
}
\put(80,-180)
{
\put(-140,46.5){$x(\frac{3}{2})$}
\put(-19,-7.5)
{
\put(50,0){\dashbox{3}(38,15){\small $*$}}
\put(150,0){\makebox(38,15){\small $*$}}
\put(50,50){\makebox(38,15){\small $*$}}
\put(150,50){\dashbox{3}(38,15){\small $*$}}
\put(50,100){\dashbox{3}(38,15){\small $*$}}
\put(150,100){\makebox(38,15){\small $*$}}
\put(100,-25){\dashbox{3}(38,15){\small $\diamond$}}
\put(100,0){\framebox(38,15){\small $\diamond$}}
\put(100,25){\dashbox{3}(38,15){\small $\diamond$}}
\put(100,50){\framebox(38,15){\small $\diamond$}}
\put(100,75){\dashbox{3}(38,15){\small $\diamond$}}
\put(100,100){\framebox(38,15){\small $\diamond$}}
\put(100,125){\dashbox{3}(38,15){\small $\diamond$}}
}
%
\put(70,0){$\vector(1,0){10}$}
\put(130,0){$\vector(-1,0){10}$}
\put(70,50){$\vector(1,0){10}$}
\put(130,50){$\vector(-1,0){10}$}
\put(70,100){$\vector(1,0){10}$}
\put(130,100){$\vector(-1,0){10}$}
\put(50,32){\vector(0,-1){14}}
\put(50,68){\vector(0,1){14}}
\put(100,-8){\vector(0,-1){9}}
\put(100,8){\vector(0,1){9}}
\put(100,42){\vector(0,-1){9}}
\put(100,58){\vector(0,1){9}}
\put(100,92){\vector(0,-1){9}}
\put(100,108){\vector(0,1){9}}
\put(150,18){\vector(0,1){14}}
\put(150,82){\vector(0,-1){14}}

\put(120,-15){\vector(2,1){10}}
\put(120,15){\vector(2,-1){10}}
\put(120,85){\vector(2,1){10}}
\put(120,115){\vector(2,-1){10}}
\put(80,35){\vector(-2,1){10}}
\put(80,65){\vector(-2,-1){10}}
%
%
\put(30,0)
{
\put(-169,-7.5)
{
\put(40,-25){\framebox(48,15){\small $x^{(r-2)}_1(1)$}}
\put(40,0){\dashbox{3}(48,15){\small $*$}}
\put(40,25){\framebox(48,15){\small $x^{(r-2)}_3(1)$}}
\put(40,50){\dashbox{3}(48,15){\small $*$}}
\put(40,75){\framebox(48,15){\small $x^{(r-2)}_5(1)$}}
\put(40,100){\dashbox{3}(48,15){\small $*$}}
\put(40,125){\framebox(48,15){\small $x^{(r-2)}_7(1)$}}
\put(100,-25){\dashbox{3}(48,15){\small $*$}}
\put(100,0){\framebox(48,15){\small $x^{(r-1)}_2(1)$}}
\put(100,25){\dashbox{3}(48,15){\small $*$}}
\put(100,50){\framebox(48,15){\small $x^{(r-1)}_4(1)$}}
\put(100,75){\dashbox{3}(48,15){\small $*$}}
\put(100,100){\framebox(48,15){\small $x^{(r-1)}_6(1)$}}
\put(100,125){\dashbox{3}(48,15){\small $*$}}
}
\put(-45,-8){\vector(0,-1){9}}
\put(-45,8){\vector(0,1){9}}
\put(-45,42){\vector(0,-1){9}}
\put(-45,58){\vector(0,1){9}}
\put(-45,92){\vector(0,-1){9}}
\put(-45,108){\vector(0,1){9}}
\put(-105,-17){\vector(0,1){9}}
\put(-105,17){\vector(0,-1){9}}
\put(-105,33){\vector(0,1){9}}
\put(-105,67){\vector(0,-1){9}}
\put(-105,83){\vector(0,1){9}}
\put(-105,117){\vector(0,-1){9}}
\put(-70,-25){$\vector(-1,0){10}$}
\put(-80,0){$\vector(1,0){10}$}
\put(-70,25){$\vector(-1,0){10}$}
\put(-80,50){$\vector(1,0){10}$}
\put(-70,75){$\vector(-1,0){10}$}
\put(-80,100){$\vector(1,0){10}$}
\put(-70,125){$\vector(-1,0){10}$}
\put(-145,-26){$\dots$}
\put(-145,-1){$\dots$}
\put(-145,24){$\dots$}
\put(-145,49){$\dots$}
\put(-145,74){$\dots$}
\put(-145,99){$\dots$}
\put(-145,124){$\dots$}
} 
\dottedline(-100,-40)(170,-40)
\put(-130,-42){$\updownarrow$}
\put(-145,-42){$\mu^{\bullet}_-$}
}
\end{picture}
\caption{(Continues from Figure \ref{fig:labelxC1}).
}
\label{fig:labelxC2}
\end{figure}

\subsection{Cluster algebra and alternative labels}

It is standard to identify
a quiver $Q$
with no loop and no 2-cycle
and a skew-symmetric matrix $B$.
We use the convention for the direction of arrows as
\begin{align}
i \longrightarrow j\quad
\Longleftrightarrow
\quad
B_{ij}=1.
\end{align}
(In this paper we only encounter the situation
where $B_{ij}=-1,0,1$.)
Let $B_{\ell}(C_r)$
be the corresponding skew-symmetric matrix 
to the quiver $Q_{\ell}(C_r)$.
In the rest of the section,
we set the matrix $B=(B_{\mathbf{i}\mathbf{j}})_{\mathbf{i},
\mathbf{j}\in \mathbf{I}}=B_{\ell}(C_r)$
unless otherwise mentioned.

Let $\mathcal{A}(B,x,y)$ 
be the {\em cluster algebra
 with coefficients
in the universal
semifield
$\mathbb{Q}_{\mathrm{sf}}(y)$},
where $(B,x,y)$ is the initial seed  \cite{FZ2}.
See also \cite[Section 2.1]{IIKKN} for
the conventions and notations on
cluster algebras we employ.
(Here we use the symbol $+$ instead of $\oplus$ 
in $\mathbb{Q}_{\mathrm{sf}}(y)$,
since it is the ordinary addition of subtraction-free
expressions of rational functions of $y$.)

\begin{definition}
The {\em coefficient group $\mathcal{G}(B,y)$
associated with $\mathcal{A}(B,x,y)$}
is the multiplicative subgroup of
the semifield $\mathbb{Q}_{\mathrm{sf}}(y)$ generated by all
the coefficients $y_{\mathbf{i}}'$ of $\mathcal{A}(B,x,y)$
together with $1+y_{\mathbf{i}}'$.
\end{definition}

In view of Lemma \ref{lem:CQmut}
we set $x(0)=x$, $y(0)=y$ and define 
clusters $x(u)=(x_{\mathbf{i}}(u))_{\mathbf{i}\in \mathbf{I}}$
 ($u\in \frac{1}{2}\mathbb{Z}$)
 and coefficient tuples $y(u)=(y_\mathbf{i}(u))_{\mathbf{i}\in \mathbf{I}}$
 ($u\in \frac{1}{2}\mathbb{Z}$)
by the sequence of mutations
\begin{align}
\label{eq:Cmutseq}
\begin{split}
\cdots
& 
\mathop{\longleftrightarrow}^{\mu^{\bullet}_-}
\
(B,x(0),y(0))\
\mathop{\longleftrightarrow}^{\mu^{\bullet}_+
\mu^{\circ}_+}
\
(-B,x({\textstyle \frac{1}{2}),y(\frac{1}{2})})
\\
&
\mathop{\longleftrightarrow}^{\mu^{\bullet}_-}
\
(\boldsymbol{r}(B),x(1),y(1))
\
\mathop{\longleftrightarrow}^{\mu^{\bullet}_+
\mu^{\circ}_-}
\
(-\boldsymbol{r}(B),x({\textstyle \frac{3}{2}),y(\frac{3}{2})})
\
\mathop{\longleftrightarrow}^{\mu^{\bullet}_-}
\
\cdots,
\end{split}
\end{align}
where $\boldsymbol{r}(B)=B'$ is defined
by $B'_{\mathbf{i}\mathbf{j}}=B_{\boldsymbol{r}(\mathbf{i})
\boldsymbol{r}(\mathbf{j})}$.

For a pair $(\mathbf{i},u)\in
 \mathbf{I}\times \frac{1}{2}\mathbb{Z}$,
we set the parity condition $\mathbf{p}_+$
and $\mathbf{p}_-$ by
\begin{align}
\label{eq:Cpp}
\mathbf{p}_+:&
\begin{cases}
 \mathbf{i}\in \mathbf{I}^{\circ}_+\sqcup \mathbf{I}^{\bullet}_+
& u\equiv 0\\
 \mathbf{i}\in \mathbf{I}^{\bullet}_-
& u\equiv \frac{1}{2},\frac{3}{2}\\
 \mathbf{i}\in \mathbf{I}^{\circ}_-\sqcup \mathbf{I}^{\bullet}_+
& u\equiv 1,
\end{cases}
\quad
\mathbf{p}_-:
\begin{cases}
 \mathbf{i}\in \mathbf{I}^{\circ}_+\sqcup \mathbf{I}^{\bullet}_+
& u\equiv \frac{1}{2}\\
 \mathbf{i}\in \mathbf{I}^{\bullet}_-
& u\equiv 0,1\\
 \mathbf{i}\in \mathbf{I}^{\circ}_-\sqcup \mathbf{I}^{\bullet}_+
& u\equiv \frac{3}{2},
\end{cases}
\end{align}
where $\equiv$ is modulo $2\mathbb{Z}$.
We have
\begin{align}
\label{eq:Cppm}
(\mathbf{i},u):\mathbf{p}_+
\quad
\Longleftrightarrow
\quad
(\mathbf{i},u+\frac{1}{2}):\mathbf{p}_-.
\end{align}
Each $(\mathbf{i},u):\mathbf{p}_+$
is a mutation point of \eqref{eq:Cmutseq} in the forward
direction of $u$,
and each  $(\mathbf{i},u):\mathbf{p}_-$
is so in the backward direction of $u$.
Notice that there are also some $(\mathbf{i},u)$ which
do not satisfy $\mathbf{p}_+$ nor $\mathbf{p}_-$,
and are not mutation points of \eqref{eq:Cmutseq};
explicitly, they are $(\mathbf{i},u)$
with $\mathbf{i}\in \mathbf{I}^{\circ}_+$,
$u\equiv 1,\frac{3}{2}$ mod $2\mathbb{Z}$,
or 
with $\mathbf{i}\in \mathbf{I}^{\circ}_-$,
$u\equiv 0,\frac{1}{2}$ mod $2\mathbb{Z}$.

There is a correspondence between the parity condition
$\mathbf{p}_{\pm}$ here and $\mathbf{P}_{\pm}$,
$\mathbf{P}'_{\pm}$
in \eqref{eq:CPcond} and \eqref{eq:CPcond2}.
\begin{lemma}
 Below $\equiv$ means the equivalence modulo
$2\mathbb{Z}$.
\par
(i)
The map
\begin{align}
\begin{matrix}
g: &\mathcal{I}_{\ell+}&\rightarrow & \{ (\mathbf{i},u): \mathbf{p}_+
\}\hfill \\
&(a,m,u-\frac{1}{t_a})&\mapsto &
\begin{cases}
((a,m),u)& \mbox{\rm  $a\neq r$}\\
((r+1,m),u)& \mbox{\rm $a=r$;
$m+u\equiv 0$}\\
((r,m),u)& 
 \mbox{$a=r$; $m+u\equiv 1$}\\
\end{cases}
\end{matrix}
\end{align}
is a bijection.
\par
(ii)
The map
\begin{align}
\begin{matrix}
g': &\mathcal{I}'_{\ell+}&\rightarrow & \{ (\mathbf{i},u): \mathbf{p}_+
\ \mbox{\rm or}\ \mathbf{p}_-\}\hfill \\
&(a,m,u)&\mapsto &
\begin{cases}
((a,m),u)& a\neq r\\
((r+1,m),u)& \mbox{\rm $a=r$;
$m+u\equiv 0$}\\
((r,m),u)& 
 \mbox{$a=r$; $m+u\equiv 1$}\\
\end{cases}
\end{matrix}
\end{align}
is a bijection.
\end{lemma}

We introduce alternative labels
$x_{\mathbf{i}}(u)=x^{(a)}_m(u-1/t_a)$
($(a,m,u-1/t_a)\in \mathcal{I}_{\ell+}$)
for $(\mathbf{i},u)=g((a,m,u-1/t_a))$
and
$y_{\mathbf{i}}(u)=y^{(a)}_m(u)$
($(a,m,u)\in \mathcal{I}'_{\ell+}$)
for $(\mathbf{i},u)=g'((a,m,u))$,
respectively.
See Figures \ref{fig:labelxC1}--\ref{fig:labelxC2}.

\subsection{T-system and cluster algebra}
The result in this subsection is completely parallel
to the $B_r$ case \cite{IIKKN}.

Let $\mathcal{A}(B,x)$ be the cluster algebra
with trivial coefficients, where $(B,x)$ is
the initial seed \cite{FZ2}.
Let $\mathbf{1}=\{1\}$ 
be the {\em trivial semifield}
and $\pi_{\mathbf{1}}:
\mathbb{Q}_{\mathrm{sf}}(y)\rightarrow 
\mathbf{1}$, $y_{\mathbf{i}}\mapsto 1$ be the projection.
Let $[x_{\mathbf{i}}(u)]_{\mathbf{1}}$
denote the image of $x_{\mathbf{i}}(u)$
 by the algebra homomorphism
$\mathcal{A}(B,x,y)\rightarrow \mathcal{A}(B,x)$
 induced from $\pi_{\mathbf{1}}$.
It is called the {\em trivial evaluation}.

Recall that $G(b,k,v;a,m,u)$ is defined in \eqref{eq:Tu}.

\begin{lemma}
\label{lem:Cx2}
The family $\{x^{(a)}_m(u)
\mid (a,m,u)\in \mathcal{I}_{\ell+}\}$
satisfies a system of relations
\begin{align}
\label{eq:Cx2}
\begin{split}
x^{(a)}_{m}\left(u-\textstyle\frac{1}{t_a}\right)
x^{(a)}_{m}\left(u+\textstyle\frac{1}{t_a}\right)
&=
\frac{y^{(a)}_m(u)}{1+y^{(a)}_m(u)}
\prod_{(b,k,v)\in \mathcal{I}_{\ell+}}
x^{(b)}_{k}(v)^{G(b,k,v;\, a,m,u)}
\\
&\qquad +
\frac{1}{1+y^{(a)}_m(u)}
x^{(a)}_{m-1}(u)x^{(a)}_{m+1}(u),
\end{split}
\end{align}
where $(a,m,u)\in \mathcal{I}'_{\ell+}$.
In particular,
the family $\{ [x^{(a)}_m(u)]_{\mathbf{1}}
\mid (a,m,u)\in \mathcal{I}_{\ell+}\}$
satisfies the T-system $\mathbb{T}_{\ell}(C_r)$
in $\mathcal{A}(B,x)$
by replacing $T^{(a)}_m(u)$ with $[x^{(a)}_m(u)]_{\mathbf{1}}$.
\end{lemma}

\begin{definition}
\label{def:Tsub}
The {\em T-subalgebra
$\mathcal{A}_T(B,x)$
of ${\mathcal{A}}(B,x,y)$
associated with the sequence \eqref{eq:Cmutseq}}
is the subalgebra of
${\mathcal{A}}(B,x)$
generated by
$[x_{\mathbf{i}}(u)]_{\mathbf{1}}$
($(\mathbf{i},u)\in \mathbf{I}\times \frac{1}{2}\mathbb{Z}$).
\end{definition}

\begin{theorem}
\label{thm:CTiso}
The ring $\EuScript{T}^{\circ}_{\ell}(C_r)_+$ is isomorphic to
$\mathcal{A}_T(B,x)$ by the correspondence
$T^{(a)}_m(u)\mapsto [x^{(a)}_m(u)]_{\mathbf{1}}$.
\end{theorem}

\subsection{Y-system and cluster algebra}

The result in this subsection is completely parallel
to the $B_r$ case \cite{IIKKN}.

\begin{lemma}
\label{lem:Cy2}
The family $\{ y^{(a)}_m(u)
\mid (a,m,u)\in \mathcal{I}'_{\ell+}\}$
satisfies the Y-system $\mathbb{Y}_{\ell}(C_r)$
by replacing $Y^{(a)}_m(u)$ with $y^{(a)}_m(u)$.
\end{lemma}

\begin{definition}
\label{def:Ysub}
The {\em Y-subgroup
$\mathcal{G}_Y(B,y)$
of ${\mathcal{G}}(B,y)$
associated with the sequence \eqref{eq:Cmutseq}}
is the subgroup of
${\mathcal{G}}(B,y)$ 
generated by
$y_{\mathbf{i}}(u)$
($(\mathbf{i},u)\in \mathbf{I}\times \frac{1}{2}\mathbb{Z}$)
and $1+y_{\mathbf{i}}(u)$
($(\mathbf{i},u):\mathbf{p}_+$ or $\mathbf{p}_-$).
\end{definition}

\begin{theorem}
\label{thm:CYiso}
The group $\EuScript{Y}^{\circ}_{\ell}(C_r)_+$ is isomorphic to
$\mathcal{G}_Y(B,y)$ by the correspondence
$Y^{(a)}_m(u)\mapsto y^{(a)}_m(u)$
and $1+Y^{(a)}_m(u)\mapsto 1+y^{(a)}_m(u)$.
\end{theorem}

\subsection{Tropical Y-system at level 2}

The {\em tropical semifield} $\mathrm{Trop}(y)$
is an abelian multiplicative group freely generated by
the elements $y_{\mathbf{i}}$ ($\mathbf{i}\in \mathbf{I}$)
with the addition $\oplus$ 
\begin{align}
\label{eq:Ctrop}
\prod_{\mathbf{i}\in \mathbf{I}}y_{\mathbf{i}}^{a_{\mathbf{i}}}
\oplus
\prod_{\mathbf{i}\in \mathbf{I}}y_{\mathbf{i}}^{b_{\mathbf{i}}}
=
\prod_{\mathbf{i}\in \mathbf{I}}
y_{\mathbf{i}}^{\min(a_\mathbf{i},b_\mathbf{i})}.
\end{align}
Let $\pi_{\mathbf{T}}:
\mathbb{Q}_{\mathrm{sf}}(y)\rightarrow 
\mathrm{Trop}(y)$, $y_{\mathbf{i}}\mapsto 
y_{\mathbf{i}}$ be the projection.
Let $[y_{\mathbf{i}}(u)]_{\mathbf{T}}$
and $[\mathcal{G}_Y(B,y)]_{\mathbf{T}}$
denote the images of $y_{\mathbf{i}}(u)$
and $\mathcal{G}_Y(B,y)$
by the multiplicative group
 homomorphism induced from $\pi_{\mathbf{T}}$, respectively.
They are called the {\em tropical evaluations},
and the resulting relations in
the group $[\mathcal{G}_Y(B,y)]_{\mathbf{T}}$
is called the {\em tropical Y-system}.

We say a (Laurent) monomial $m=\prod_{\mathbf{i}\in \mathbf{I}}
y_{\mathbf{i}}^{k_{\mathbf{i}}}$
is {\em positive} (resp. {\em negative})
if $m\neq 1$ and $k_{\mathbf{i}}\geq 0$
(resp. $k_{\mathbf{i}}\leq 0$)
for any $\mathbf{i}$.

The following properties
of the tropical Y-system at level 2
will be the key in the entire method.

\begin{proposition}
\label{prop:Clev2}
 For 
$[\mathcal{G}_Y(B,y)]_{\mathbf{T}}$
with $B=B_{2}(C_r)$, the following facts hold.
\par
(i) Let $u$ be in the region $0\le u < 2$.
For any $(\mathbf{i},u):\mathbf{p}_+$,
the  monomial $[y_{\mathbf{i}}(u)]_{\mathbf{T}}$
is positive.
\par
(ii) Let $u$ be in the region $-h^{\vee}\le u < 0$.
\begin{itemize}
\item[\em (a)]
 Let $\mathbf{i}=(i,2)$ $(i\leq r-1)$,
$(r,1)$, or $(r+1,1)$.
For any $(\mathbf{i},u):\mathbf{p}_+$,
the  monomial $[y_{\mathbf{i}}(u)]_{\mathbf{T}}$
is negative.
\item[\em (b)]
Let $\mathbf{i}=(i,1), (i,3)$  $(i\leq r-1)$.
For any $(\mathbf{i},u):\mathbf{p}_+$,
the  monomial $[y_{\mathbf{i}}(u)]_{\mathbf{T}}$
is positive for $u=-\frac{1}{2}h^{\vee},
-\frac{1}{2}h^{\vee}-\frac{1}{2}$
and negative otherwise.
\end{itemize}
\par
(iii)
$y_{ii'}(2)=y_{i,4-i'}^{-1}$ if $i\leq r-1$ and
$y_{ii'}^{-1}$ if $i=r,r+1$.
\par
(iv) For even $r$,  $y_{ii'}(-h^{\vee})=
y_{ii'}^{-1}$ if $i\leq r-1$ and 
$y_{2r+1-i,i'}^{-1}$ if  $i=r, r+1$.
 For odd $r$,  $y_{ii'}(-h^{\vee})=
y_{ii'}^{-1}$.
\end{proposition}

One can directly verify  (i) and (iii)
in the same way as the $B_r$ case \cite[Proposition 3.2]{IIKKN}.
In the rest of this subsection
 we give the outline of the proof of
(ii) and (iv).
Note that (ii) and (iv) can be proved
independently for each variable $y_{\mathbf{i}}$.
(To be precise, we also need to assure that
each monomial is not 1 in total.
However, this can be easily followed up, so that we do not
describe details here.)
Below we separate the variables into two parts.
Here is a brief summary of the results.

(1) {\em The D part.} 
 The powers of $[y_{\mathbf{i}}(u)]_{\mathbf{T}}$
in  the variables
$y_{i,2}$ ($i\leq r-1$)
and $y_{r,1}$, $y_{r+1,1}$
are described by the root system of type $D_{r+1}$
with a Coxeter-like transformation.
It turns out that they are further described
by  (a subset of ) the root system of type $A_{2r+1}$
with the Coxeter transformation.

(2) {\em The A part.} 
 The powers of $[y_{\mathbf{i}}(u)]_{\mathbf{T}}$
in  the variables
$y_{i,1}$ and $y_{i,3}$ ($i\leq r-1$),
are mainly described by the root system of type $A_{r-1}$
with the Coxeter transformation.

\subsubsection{D part}

Let us consider the $D$ part first.
Let $D_{r+1}$ be the Dynkin diagram of type $D$
with index set $J=\{1,\dots,r+1\}$.
We assign the sign +/$-$ to vertices of $D_{r+1}$
(no sign for $r$ and $r+1$) as inherited from
$Q_2(C_r)$.
\begin{align*}
\begin{picture}(310,68)(0,20)
\put(45,70)
{
\put(0,0){\circle{5}}
\put(30,0){\circle{5}}
\put(60,0){\circle{5}}
\put(90,0){\circle{5}}
\put(120,0){\circle{5}}
\put(150,0){\circle{5}}
\put(180,-10){\circle{5}}
\put(180,10){\circle{5}}
\put(27,0){\line(-1,0){24}}
\put(57,0){\line(-1,0){24}}
\put(87,0){\line(-1,0){24}}
\put(93,0){\line(1,0){24}}
\put(147,0){\line(-1,0){24}}
\put(153,1){\line(3,1){24}}
\put(153,-1){\line(3,-1){24}}
\put(-4,8)
{
\put(2,2){\small $1$}
\put(32,2){\small $2$}
\put(62,2){\small }
\put(92,2){\small $$}
\put(122,2){\small }
\put(142,2){\small $r-1$}
\put(192,0){\small $r$}
\put(192,-20){\small $r+1$}
}
\put(-5,-14)
{
\put(2,2){\small $-$}
\put(32,2){\small $+$}
\put(62,2){\small $-$}
\put(122,2){\small $+$}
\put(152,2){\small $-$}
}
\put(230,-2){$r$: even}
}
%
%
\put(45,30)
{
\put(-30,0){\circle{5}}
\put(0,0){\circle{5}}
\put(30,0){\circle{5}}
\put(60,0){\circle{5}}
\put(90,0){\circle{5}}
\put(120,0){\circle{5}}
\put(150,0){\circle{5}}
\put(180,10){\circle{5}}
\put(180,-10){\circle{5}}
\put(-3,0){\line(-1,0){24}}
\put(27,0){\line(-1,0){24}}
\put(57,0){\line(-1,0){24}}
\put(87,0){\line(-1,0){24}}
\put(93,0){\line(1,0){24}}
\put(147,0){\line(-1,0){24}}
\put(153,1){\line(3,1){24}}
\put(153,-1){\line(3,-1){24}}
\put(-4,8)
{
\put(-28,2){\small $1$}
\put(2,2){\small $2$}
\put(62,2){\small }
\put(122,2){\small }
\put(142,2){\small $r-1$}
\put(192,0){\small $r$}
\put(192,-20){\small $r+1$}
}
\put(-5,-14)
{
\put(-28,2){\small $+$}
\put(2,2){\small $-$}
\put(32,2){\small $+$}
\put(62,2){\small $-$}
\put(122,2){\small $+$}
\put(152,2){\small $-$}
}
\put(230,-2){$r$: odd}
}
\end{picture}
\end{align*}

Let $\Pi=\{ \alpha_1,\dots,\alpha_{r+1}\}$, $-\Pi$,
 $\Phi_+$ be the set of the simple roots,
the negative simple roots, the positive roots, respectively,
of type $D_{r+1}$.
Following \cite{FZ1}, we introduce the {\em piecewise-linear
analogue} $\sigma_i$ of the simple reflection $s_i$,
acting on the set
of the {\em almost positive roots}
$\Phi_{\geq -1}=\Phi_{+}\sqcup (-\Pi)$,
by
\begin{align}
\label{eq:pl}
\begin{split}
\sigma_i(\alpha)&=s_i(\alpha),\quad  \alpha\in \Phi_+,\\
\sigma_i(-\alpha_j)&=
\begin{cases}
\alpha_j& j=i,\\
-\alpha_j&\mbox{otherwise}.\\
\end{cases}
\end{split}
\end{align}
Let
\begin{align}
\sigma_+=\prod_{i\in J_+} \sigma_i,\quad
\sigma_-=\prod_{i\in J_-} \sigma_i,
\end{align}
where $J_{\pm}$ is the set of the
vertices of $D_{r+1}$
 with property $\pm$. We define $\sigma$ as the composition
\begin{align}
\label{eq:Csigma}
\sigma=\sigma_- \sigma_+\sigma_{r+1}
\sigma_- \sigma_+\sigma_{r}.
\end{align}

\begin{lemma}
\label{lem:Corbit}
 The following facts hold.
\par\noindent
(I) Let $r$ be even.
\par
 (i) For $i\leq r-1$,
$\sigma^{k}(-\alpha_i)\in\Phi_+$, $(1\leq k\leq r/2)$,
$\sigma^{r/2+1}(-\alpha_i)=-\alpha_{i}$.
\par
 (ii) For $i\leq r-1$,
$\sigma^{k}(\alpha_i)\in\Phi_+$, $(0\leq k\leq r/2)$,
$\sigma^{r/2+1}(\alpha_i)=\alpha_{i}$.
\par
(iii) 
$\sigma^{k}(-\alpha_{r})\in\Phi_+$, $(1\leq k\leq r/2)$,
$\sigma^{r/2+1}(-\alpha_{r})=-\alpha_{r+1}$.
\par
(iv) 
$\sigma^{k}(-\alpha_{r+1})\in\Phi_+$, $(1\leq k\leq r/2+1)$,
$\sigma^{r/2+2}(-\alpha_{r+1})=-\alpha_{r}$.
\par
(v) The elements in $\Phi_+$ in (i)--(iv) exhaust the set $\Phi_+$,
thereby providing the orbit decomposition
of $\Phi_+$ by $\sigma$.
\par\noindent
(II) Let $r$ be odd.
\par
 (i) For $i\in J_+$,
$\sigma^{k}(-\alpha_i)\in\Phi_+$, $(1\leq k\leq r+1)$,
$\sigma^{r+2}(-\alpha_i)=-\alpha_{i}$,
$\sigma^{(r+1)/2}(-\alpha_i)=\alpha_{i}$.
\par
 (ii) For $i\in J_-$,
$\sigma^{k}(-\alpha_i)\in\Phi_+$, $(1\leq k\leq r+1)$,
$\sigma^{r+2}(-\alpha_i)=-\alpha_{i}$,
$\sigma^{(r+3)/2}(-\alpha_i)=\alpha_{i}$.
\par
(iii) 
$\sigma^{k}(-\alpha_r)\in\Phi_+$, $(1\leq k\leq (r+1)/2)$,
$\sigma^{(r+3)/2}(-\alpha_r)=-\alpha_{r}$.
\par
(iv) 
$\sigma^{k}(-\alpha_{r+1})\in\Phi_+$, $(1\leq k\leq (r+1)/2)$,
$\sigma^{(r+3)/2}(-\alpha_{r+1})=-\alpha_{r+1}$.\par
(v) The elements in $\Phi_+$ in (i)--(iv) exhaust the set $\Phi_+$,
thereby providing the orbit decomposition
of $\Phi_+$ by $\sigma$.
\end{lemma}
\begin{proof}
They are verified by explicitly calculating
$\sigma^{k}(-\alpha_i)$ and $\sigma^{k}(\alpha_i)$.
The examples for $r=10$ (for even $r$) and 9 (for odd $r$)
are given in
Tables \ref{tab:orbit10} and \ref{tab:orbit9},
respectively,
where we use the notations
\begin{align}
\label{eq:Cij}
\begin{split}
[i,j]&=\alpha_i+\cdots +\alpha_j\quad (1\leq i<j \leq r),
\quad [i]=\alpha_i\quad (1\leq i \leq  r),
\\
\{i,j\}&=(\alpha_i+\cdots +\alpha_{r-1})
+(\alpha_j+\cdots +\alpha_{r+1})
\quad (1\leq i<j \leq r+1, i\leq r-1),
\end{split}
\end{align}
and $\{r+1\}=\alpha_{r+1}$.
In fact, it is not difficult to read off the general rule from
these examples.
\end{proof}

\begin{table}
\rotatebox{90}
{
\begin{minipage}{\textheight}
\setlength{\unitlength}{0.95pt}
\begin{picture}(310,230)(-5,-30)
\put(100,200){ $-1$}
\put(140,200){ $-2$}
\put(180,200){ $-3$}
\put(220,200){ $-4$}
\put(260,200){ $-5$}
\put(300,200){ $-6$}
\put(340,200){ $-7$}
\put(380,200){ $-8$}
\put(420,200){ $-9$}
\put(457,200){ $-10$}
\put(497,200){ $-11$}
\put(-2,180){ 1\ $-$}
\put(36,180){  $-\alpha_{1}$}
\put(80,180){ [1]}
\put(117,180){ [2,3]}
\put(157,180){ [4,5]}
\put(197,180){ [6,7]}
\put(237,180){ [8,9]}
\put(277,180){ $\{11\}$}
\put(317,180){ $[9,10]$}
\put(357,180){ [7,8]}
\put(397,180){ $[5,6]$}
\put(437,180){ $[3,4]$}
\put(477,180){ $[1,2]$}
\put(518,180){ $-\alpha_{1}$}
\put(560,180){ $\alpha_{1}$}
\put(-2,160){ 2\ $+$}
\put(25,160){  $\alpha_2$}
\put(53,160){  $-\alpha_2$}
\put(97,160){ [1,3]}
\put(137,160){ [2,5]}
\put(177,160){ [4,7]}
\put(217,160){ [6,9]}
\put(257,160){ $\{8,11\}$}
\put(297,160){ $\{9,10\}$}
\put(337,160){ $[7,10]$}
\put(377,160){ [5,8]}
\put(417,160){ [3,6]}
\put(457,160){ [1,4]}
\put(500,160){ [2]}
\put(536,160){ $-\alpha_{2}$}
\put(-2,140){ 3\ $-$}
\put(36,140){  $-\alpha_{3}$}
\put(80,140){ [3]}
\put(117,140){ [1,5]}
\put(157,140){ [2,7]}
\put(197,140){ [4,9]}
\put(237,140){ \{6,11\}}
\put(277,140){ $\{8,9\}$}
\put(317,140){ $\{7,10\}$}
\put(357,140){ $[5,10]$}
\put(397,140){ [3,8]}
\put(437,140){ $[1,6]$}
\put(477,140){ $[2,4]$}
\put(518,140){ $-\alpha_{3}$}
\put(560,140){ $\alpha_{3}$}
\put(-2,120){ 4\ $+$}
\put(25,120){  $\alpha_4$}
\put(53,120){  $-\alpha_4$}
\put(97,120){ [3,5]}
\put(137,120){ [1,7]}
\put(177,120){ [2,9]}
\put(217,120){ \{4,11\}}
\put(257,120){ $\{6,9\}$}
\put(297,120){ $\{7,8\}$}
\put(337,120){ $\{5,10\}$}
\put(377,120){ $[3,10]$}
\put(417,120){ [1,8]}
\put(457,120){ [2,6]}
\put(500,120){ [4]}
\put(536,120){ $-\alpha_{4}$}
\put(-2,100){ 5 $-$}
\put(36,100){  $-\alpha_{5}$}
\put(80,100){ [5]}
\put(117,100){ [3,7]}
\put(157,100){ [1,9]}
\put(197,100){ \{2,11\}}
\put(237,100){ $\{4,9\}$}
\put(277,100){ $\{6,7\}$}
\put(317,100){ $\{5,8\}$}
\put(357,100){ $\{3,10\}$}
\put(397,100){ [1,10]}
\put(437,100){ [2,8]}
\put(477,100){ $[4,6]$}
\put(518,100){ $-\alpha_{5}$}
\put(560,100){ $\alpha_{5}$}
\put(-2,80){ 6\ $+$}
\put(25,80){  $\alpha_6$}
\put(53,80){  $-\alpha_6$}
\put(97,80){ [5,7]}
\put(137,80){ [3,9]}
\put(177,80){ \{1,11\}}
\put(217,80){ $\{2,9\}$}
\put(257,80){ $\{4,7\}$}
\put(297,80){ $\{5,6\}$}
\put(337,80){ $\{3,8\}$}
\put(377,80){ $\{1,10\}$}
\put(417,80){ [2,10]}
\put(457,80){ [4,8]}
\put(500,80){ [6]}
\put(536,80){ $-\alpha_{6}$}
\put(-2,60){ 7\ $-$}
\put(36,60){  $-\alpha_{7}$}
\put(80,60){ [7]}
\put(117,60){ [5,9]}
\put(157,60){ \{3,11\}}
\put(197,60){ $\{1,9\}$}
\put(237,60){ $\{2,7\}$}
\put(277,60){ $\{4,5\}$}
\put(317,60){ $\{3,6\}$}
\put(357,60){ $\{1,8\}$}
\put(397,60){ $\{2,10\}$}
\put(437,60){ [4,10]}
\put(477,60){ $[6,8]$}
\put(518,60){ $-\alpha_{7}$}
\put(560,60){ $\alpha_{7}$}
\put(-2,40){ 8\ $+$}
\put(25,40){  $\alpha_8$}
\put(53,40){  $-\alpha_8$}
\put(97,40){ [7,9]}
\put(137,40){ \{5,11\}}
\put(177,40){ $\{3,9\}$}
\put(217,40){ $\{1,7\}$}
\put(257,40){ $\{2,5\}$}
\put(297,40){ $\{3,4\}$}
\put(337,40){ $\{1,6\}$}
\put(377,40){ $\{2,8\}$}
\put(417,40){ $\{4,10\}$}
\put(457,40){ [6,10]}
\put(500,40){ [8]}
\put(536,40){ $-\alpha_{8}$}
\put(-2,20){ 9\ $-$}
\put(36,20){  $-\alpha_{9}$}
\put(80,20){ [9]}
\put(117,20){ \{7,11\}}
\put(157,20){ $\{5,9\}$}
\put(197,20){ $\{3,7\}$}
\put(237,20){ $\{1,5\}$}
\put(277,20){ $\{2,3\}$}
\put(317,20){ $\{1,4\}$}
\put(357,20){ $\{2,6\}$}
\put(397,20){ $\{4,8\}$}
\put(437,20){ $\{6,10\}$}
\put(477,20){ [8,10]}
\put(518,20){ $-\alpha_{9}$}
\put(560,20){ $\alpha_{9}$}
%
\put(18,0){  $-\alpha_{11}$}
\put(52,0){  $-\alpha_{10}$}
\put(97,0){ \{9,11\}}
\put(137,0){ $\{7,9\}$}
\put(177,0){ $\{5,7\}$}
\put(217,0){ $\{3,5\}$}
\put(257,0){ $\{1,3\}$}
\put(297,0){ $\{1,2\}$}
\put(337,0){ $\{2,4\}$}
\put(377,0){ $\{4,6\}$}
\put(417,0){ $\{6,8\}$}
\put(457,0){ $\{8,10\}$}
\put(498,0){ [10]}
\put(533,0){ $-\alpha_{11}$}
\put(570,0){ $-\alpha_{10}$}
\put(-5,212){\line(1,0){605}}
\put(-5,193){\line(1,0){605}}
\put(-5,-7){\line(1,0){605}}
\put(-5,-7){\line(0,1){219}}
\put(20,-7){\line(0,1){219}}
\put(80,-7){\line(0,1){219}}
\put(520,-7){\line(0,1){219}}
\put(600,-7){\line(0,1){219}}
%
\end{picture}
\caption{The orbits  $\sigma^k(-\alpha_i)$
and $\sigma^k(\alpha_i)$
 in $\Phi_+$
by $\sigma$ of $\eqref{eq:Csigma}$
 for $r=10$.
The orbits of $-\alpha_i$
and $\alpha_i$ ($i\leq 8$),
for example, $-\alpha_1\rightarrow [2,3]
\rightarrow [6,7]\rightarrow \cdots \rightarrow -\alpha_1$
and
$\alpha_1
\rightarrow [4,5]
\rightarrow [8,9]\rightarrow \cdots \rightarrow \alpha_1 $,
 are aligned alternatively.
The orbits of $-\alpha_{10}$ and $-\alpha_{11}$,
namely,
 $-\alpha_{10}\rightarrow \{7,9\}
\rightarrow \{3,5\}\rightarrow \cdots \rightarrow -\alpha_{11}$,
and 
 $-\alpha_{11}\rightarrow \{9,11\}
\rightarrow \{5,7\}\rightarrow \cdots \rightarrow -\alpha_{10}$,
are aligned alternatively.
The numbers $-1, -2, \cdots$ in the head line
will be identified with the parameter $u$
in \eqref{eq:Calpha}.
}
\label{tab:orbit10}
\end{minipage}
}
\end{table}


\begin{table}
\rotatebox{90}
{
\begin{minipage}{\textheight}
\setlength{\unitlength}{0.95pt}
\begin{picture}(300,230)(-25,-30)
\put(100,180){ $-1$}
\put(140,180){ $-2$}
\put(180,180){ $-3$}
\put(220,180){ $-4$}
\put(260,180){ $-5$}
\put(300,180){ $-6$}
\put(340,180){ $-7$}
\put(380,180){ $-8$}
\put(420,180){ $-9$}
\put(456,180){ $-10$}
\put(-2,160){ 1\ $+$}
\put(25,160){  $\alpha_1$}
\put(53,160){  $-\alpha_1$}
\put(97,160){ [1,2]}
\put(137,160){ [3,4]}
\put(177,160){ [5,6]}
\put(217,160){ [7,8]}
\put(260,160){ \{10\}}
\put(297,160){ [8,9]}
\put(337,160){ [6,7]}
\put(377,160){ [4,5]}
\put(417,160){ [2,3]}
\put(460,160){ [1]}
\put(496,160){ $-\alpha_{1}$}
\put(-2,140){ 2\ $-$}
\put(36,140){  $-\alpha_{2}$}
\put(80,140){ [2]}
\put(117,140){ [1,4]}
\put(157,140){ [3,6]}
\put(197,140){ [5,8]}
\put(237,140){ \{7,10\}}
\put(277,140){ $\{8,9\}$}
\put(317,140){ [6,9]}
\put(357,140){ [4,7]}
\put(397,140){ [2,5]}
\put(437,140){ $[1,3]$}
\put(478,140){ $-\alpha_{2}$}
\put(522,140){ $\alpha_{2}$}
\put(-2,120){ 3\ $+$}
\put(25,120){  $\alpha_3$}
\put(53,120){  $-\alpha_3$}
\put(97,120){ [2,4]}
\put(137,120){ [1,6]}
\put(177,120){ [3,8]}
\put(217,120){ \{5,10\}}
\put(257,120){ $\{7,8\}$}
\put(297,120){ $\{6,9\}$}
\put(337,120){ [4,9]}
\put(377,120){ [2,7]}
\put(417,120){ [1,5]}
\put(460,120){ [3]}
\put(496,120){ $-\alpha_{3}$}
\put(-2,100){ 4 $-$}
\put(36,100){  $-\alpha_{4}$}
\put(80,100){ [4]}
\put(117,100){ [2,6]}
\put(157,100){ [1,8]}
\put(197,100){ \{3,10\}}
\put(237,100){ $\{5,8\}$}
\put(277,100){ $\{6,7\}$}
\put(317,100){ $\{4,9\}$}
\put(357,100){ [2,9]}
\put(397,100){ [1,7]}
\put(437,100){ $[3,5]$}
\put(478,100){ $-\alpha_{4}$}
\put(522,100){ $\alpha_{4}$}
\put(-2,80){ 5\ $+$}
\put(25,80){  $\alpha_5$}
\put(53,80){  $-\alpha_5$}
\put(97,80){ [4,6]}
\put(137,80){ [2,8]}
\put(177,80){ \{1,10\}}
\put(217,80){ $\{3,8\}$}
\put(257,80){ $\{5,6\}$}
\put(297,80){ $\{4,7\}$}
\put(337,80){ $\{2,9\}$}
\put(377,80){ [1,9]}
\put(417,80){ [3,7]}
\put(460,80){ [5]}
\put(496,80){ $-\alpha_{5}$}
\put(-2,60){ 6\ $-$}
\put(36,60){  $-\alpha_{6}$}
\put(80,60){ [6]}
\put(117,60){ [4,8]}
\put(157,60){ \{2,10\}}
\put(197,60){ $\{1,8\}$}
\put(237,60){ $\{3,6\}$}
\put(277,60){ $\{4,5\}$}
\put(317,60){ $\{2,7\}$}
\put(357,60){ $\{1,9\}$}
\put(397,60){ [3,9]}
\put(437,60){ $[5,7]$}
\put(478,60){ $-\alpha_{6}$}
\put(522,60){ $\alpha_{6}$}
\put(-2,40){ 7\ $+$}
\put(25,40){  $\alpha_7$}
\put(53,40){  $-\alpha_7$}
\put(97,40){ [6,8]}
\put(137,40){ \{4,10\}}
\put(177,40){ $\{2,8\}$}
\put(217,40){ $\{1,6\}$}
\put(257,40){ $\{3,4\}$}
\put(297,40){ $\{2,5\}$}
\put(337,40){ $\{1,7\}$}
\put(377,40){ $\{3,9\}$}
\put(417,40){ [5,9]}
\put(460,40){ [7]}
\put(496,40){ $-\alpha_{7}$}
\put(-2,20){ 8\ $-$}
\put(36,20){  $-\alpha_{8}$}
\put(80,20){ [8]}
\put(117,20){ \{6,10\}}
\put(157,20){ $\{4,8\}$}
\put(197,20){ $\{2,6\}$}
\put(237,20){ $\{1,4\}$}
\put(277,20){ $\{2,3\}$}
\put(317,20){ $\{1,5\}$}
\put(357,20){ $\{3,7\}$}
\put(397,20){ $\{5,9\}$}
\put(437,20){ [7,9]}
\put(478,20){ $-\alpha_{8}$}
\put(522,20){ $\alpha_{8}$}
\put(17,0){  $-\alpha_{10}$}
\put(53,0){  $-\alpha_{9}$}
\put(97,0){ \{8,10\}}
\put(137,0){ $\{6,8\}$}
\put(177,0){ $\{4,6\}$}
\put(217,0){ $\{2,4\}$}
\put(257,0){ $\{1,2\}$}
\put(297,0){ $\{1,3\}$}
\put(337,0){ $\{3,5\}$}
\put(377,0){ $\{5,7\}$}
\put(417,0){ $\{7,9\}$}
\put(460,0){ [9]}
\put(494,0){ $-\alpha_{10}$}
\put(532,0){ $-\alpha_{9}$}
\put(-5,192){\line(1,0){565}}
\put(-5,173){\line(1,0){565}}
\put(-5,-7){\line(1,0){565}}
\put(-5,-7){\line(0,1){199}}
\put(20,-7){\line(0,1){199}}
\put(80,-7){\line(0,1){199}}
\put(480,-7){\line(0,1){199}}
\put(560,-7){\line(0,1){199}}
%
\end{picture}
\caption{The orbit of $\sigma^k(-\alpha_i)$ in $\Phi_+$
by $\sigma$ of $\eqref{eq:Csigma}$
 for $r=9$.
The orbit of $-\alpha_i$ ($i\leq 8$),
for example, $-\alpha_1\rightarrow [3,4]
\rightarrow [7,8]\rightarrow \cdots \rightarrow \alpha_1 
\rightarrow [1,2]
\rightarrow [5,6]\rightarrow \cdots \rightarrow -\alpha_1 $,
is aligned in a cyclic and alternative way.
The orbits of $-\alpha_9$ and $-\alpha_{10}$,
namely,
 $-\alpha_9\rightarrow \{6,8\}
\rightarrow \{2,4\}\rightarrow \cdots \rightarrow -\alpha_9$,
and 
 $-\alpha_{10}\rightarrow \{8,10\}
\rightarrow \{4,6\}\rightarrow \cdots \rightarrow -\alpha_{10}$,
are aligned alternatively.
The numbers $-1, -2, \cdots$ in the head line
will be identified with the parameter $u$
in \eqref{eq:Calpha}.
}
\label{tab:orbit9}
\end{minipage}
}
\end{table}

The orbits  $\sigma(-\alpha_i)$ and $\sigma(\alpha_i)$
are further described by (a subset of)
 the  {\em root system of type $A_{2r+1}$}.
Let $\Pi'=\{\alpha'_1,\dots,\alpha'_{2r+1}\}$
and $\Phi'_+$
 be the sets of the simple roots
and the positive roots
of type $A_{2r+1}$, respectively,
 with standard index set $J'=\{1,\dots,
2r+1\}$. Define,
$J'_+=\{i\in J'\mid \mbox{$i-r$ is even}\}$,
$J'_-=\{i\in J'\mid \mbox{$i-r$ is odd}\}$.
We introduce the notations $[i,j]'=\alpha'_i+\cdots
+\alpha'_j$ ($1\leq i< j\leq 2r+1$) and  $[i]'=\alpha'_i$
as parallel to \eqref{eq:Cij}.
Let
 $O'_i=\{ (\sigma')^k (-\alpha'_i)
\mid 1\leq k\leq r+1\}$
 be the orbit of $-\alpha'_i$
in $\Phi'_+$ by  $\sigma'=\sigma'_-
\sigma'_+$,
$\sigma'_{\pm}=\prod_{i\in J'_{\pm}} \sigma'_i$,
where $\sigma'_i$ is the piecewise-linear analogue
of the simple reflection $s'_i$ as \eqref{eq:pl}.

\begin{lemma}
\label{lem:Ccox}
Let
\begin{align}
\rho:\Phi_+\rightarrow
\bigsqcup_{i=1}^rO'_i
\end{align}
be the map defined by
\begin{align}
\begin{split}
[i,j]&\mapsto
\begin{cases}
[i,j]'& \mbox{$j-r$: odd}\\
[2r+2-j,2r+2-i]'& \mbox{$j-r$: even},\\
\end{cases}\\
\{i,j\}&\mapsto
\begin{cases}
[i,2r+2-j]'& \mbox{$j-r$: odd}\\
[j,2r+2-i]'& \mbox{$j-r$: even},\\
\end{cases}\\
\{ r+1 \} &\mapsto [r,r+1]',
\end{split}
\end{align}
where  $[i]=[i,i]$.
Then, $\rho$ is a bijection.
Furthermore, under the bijection $\rho$,
the action of $\sigma$ is translated into the
one of the square of the Coxeter element
$s'=s'_- s'_+$ of type $A_{2r+1}$
acting on $\Phi'_+$,
where $s'_{\pm}=\prod_{i\in J'_{\pm}} s'_i$.

\end{lemma}

For $-h^{\vee}\leq u< 0$, define
\begin{align}
\label{eq:Calpha}
&\alpha_{i}(u)=
\begin{cases}
\sigma^{-u/2}(-\alpha_i)
& \mbox{\rm $i\in J_+$, $u\equiv 0$,}\\
\sigma^{-(u-1)/2}(\alpha_i)
& \mbox{\rm $i\in J_+$, $u\equiv -1$,}\\
\sigma^{-(2u-1)/4}(-\alpha_i)
& \mbox{\rm $i\in J_-$, $u\equiv -\frac{3}{2}$,}\\
\sigma^{-(2u+1)/4}(\alpha_i)
& \mbox{\rm $i\in J_-$, $u\equiv -\frac{1}{2}$,}\\
\sigma^{-u/2}(-\alpha_r)
& \mbox{\rm $i=r$, $u\equiv 0$,}\\
\sigma^{-(u-1)/2}(-\alpha_{r+1})
& \mbox{\rm $i=r+1$, $u\equiv -1$,}\\
\end{cases}
\end{align}
where $\equiv$ is  modulo $2\mathbb{Z}$.
Note that they correspond to the positive roots
in Tables \ref{tab:orbit10} and \ref{tab:orbit9}
with $u$ being the parameter in the head lines.
By Lemma \ref{lem:Corbit} they are all the positive roots
of $D_{r+1}$.

\begin{lemma}
\label{lem:Ctrec}
The family in \eqref{eq:Calpha} satisfies the recurrence relations
\begin{align}
\label{eq:Calpha1}
\begin{split}
\textstyle
\alpha_{i}(u-\frac{1}{2})+\alpha_{i}(u+\frac{1}{2})
&=\alpha_{i-1}(u)+\alpha_{i+1}(u)
\quad (1\leq i\leq r-2),\\
\textstyle
\alpha_{r-1}(u-\frac{1}{2})+\alpha_{r-1}(u+\frac{1}{2})
&=
\begin{cases}
\alpha_{r-2}(u)+\alpha_{r}(u) &\mbox{\rm ($u$: even)}\\
\alpha_{r-2}(u)+\alpha_{r+1}(u) &\mbox{\rm ($u$: odd)},\\
\end{cases}\\
\alpha_{r}(u-1)+\alpha_{r}(u+1)
&
\textstyle
=\alpha_{r-1}(u-\frac{1}{2})+\alpha_{r-1}(u+\frac{1}{2})
 \quad \mbox{\rm ($u$: odd)},\\
\alpha_{r+1}(u-1)+\alpha_{r+1}(u+1)
&
\textstyle
=\alpha_{r-1}(u-\frac{1}{2})+\alpha_{r-1}(u+\frac{1}{2})
 \quad \mbox{\rm ($u$: even)},\\
\end{split}
\end{align}
where $\alpha_0(u)=0$.
\end{lemma}
\begin{proof}
These relations are easily verified by the 
explicit expressions of $\alpha_i(u)$.
See Tables \ref{tab:orbit10} and \ref{tab:orbit9}.
The first two relations are
also obtained from Lemma \ref{lem:Ccox} and \cite[Eq.~(10.9)]{FZ2}.
\end{proof}

Let us return to
 prove  (ii) of Proposition \ref{prop:Clev2}
for the $D$ part.
For a monomial $m$ in $y=(y_{\mathbf{i}})_{\mathbf{i}\in
\mathbf{I}}$,
let $\pi_D(m)$ denote the specialization
with $y_{i,1}=y_{i,3}=1$ ($i\leq r-1$).
For simplicity, we set
$y_{i2}=y_i$ ($i\leq  r-1$), $y_{r1}=y_r$,
$y_{r+1,1}=y_{r+1}$,
and also,
$y_{i2}(u)=y_{i}(u)$
($i\leq r-1$), $y_{r1}(u)=y_r(u)$, $y_{r+1,1}(u)=y_r(u)$.
We define the vectors $\mathbf{t}_{i}(u)
=(t_{i}(u)_k)_{k=1}^{r+1}$
by
\begin{align}
\pi_D([y_{i}(u)]_{\mathbf{T}})
=
\prod_{k=1}^{r+1}
 y_{k}^{t_{i}(u)_k}.
\end{align}
We also identify each vector $\mathbf{t}_{i}(u)$
with $\alpha=\sum_{k=1}^{r+1}
t_{i}(u)_k \alpha_k \in \mathbb{Z}\Pi$.

\begin{proposition}
\label{prop:Ctvec}
Let $-h^{\vee}\leq u< 0$.
Then, we have
\begin{align}
\label{eq:Ctvec1}
\mathbf{t}_{i}(u)=-\alpha_i(u)
\end{align}
for $(i,u)$ in \eqref{eq:Calpha},
and 
\begin{align}
\label{eq:CpiA}
&\pi_D([y_{i1}(u)]_{\mathbf{T}})
=\pi_D([y_{i3}(u)]_{\mathbf{T}})=1
\quad 
\mbox{\rm ($i\leq r-1$, $r+i+2u$: even)}.
\end{align}
\end{proposition}
Note that these formulas
determine
$\pi_D([y_{\mathbf{i}}(u)]_{\mathbf{T}})$
 for any $(\mathbf{i},u):\mathbf{p}_+$. 

\begin{proof} 
We can verify the claim
for $-2\leq u \leq -\frac{1}{2}$ by direct
computation.
Then, by induction on $u$ in the backward direction,
one can  establish  the claim,
 together with
the recurrence relations among
$\mathbf{t}_{i}(u)$'s with $(i,u)$ in \eqref{eq:Calpha},
\begin{align}
\label{eq:Ctrec}
\begin{split}
\textstyle
\mathbf{t}_{i}(u-\frac{1}{2})+\mathbf{t}_{i}(u+\frac{1}{2})
&=\mathbf{t}_{i-1}(u)+\mathbf{t}_{i+1}(u)
\quad (1\leq i\leq r-2),\\
\textstyle
\mathbf{t}_{r-1}(u-\frac{1}{2})+\mathbf{t}_{r-1}(u+\frac{1}{2})
&=
\begin{cases}
\mathbf{t}_{r-2}(u)+\mathbf{t}_{r}(u) &\mbox{\rm ($u$: even)}\\
\mathbf{t}_{r-2}(u)+\mathbf{t}_{r+1}(u) &\mbox{\rm ($u$: odd)},\\
\end{cases}\\
\mathbf{t}_{r}(u-1)+\mathbf{t}_{r}(u+1)
&
\textstyle
=\mathbf{t}_{r-1}(u-\frac{1}{2})+\mathbf{t}_{r-1}(u+\frac{1}{2})
 \quad \mbox{\rm ($u$: odd)},\\
\mathbf{t}_{r+1}(u-1)+\mathbf{t}_{r+1}(u+1)
&
\textstyle
=\mathbf{t}_{r-1}(u-\frac{1}{2})+\mathbf{t}_{r-1}(u+\frac{1}{2})
 \quad \mbox{\rm ($u$: even)}.
\end{split}
\end{align}
Note that \eqref{eq:Ctrec} coincides with
\eqref{eq:Calpha1} under \eqref{eq:Ctvec1}.
To derive \eqref{eq:Ctrec}, one uses
 the mutations as in \cite[Figure 6]{IIKKN}
(or  the tropical version of the Y-system
$\mathbb{Y}_2(C_r)$ directly)
 and the positivity/negativity
of $\pi_D([y_{\mathbf{i}}(u)]_{\mathbf{T}})$
resulting from \eqref{eq:Ctvec1} and \eqref{eq:CpiA}
by induction hypothesis.
\end{proof}

Now (ii) and (iv)  in Proposition \ref{prop:Clev2}
for the $D$ part
 follow from  Proposition \ref{prop:Ctvec}.

\subsubsection{A part} The $A$ part can be
studied in a similar way to the $D$ part.
Thus, we  present only the result.

First we note that the quiver $Q_2(C_r)$
is symmetric under the exchange
$y_{i,1}\leftrightarrow y_{i,3}$ ($i\leq r-1$).
Thus, one can concentrate on
the powers of $[y_{\mathbf{i}}(u)]_{\mathbf{T}}$
in  the variables
$y_{i,1}$ ($i\leq r-1$).

Let $A_{r-1}$ be the Dynkin diagram of type $A$
with index set $J=\{1,\dots,r-1\}$.
We assign 
the sign +/$-$ to vertices
(except  for $r$)
 of $A_{r-1}$ as inherited from
$Q_2(C_r)$.
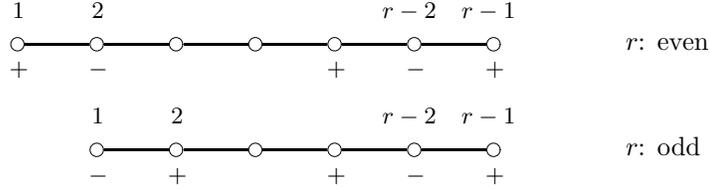
\begin{align*}
\begin{picture}(280,68)(30,20)
\put(45,70)
{
\put(0,0){\circle{5}}
\put(30,0){\circle{5}}
\put(60,0){\circle{5}}
\put(90,0){\circle{5}}
\put(120,0){\circle{5}}
\put(150,0){\circle{5}}
\put(180,0){\circle{5}}
\put(27,0){\line(-1,0){24}}
\put(57,0){\line(-1,0){24}}
\put(87,0){\line(-1,0){24}}
\put(93,0){\line(1,0){24}}
\put(147,0){\line(-1,0){24}}
\put(177,0){\line(-1,0){24}}
\put(-4,8)
{
\put(2,2){\small $1$}
\put(32,2){\small $2$}
\put(62,2){\small }
\put(122,2){\small }
\put(142,2){\small $r-2$}
\put(172,2){\small $r-1$}
}
\put(-5,-14)
{
\put(2,2){\small $+$}
\put(32,2){\small $-$}
\put(122,2){\small $+$}
\put(152,2){\small $-$}
\put(182,2){\small $+$}
}
\put(230,-2){$r$: even}
}
%
%
\put(45,30)
{
\put(30,0){\circle{5}}
\put(60,0){\circle{5}}
\put(90,0){\circle{5}}
\put(120,0){\circle{5}}
\put(150,0){\circle{5}}
\put(180,0){\circle{5}}
\put(57,0){\line(-1,0){24}}
\put(87,0){\line(-1,0){24}}
\put(93,0){\line(1,0){24}}
\put(147,0){\line(-1,0){24}}
\put(177,0){\line(-1,0){24}}
\put(-4,8)
{
\put(32,2){\small $1$}
\put(62,2){\small $2$}
\put(122,2){\small }
\put(142,2){\small $r-2$}
\put(172,2){\small $r-1$}
}
\put(-5,-14)
{
\put(32,2){\small $-$}
\put(62,2){\small $+$}
\put(122,2){\small $+$}
\put(152,2){\small $-$}
\put(182,2){\small $+$}
}
\put(230,-2){$r$: odd}
}
\end{picture}
\end{align*}

Let $\Pi=\{ \alpha_1,\dots,\alpha_{r-1}\}$, $-\Pi$,
 $\Phi_+$ be the set of the simple roots,
the negative simple roots, the positive roots, respectively,
of type $A_{r-1}$.
Again, we introduce the piecewise-linear
analogue $\sigma_i$ of the simple reflection $s_i$,
acting on $\Phi_{\geq -1}=\Phi_{+}\sqcup (-\Pi)$ as 
\eqref{eq:pl}.
Let
\begin{align}
\sigma_+=\prod_{i\in J_+} \sigma_i,\quad
\sigma_-=\prod_{i\in J_-} \sigma_i,
\end{align}
where $J_{\pm}$ is the set of the
vertices of $A_{r-1}$
 with property $\pm$. We define $\sigma$ as the composition
\begin{align}
\label{eq:CsigmaA}
\sigma=\sigma_- \sigma_+.
\end{align}

For a monomial $m$ in $y=(y_{\mathbf{i}})_{\mathbf{i}\in
\mathbf{I}}$,
let $\pi_A(m)$ denote the specialization
with $y_{i,2}=y_{i,3}=1$ ($i\leq r-1$) and $y_{r1}=y_{r+1,1}=1$.
We set
$y_{i1}=y_i$ ($i\leq  r-1$).

We define the vectors $\mathbf{t}_{\mathbf{i}}(u)
=(t_{\mathbf{i}}(u)_k)_{k=1}^{r-1}$
by
\begin{align}
\pi_A([y_{\mathbf{i}}(u)]_{\mathbf{T}})
=
\prod_{k=1}^{r-1}
 y_{k}^{t_{\mathbf{i}}(u)_k}.
\end{align}
We also identify each vector $\mathbf{t}_{\mathbf{i}}(u)$
with $\alpha=\sum_{k=1}^{r-1}
t_{\mathbf{i}}(u)_k \alpha_k \in \mathbb{Z}\Pi$.

With these notations,
 the result for the $A$ part is summarized as follows.
\begin{proposition}
\label{prop:CtvecA}
Let $-h^{\vee}\leq u< 0$.
For $(\mathbf{i},u):\mathbf{p}_+$,
$\mathbf{t}_{\mathbf{i}}(u)$ is given by,
for $i\leq r-1$,
\begin{align}
\label{eq:CtA1}
\begin{split}
\mathbf{t}_{i1}(u)&=
\begin{cases}
-\sigma^{-u}(-\alpha_i)& i\in J_+\\
-\sigma^{-(2u-1)/2}(-\alpha_i)& i\in J_-,\\
\end{cases}
\\
\mathbf{t}_{i2}(u)&=
\begin{cases}
-[
2r+2-i+ 2u
,r-1]& -\frac{1}{2}h^{\vee}\leq u <0 \\
-[-1-i- 2u 
,r-1]& -h^{\vee}\leq u < -\frac{1}{2}h^{\vee}, \\
\end{cases}\\
\mathbf{t}_{i3}(u)&=0,\\
\end{split}
\end{align}
and
\begin{align}
\label{eq:CtA2}
\begin{split}
\mathbf{t}_{r1}(u)&=
\begin{cases}
-[r+2+ 2u ,r-1]& -\frac{1}{2}h^{\vee}\leq u <0 \\
-[-1-r- 2u,r-1]& -h^{\vee}\leq u < -\frac{1}{2}h^{\vee}, \\
\end{cases}\\
\mathbf{t}_{r+1,1}(u)&=
\begin{cases}
-[r+2+ 2u,r-1]& -\frac{1}{2}h^{\vee}\leq u <0 \\
-[-1-r- 2u,r-1]& -h^{\vee}\leq u < -\frac{1}{2}h^{\vee}, \\
\end{cases}\\
\end{split}
\end{align}
where $[i,j]=\alpha_i+\cdots +\alpha_j$ if $i\leq j$
and $0$ if $i> j$.
\end{proposition}
Note that
$\mathbf{t}_{i1}(-\frac{h^{\vee}}{2})=\alpha_{r-i}$
($i\in J_-$ for even $r$ and $i\in J_+$ for odd $r$)
and 
$\mathbf{t}_{i1}(-\frac{h^{\vee}}{2}
-\frac{1}{2})=\alpha_{r-i}$
($i\in J_+$ for even $r$ and $i\in J_-$ for odd $r$),
and that they are the only positive monomials
in \eqref{eq:CtA1} and \eqref{eq:CtA2}.
Now (ii) and (iv)  in Proposition \ref{prop:Clev2}
for the $A$ part
 follow from  Proposition \ref{prop:CtvecA}.

This completes the proof of Proposition \ref{prop:Clev2}.

\subsection{Tropical Y-systems at higher levels}

By the same method for the $B_r$ case \cite[Proposition 4.1]{IIKKN},
one can establish the `factorization property' of the
tropical Y-system at higher levels.
As a result, we obtain a generalization of Proposition
\ref{prop:Clev2}.

\begin{proposition}
\label{prop:Clevh}
 For 
$[\mathcal{G}_Y(B,y)]_{\mathbf{T}}$
with $B=B_{\ell}(C_r)$, the following facts hold.
\par
(i) Let $u$ be in the region $0\le u < \ell$.
For any $(\mathbf{i},u):\mathbf{p}_+$,
the  monomial $[y_{\mathbf{i}}(u)]_{\mathbf{T}}$
is positive.
\par
(ii) Let $u$ be in the region $-h^{\vee}\le u < 0$.
\begin{itemize}
\item[\em (a)]
 Let $\mathbf{i}\in \mathbf{I}^{\circ}$
or $\mathbf{i}=(i,i')$ $(i\leq r-1, i'\in 2\mathbb{N})$.
For any $(\mathbf{i},u):\mathbf{p}_+$,
the  monomial $[y_{\mathbf{i}}(u)]_{\mathbf{T}}$
is negative.
\item[\em (b)]
Let $\mathbf{i}=(i,i')$  $(i\leq r-1, i'\not\in
2\mathbb{N})$.
For any $(\mathbf{i},u):\mathbf{p}_+$,
the  monomial $[y_{\mathbf{i}}(u)]_{\mathbf{T}}$
is positive for $u=-\frac{1}{2}h^{\vee},
-\frac{1}{2}h^{\vee}-\frac{1}{2}$
and negative otherwise.
\end{itemize}
\par
(iii)
$y_{ii'}(\ell)=y_{i,2\ell-i'}^{-1}$ if $i\leq r-1$ and
$y_{i,\ell-i'}^{-1}$ if $i=r,r+1$.
\par
(iv) For even $r$,  $y_{ii'}(-h^{\vee})=
y_{ii'}^{-1}$ if $i\leq r-1$ and 
$y_{2r+1-i,i'}^{-1}$ if  $i=r, r+1$.
 For odd $r$,  $y_{ii'}(-h^{\vee})=
y_{ii'}^{-1}$.
\end{proposition}

We obtain two important corollaries of
Propositions \ref{prop:Clev2} and  \ref{prop:Clevh}.

\begin{theorem}
\label{thm:CtYperiod}
For $[\mathcal{G}_Y(B,y)]_{\mathbf{T}}$
the following relations hold.
\par
(i) Half periodicity: 
$[y_{\mathbf{i}}(u+h^{\vee}+\ell)]_{\mathbf{T}}
=[y_{\boldsymbol{\omega}(\mathbf{i})}(u)]_{\mathbf{T}}$.
\par
(ii) 
 Full periodicity: 
$[y_{\mathbf{i}}(u+2(h^{\vee}+\ell))]_{\mathbf{T}}
=[y_{\mathbf{i}}(u)]_{\mathbf{T}}$.
\end{theorem}

\begin{theorem}
\label{thm:Clevhd}
For $[\mathcal{G}_Y(B,y)]_{\mathbf{T}}$,
let $N_+$ and $N_-$ denote the
total numbers of the positive and negative monomials,
respectively,
among $[y_{\mathbf{i}}(u)]_{\mathbf{T}}$
for $(\mathbf{i},u):\mathbf{p}_+$
in the region $0\leq u < 2(h^{\vee}+\ell)$.
Then, we have
\begin{align}
N_+=2\ell(2r\ell-\ell-1),
\quad
N_-=2r(2\ell r-r-1).
\end{align}
\end{theorem}
We observe  the symmetry (the {\em level-rank duality})
 for the numbers $N_+$ and $N_-$
under the exchange of $r$ and $\ell$.

\subsection{Periodicities and dilogarithm identities}

Applying \cite[Theorem 5.1]{IIKKN} to
Theorem  \ref{thm:CtYperiod},
we obtain the periodicities:

\begin{theorem}
\label{thm:Cxperiod}
For $\mathcal{A}(B,x,y)$,
the following relations hold.
\par
(i) Half periodicity: 
$x_{\mathbf{i}}(u+h^{\vee}+\ell)
=x_{\boldsymbol{\omega}(\mathbf{i})}(u)
$.
\par
(ii) 
 Full periodicity: 
$x_{\mathbf{i}}(u+2(h^{\vee}+\ell))
=x_{\mathbf{i}}(u)
$.
\end{theorem}

\begin{theorem}
\label{thm:Cyperiod}
For $\mathcal{G}(B,y)$,
the following relations hold.
\par
(i) Half periodicity: 
$y_{\mathbf{i}}(u+h^{\vee}+\ell)
=y_{\boldsymbol{\omega}(\mathbf{i})}(u)
$.
\par
(ii) 
 Full periodicity: 
$y_{\mathbf{i}}(u+2(h^{\vee}+\ell))
=y_{\mathbf{i}}(u)
$.
\end{theorem}

Then,
 Theorems \ref{thm:Tperiod} and \ref{thm:Yperiod}
 for $C_r$ follow from
Theorems
\ref{thm:CTiso}, \ref{thm:CYiso},
\ref{thm:Cxperiod}, and \ref{thm:Cyperiod}.
Furthermore, Theorem \ref{thm:DI2}  for $C_r$ is obtained from
the above periodicities  and Theorem \ref{thm:Clevhd} as
in the $B_r$ case \cite[Section 6]{IIKKN}.

\section{Type $F_4$}

The $F_4$ case is quite parallel to the $B_r$ and $C_r$ case.
We do not repeat the same definitions
unless otherwise mentioned.
Again, the properties of the tropical Y-system at level 2
(Proposition \ref{prop:Flev2}) are crucial and
specific to $F_4$.


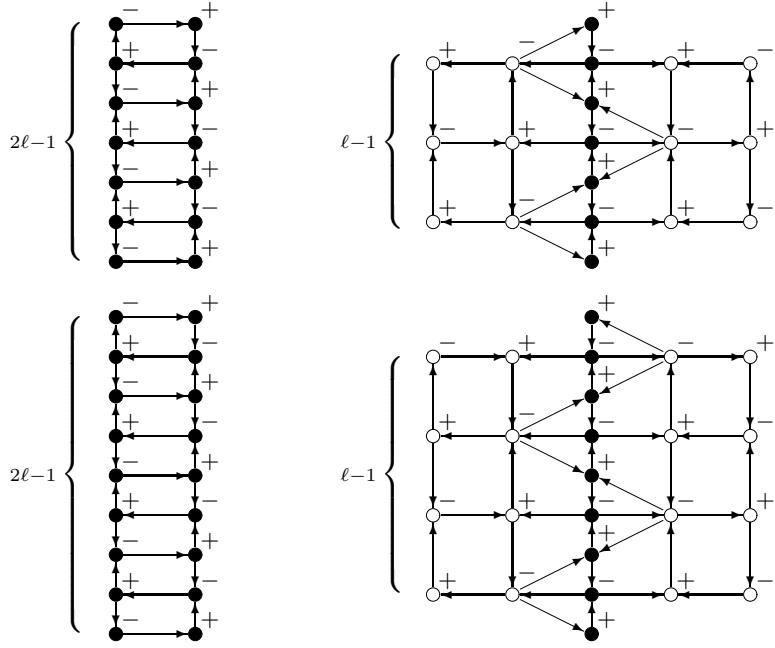
\begin{figure}
\begin{picture}(285,100)(-65,-20)
\put(0,-15)
{
\put(-30,0){\circle*{5}}
\put(0,0){\circle*{5}}
\put(-27,0){\vector(1,0){24}}
}
\put(0,0)
{
\put(-30,0){\circle*{5}}
\put(0,0){\circle*{5}}
\put(90,0){\circle{5}}
\put(120,0){\circle{5}}
\put(150,0){\circle*{5}}
\put(150,15){\circle*{5}}
\put(180,0){\circle{5}}
\put(210,0){\circle{5}}
\put(-3,0){\vector(-1,0){24}}
\put(117,0){\vector(-1,0){24}}
\put(147,0){\vector(-1,0){24}}
\put(153,0){\vector(1,0){24}}
\put(207,0){\vector(-1,0){24}}
}
\put(0,15)
{
\put(-30,0){\circle*{5}}
\put(0,0){\circle*{5}}
\put(-27,0){\vector(1,0){24}}
}
\put(0,30)
{
\put(-30,0){\circle*{5}}
\put(0,0){\circle*{5}}
\put(90,0){\circle{5}}
\put(120,0){\circle{5}}
\put(150,0){\circle*{5}}
\put(150,15){\circle*{5}}
\put(180,0){\circle{5}}
\put(210,0){\circle{5}}
\put(-3,0){\vector(-1,0){24}}
\put(93,0){\vector(1,0){24}}
\put(147,0){\vector(-1,0){24}}
\put(153,0){\vector(1,0){24}}
\put(183,0){\vector(1,0){24}}
}
\put(0,45)
{
\put(-30,0){\circle*{5}}
\put(0,0){\circle*{5}}
\put(-27,0){\vector(1,0){24}}
}
\put(0,60)
{
\put(-30,0){\circle*{5}}
\put(0,0){\circle*{5}}
\put(90,0){\circle{5}}
\put(120,0){\circle{5}}
\put(150,0){\circle*{5}}
\put(180,0){\circle{5}}
\put(210,0){\circle{5}}
\put(150,15){\circle*{5}}
\put(-3,0){\vector(-1,0){24}}
\put(117,0){\vector(-1,0){24}}
\put(147,0){\vector(-1,0){24}}
\put(153,0){\vector(1,0){24}}
\put(207,0){\vector(-1,0){24}}
}
\put(0,75)
{
\put(-30,0){\circle*{5}}
\put(0,0){\circle*{5}}
\put(-27,0){\vector(1,0){24}}
}
\put(150,-15){\circle*{5}}
%
\put(-30,0)
{
\put(0,-3){\vector(0,-1){9}}
\put(0,3){\vector(0,1){9}}
\put(0,27){\vector(0,-1){9}}
\put(0,33){\vector(0,1){9}}
\put(0,57){\vector(0,-1){9}}
\put(0,63){\vector(0,1){9}}
}
\put(0,0)
{
\put(0,-12){\vector(0,1){9}}
\put(0,12){\vector(0,-1){9}}
\put(0,18){\vector(0,1){9}}
\put(0,42){\vector(0,-1){9}}
\put(0,48){\vector(0,1){9}}
\put(0,72){\vector(0,-1){9}}
}
\put(150,0)
{
\put(0,-12){\vector(0,1){9}}
\put(0,12){\vector(0,-1){9}}
\put(0,18){\vector(0,1){9}}
\put(0,42){\vector(0,-1){9}}
\put(0,48){\vector(0,1){9}}
\put(0,72){\vector(0,-1){9}}
}
\put(90,3){\vector(0,1){24}}
\put(90,57){\vector(0,-1){24}}
\put(120,27){\vector(0,-1){24}}
\put(120,33){\vector(0,1){24}}
\put(180,3){\vector(0,1){24}}
\put(180,57){\vector(0,-1){24}}
\put(210,27){\vector(0,-1){24}}
\put(210,33){\vector(0,1){24}}
\put(123,-2){\vector(2,-1){24}}
\put(123,2){\vector(2,1){24}}
\put(123,58){\vector(2,-1){24}}
\put(123,62){\vector(2,1){24}}
\put(177,28){\vector(-2,-1){24}}
\put(177,32){\vector(-2,1){24}}
\put(0,-14)
{
\put(-28,2){\small $ -$}
\put(2,2){\small $+$}
}
\put(0,1)
{
\put(-28,2){\small $ +$}
\put(2,2){\small $-$}
}
\put(0,16)
{
\put(-28,2){\small $ -$}
\put(2,2){\small $+$}
}
\put(0,31)
{
\put(-28,2){\small $ +$}
\put(2,2){\small $-$}
}
\put(0,46)
{
\put(-28,2){\small $ -$}
\put(2,2){\small $+$}
}
\put(0,61)
{
\put(-28,2){\small $ +$}
\put(2,2){\small $-$}
}
\put(0,76)
{
\put(-28,2){\small $ -$}
\put(2,2){\small $+$}
}
\put(0,1)
{
\put(92,2){\small $+$}
\put(92,32){\small $-$}
\put(92,62){\small $+$}
\put(122,5){\small $-$}
\put(122,32){\small $+$}
\put(122,65){\small $-$}
\put(152,-13){\small $+$}
\put(152,2){\small $-$}
\put(152,20){\small $+$}
\put(152,32){\small $-$}
\put(152,47){\small $+$}
\put(152,62){\small $-$}
\put(152,77){\small $+$}
\put(182,2){\small $+$}
\put(182,32){\small $-$}
\put(182,62){\small $+$}
\put(212,2){\small $-$}
\put(212,32){\small $+$}
\put(212,62){\small $-$}
}
\put(55,28){${\scriptstyle\ell -1}\left\{ \makebox(0,37){}\right.$}
\put(-70,28){${\scriptstyle2\ell -1}\left\{ \makebox(0,50){}\right.$}
\end{picture}
%
%
\begin{picture}(285,140)(-65,-20)
\put(0,-15)
{
\put(-30,0){\circle*{5}}
\put(0,0){\circle*{5}}
\put(-27,0){\vector(1,0){24}}
}
\put(0,0)
{
\put(-30,0){\circle*{5}}
\put(0,0){\circle*{5}}
\put(90,0){\circle{5}}
\put(120,0){\circle{5}}
\put(150,0){\circle*{5}}
\put(150,15){\circle*{5}}
\put(180,0){\circle{5}}
\put(210,0){\circle{5}}
\put(-3,0){\vector(-1,0){24}}
\put(117,0){\vector(-1,0){24}}
\put(147,0){\vector(-1,0){24}}
\put(153,0){\vector(1,0){24}}
\put(207,0){\vector(-1,0){24}}
}
\put(0,15)
{
\put(-30,0){\circle*{5}}
\put(0,0){\circle*{5}}
\put(-27,0){\vector(1,0){24}}
}
\put(0,30)
{
\put(-30,0){\circle*{5}}
\put(0,0){\circle*{5}}
\put(90,0){\circle{5}}
\put(120,0){\circle{5}}
\put(150,0){\circle*{5}}
\put(150,15){\circle*{5}}
\put(180,0){\circle{5}}
\put(210,0){\circle{5}}
\put(-3,0){\vector(-1,0){24}}
\put(93,0){\vector(1,0){24}}
\put(147,0){\vector(-1,0){24}}
\put(153,0){\vector(1,0){24}}
\put(183,0){\vector(1,0){24}}
}
\put(0,45)
{
\put(-30,0){\circle*{5}}
\put(0,0){\circle*{5}}
\put(-27,0){\vector(1,0){24}}
}
\put(0,60)
{
\put(-30,0){\circle*{5}}
\put(0,0){\circle*{5}}
\put(90,0){\circle{5}}
\put(120,0){\circle{5}}
\put(150,0){\circle*{5}}
\put(180,0){\circle{5}}
\put(150,15){\circle*{5}}
\put(210,0){\circle{5}}
\put(-3,0){\vector(-1,0){24}}
\put(117,0){\vector(-1,0){24}}
\put(147,0){\vector(-1,0){24}}
\put(153,0){\vector(1,0){24}}
\put(207,0){\vector(-1,0){24}}
}
\put(0,75)
{
\put(-30,0){\circle*{5}}
\put(0,0){\circle*{5}}
\put(-27,0){\vector(1,0){24}}
}
\put(0,90)
{
\put(-30,0){\circle*{5}}
\put(0,0){\circle*{5}}
\put(90,0){\circle{5}}
\put(120,0){\circle{5}}
\put(150,0){\circle*{5}}
\put(180,0){\circle{5}}
\put(210,0){\circle{5}}
\put(150,15){\circle*{5}}
\put(-3,0){\vector(-1,0){24}}
\put(93,0){\vector(1,0){24}}
\put(147,0){\vector(-1,0){24}}
\put(153,0){\vector(1,0){24}}
\put(183,0){\vector(1,0){24}}
}
\put(0,105)
{
\put(-30,0){\circle*{5}}
\put(0,0){\circle*{5}}
\put(-27,0){\vector(1,0){24}}
}
\put(150,-15){\circle*{5}}
%
\put(-30,0)
{
\put(0,-3){\vector(0,-1){9}}
\put(0,3){\vector(0,1){9}}
\put(0,27){\vector(0,-1){9}}
\put(0,33){\vector(0,1){9}}
\put(0,57){\vector(0,-1){9}}
\put(0,63){\vector(0,1){9}}
\put(0,87){\vector(0,-1){9}}
\put(0,93){\vector(0,1){9}}
}
\put(0,0)
{
\put(0,-12){\vector(0,1){9}}
\put(0,12){\vector(0,-1){9}}
\put(0,18){\vector(0,1){9}}
\put(0,42){\vector(0,-1){9}}
\put(0,48){\vector(0,1){9}}
\put(0,72){\vector(0,-1){9}}
\put(0,78){\vector(0,1){9}}
\put(0,102){\vector(0,-1){9}}
}
\put(150,0)
{
\put(0,-12){\vector(0,1){9}}
\put(0,12){\vector(0,-1){9}}
\put(0,18){\vector(0,1){9}}
\put(0,42){\vector(0,-1){9}}
\put(0,48){\vector(0,1){9}}
\put(0,72){\vector(0,-1){9}}
\put(0,78){\vector(0,1){9}}
\put(0,102){\vector(0,-1){9}}
}
\put(90,3){\vector(0,1){24}}
\put(90,57){\vector(0,-1){24}}
\put(90,63){\vector(0,1){24}}
\put(120,27){\vector(0,-1){24}}
\put(120,33){\vector(0,1){24}}
\put(120,87){\vector(0,-1){24}}
\put(180,3){\vector(0,1){24}}
\put(180,57){\vector(0,-1){24}}
\put(180,63){\vector(0,1){24}}
\put(210,27){\vector(0,-1){24}}
\put(210,33){\vector(0,1){24}}
\put(210,87){\vector(0,-1){24}}
\put(123,-2){\vector(2,-1){24}}
\put(123,2){\vector(2,1){24}}
\put(123,58){\vector(2,-1){24}}
\put(123,62){\vector(2,1){24}}
\put(177,28){\vector(-2,-1){24}}
\put(177,32){\vector(-2,1){24}}
\put(177,88){\vector(-2,-1){24}}
\put(177,92){\vector(-2,1){24}}
\put(0,-14)
{
\put(-28,2){\small $ -$}
\put(2,2){\small $+$}
}
\put(0,1)
{
\put(-28,2){\small $ +$}
\put(2,2){\small $-$}
}
\put(0,16)
{
\put(-28,2){\small $ -$}
\put(2,2){\small $+$}
}
\put(0,31)
{
\put(-28,2){\small $ +$}
\put(2,2){\small $-$}
}
\put(0,46)
{
\put(-28,2){\small $ -$}
\put(2,2){\small $+$}
}
\put(0,61)
{
\put(-28,2){\small $ +$}
\put(2,2){\small $-$}
}
\put(0,76)
{
\put(-28,2){\small $ -$}
\put(2,2){\small $+$}
}
\put(0,91)
{
\put(-28,2){\small $ +$}
\put(2,2){\small $-$}
}
\put(0,106)
{
\put(-28,2){\small $ -$}
\put(2,2){\small $+$}
}
\put(0,1)
{
\put(92,2){\small $+$}
\put(92,32){\small $-$}
\put(92,62){\small $+$}
\put(92,92){\small $-$}
\put(122,5){\small $-$}
\put(122,32){\small $+$}
\put(122,65){\small $-$}
\put(122,92){\small $+$}
\put(152,-13){\small $+$}
\put(152,2){\small $-$}
\put(152,20){\small $+$}
\put(152,32){\small $-$}
\put(152,47){\small $+$}
\put(152,62){\small $-$}
\put(152,80){\small $+$}
\put(152,92){\small $-$}
\put(152,107){\small $+$}
\put(182,2){\small $+$}
\put(182,32){\small $-$}
\put(182,62){\small $+$}
\put(182,92){\small $-$}
\put(212,2){\small $-$}
\put(212,32){\small $+$}
\put(212,62){\small $-$}
\put(212,92){\small $+$}
}
\put(55,43){${\scriptstyle\ell -1}\left\{ \makebox(0,50){}\right.$}
\put(-70,43){${\scriptstyle2\ell -1}\left\{ \makebox(0,65){}\right.$}
\end{picture}
\caption{The quiver $Q_{\ell}(F_4)$  for even $\ell$
(upper) and for odd $\ell$ (lower),
where we identify the
right column in the left quiver
with the middle column in the right quiver.}
\label{fig:quiverF}
\end{figure}

\subsection{Parity decompositions of T and Y-systems}

For a triplet $(a,m,u)\in \mathcal{I}_{\ell}$,
we reset the `parity conditions' $\mathbf{P}_{+}$ and
$\mathbf{P}_{-}$ by
\begin{align}
\label{eq:FPcond}
\mathbf{P}_{+}:& \ \mbox{
$2u$ is even if
$a=1,2$; $a+m+2u$ is odd if $a=3,4$},\\
\mathbf{P}_{-}:& \ \mbox{
$2u$ is odd if
$a=1,2$; $a+m+2u$ is even if $a=3,4$}.
\end{align}
Then, we have
$\EuScript{T}^{\circ}_{\ell}(F_4)_+
\simeq
\EuScript{T}^{\circ}_{\ell}(F_4)_-
$
by $T^{(a)}_m(u)\mapsto T^{(a)}_m(u+\frac{1}{2})$ and
\begin{align}
\EuScript{T}^{\circ}_{\ell}(F_4)
\simeq
\EuScript{T}^{\circ}_{\ell}(F_4)_+
\otimes_{\mathbb{Z}}
\EuScript{T}^{\circ}_{\ell}(F_4)_-.
\end{align}

For a triplet $(a,m,u)\in \mathcal{I}_{\ell}$ ,
we reset the parity condition $\mathbf{P}'_{+}$ and
$\mathbf{P}'_{-}$ by
\begin{align}
\label{eq:FPcond2}
\mathbf{P}'_{+}:&\ \mbox{
$2u$ is even if
$a=1,2$;
 $a+m+2u$ is even if $a=3,4$},\\
\mathbf{P}'_{-}:&\ \mbox{
$2u$ is odd if
$a=1,2$;
 $a+m+2u$ is odd if $a=3,4$}.
\end{align}
We have
\begin{align}
(a,m,u):\mathbf{P}'_+
\ \Longleftrightarrow \
\textstyle (a,m,u\pm \frac{1}{t_a}):\mathbf{P}_+.
\end{align}
Also, we have
$\EuScript{Y}^{\circ}_{\ell}(F_4)_+
\simeq
\EuScript{Y}^{\circ}_{\ell}(F_4)_-
$
by $Y^{(a)}_m(u)\mapsto Y^{(a)}_m(u+\frac{1}{2})$,
$1+Y^{(a)}_m(u)\mapsto 1+Y^{(a)}_m(u+\frac{1}{2})$,
 and
\begin{align}
\EuScript{Y}^{\circ}_{\ell}(F_4)
\simeq
\EuScript{Y}^{\circ}_{\ell}(F_4)_+
\times
\EuScript{Y}^{\circ}_{\ell}(F_4)_-.
\end{align}

\subsection{Quiver $Q_{\ell}(F_4)$}

With type $F_4$ and $\ell\geq 2$ we associate
 the quiver $Q_{\ell}(F_4)$
by Figure
\ref{fig:quiverF},
where the right column in the left quiver
and the middle column in the right quiver
are identified.
Also, we assign the empty or filled circle $\circ$/$\bullet$ and
the sign +/$-$ to each vertex.

Let us choose  the index set $\mathbf{I}$
of the vertices of $Q_{\ell}(F_4)$
so that $\mathbf{i}=(i,i')\in \mathbf{I}$ represents
the vertex 
at the $i'$th row (from the bottom)
of the $i$th column (from the left)
in the right quiver for $i=1,2,3$,
the one of the $i-1$th column in the right quiver
for $i=5,6$,
and the one of the left column in the left quiver for $i=4$.
Thus, $i=1,\dots,6$, and $i'=1,\dots,\ell-1$ if $i=1,2,5,6$
and $i'=1,\dots,2\ell-1$ if $i=3,4$.

Let $\boldsymbol{r}$ be the involution acting on $\mathbf{I}$
by the left-right reflection of the right quiver.
Let $\boldsymbol{\omega}$ be the involution acting on $\mathbf{I}$
by the up-down reflection of the left quiver
and the $180^{\circ}$ rotation of the right quiver.

\begin{lemma}
\label{lem:FQmut}
Let $Q=Q_{\ell}(F_4)$.
\par
(i)
We have the same periodic sequence of mutations of quivers
as \eqref{eq:CB2}.
\par
(ii)
 $\boldsymbol{\omega}(Q)=Q$ if $h^{\vee}+\ell$ is even,
and  $\boldsymbol{\omega}(Q)=\boldsymbol{r}(Q)$ if $h^{\vee}+\ell$ is odd.
\end{lemma}


\begin{figure}
\setlength{\unitlength}{0.94pt}
\begin{picture}(370,370)(-150,-230)
\put(-150,46.5){$x(0)$}
\put(-19,-7.5)
{
\put(0,0){\framebox(38,15){\small $x^{(1)}_1(-1)$}}
\put(50,0){\makebox(38,15){\small $*$}}
\put(150,0){\framebox(38,15){\small $x^{(2)}_1(-1)$}}
\put(200,0){\makebox(38,15){\small $*$}}
\put(0,50){\makebox(38,15){\small $*$}}
\put(50,50){\framebox(38,15){\small $x^{(2)}_2(-1)$}}
\put(150,50){\makebox(38,15){\small $*$}}
\put(200,50){\framebox(38,15){\small $x^{(1)}_2(-1)$}}
\put(0,100){\framebox(38,15){\small $x^{(1)}_3(-1)$}}
\put(50,100){\makebox(38,15){\small $*$}}
\put(150,100){\framebox(38,15){\small $x^{(2)}_3(-1)$}}
\put(200,100){\makebox(38,15){\small $*$}}
\put(100,-25){\framebox(38,15){\small $\diamond$}}
\put(100,0){\dashbox{3}(38,15){\small $\diamond$}}
\put(100,25){\framebox(38,15){\small $\diamond$}}
\put(100,50){\dashbox{3}(38,15){\small $\diamond$}}
\put(100,75){\framebox(38,15){\small $\diamond$}}
\put(100,100){\dashbox{3}(38,15){\small $\diamond$}}
\put(100,125){\framebox(38,15){\small $\diamond$}}
}
%
\put(30,0){$\vector(-1,0){10}$}
\put(80,0){$\vector(-1,0){10}$}
\put(120,0){$\vector(1,0){10}$}
\put(180,0){$\vector(-1,0){10}$}
\put(20,50){$\vector(1,0){10}$}
\put(80,50){$\vector(-1,0){10}$}
\put(120,50){$\vector(1,0){10}$}
\put(170,50){$\vector(1,0){10}$}
\put(30,100){$\vector(-1,0){10}$}
\put(80,100){$\vector(-1,0){10}$}
\put(120,100){$\vector(1,0){10}$}
\put(180,100){$\vector(-1,0){10}$}
\put(0,18){\vector(0,1){14}}
\put(0,82){\vector(0,-1){14}}
\put(50,32){\vector(0,-1){14}}
\put(50,68){\vector(0,1){14}}
\put(100,-17){\vector(0,1){9}}
\put(100,17){\vector(0,-1){9}}
\put(100,33){\vector(0,1){9}}
\put(100,67){\vector(0,-1){9}}
\put(100,83){\vector(0,1){9}}
\put(100,117){\vector(0,-1){9}}
\put(150,18){\vector(0,1){14}}
\put(150,82){\vector(0,-1){14}}
\put(200,32){\vector(0,-1){14}}
\put(200,68){\vector(0,1){14}}
\put(70,-10){\vector(2,-1){10}}
\put(70,10){\vector(2,1){10}}
\put(70,90){\vector(2,-1){10}}
\put(70,110){\vector(2,1){10}}
\put(130,40){\vector(-2,-1){10}}
\put(130,60){\vector(-2,1){10}}

\put(-169,-7.5)
{
\put(50,-25){\dashbox{3}(38,15){\small $*$}}
\put(50,0){\framebox(38,15){\small $x^{(4)}_2(-\frac{1}{2})$}}
\put(50,25){\dashbox{3}(38,15){\small $*$}}
\put(50,50){\framebox(38,15){\small $x^{(4)}_4(-\frac{1}{2})$}}
\put(50,75){\dashbox{3}(38,15){\small $*$}}
\put(50,100){\framebox(38,15){\small $x^{(4)}_6(-\frac{1}{2})$}}
\put(50,125){\dashbox{3}(38,15){\small $*$}}
\put(100,-25){\framebox(38,15){\small $x^{(3)}_1(-\frac{1}{2})$}}
\put(100,0){\dashbox{3}(38,15){\small $*$}}
\put(100,25){\framebox(38,15){\small $x^{(3)}_3(-\frac{1}{2})$}}
\put(100,50){\dashbox{3}(38,15){\small $*$}}
\put(100,75){\framebox(38,15){\small $x^{(3)}_5(-\frac{1}{2})$}}
\put(100,100){\dashbox{3}(38,15){\small $*$}}
\put(100,125){\framebox(38,15){\small $x^{(3)}_7(-\frac{1}{2})$}}
}
\put(-50,-17){\vector(0,1){9}}
\put(-50,17){\vector(0,-1){9}}
\put(-50,33){\vector(0,1){9}}
\put(-50,67){\vector(0,-1){9}}
\put(-50,83){\vector(0,1){9}}
\put(-50,117){\vector(0,-1){9}}
\put(-100,-8){\vector(0,-1){9}}
\put(-100,8){\vector(0,1){9}}
\put(-100,42){\vector(0,-1){9}}
\put(-100,58){\vector(0,1){9}}
\put(-100,92){\vector(0,-1){9}}
\put(-100,108){\vector(0,1){9}}
\put(-80,-25){$\vector(1,0){10}$}
\put(-70,0){$\vector(-1,0){10}$}
\put(-80,25){$\vector(1,0){10}$}
\put(-70,50){$\vector(-1,0){10}$}
\put(-80,75){$\vector(1,0){10}$}
\put(-70,100){$\vector(-1,0){10}$}
\put(-80,125){$\vector(1,0){10}$}
\dottedline(-120,-40)(220,-40)
\put(-130,-42){$\updownarrow$}
\put(-155,-42){$\mu^{\bullet}_+\mu^{\circ}_+$}
\put(0,-180)
{
\put(-150,46.5){$x(\frac{1}{2})$}
\put(-19,-7.5)
{
\put(0,0){\dashbox{3}(38,15){\small $*$}}
\put(50,0){\makebox(38,15){\small $*$}}
\put(150,0){\dashbox{3}(38,15){\small $*$}}
\put(200,0){\makebox(38,15){\small $*$}}
\put(0,50){\makebox(38,15){\small $*$}}
\put(50,50){\dashbox{3}(38,15){\small $*$}}
\put(150,50){\makebox(38,15){\small $*$}}
\put(200,50){\dashbox{3}(38,15){\small $*$}}
\put(0,100){\dashbox{3}(38,15){\small $*$}}
\put(50,100){\makebox(38,15){\small $*$}}
\put(150,100){\dashbox{3}(38,15){\small $*$}}
\put(200,100){\makebox(38,15){\small $*$}}
\put(100,-25){\dashbox{3}(38,15){\small $\diamond$}}
\put(100,0){\framebox(38,15){\small $\diamond$}}
\put(100,25){\dashbox{3}(38,15){\small $\diamond$}}
\put(100,50){\framebox(38,15){\small $\diamond$}}
\put(100,75){\dashbox{3}(38,15){\small $\diamond$}}
\put(100,100){\framebox(38,15){\small $\diamond$}}
\put(100,125){\dashbox{3}(38,15){\small $\diamond$}}
}
%
\put(20,0){$\vector(1,0){10}$}
\put(70,0){$\vector(1,0){10}$}
\put(130,0){$\vector(-1,0){10}$}
\put(170,0){$\vector(1,0){10}$}
\put(30,50){$\vector(-1,0){10}$}
\put(70,50){$\vector(1,0){10}$}
\put(130,50){$\vector(-1,0){10}$}
\put(180,50){$\vector(-1,0){10}$}
\put(20,100){$\vector(1,0){10}$}
\put(70,100){$\vector(1,0){10}$}
\put(130,100){$\vector(-1,0){10}$}
\put(170,100){$\vector(1,0){10}$}
\put(0,32){\vector(0,-1){14}}
\put(0,68){\vector(0,1){14}}
\put(50,18){\vector(0,1){14}}
\put(50,82){\vector(0,-1){14}}
\put(100,-8){\vector(0,-1){9}}
\put(100,8){\vector(0,1){9}}
\put(100,42){\vector(0,-1){9}}
\put(100,58){\vector(0,1){9}}
\put(100,92){\vector(0,-1){9}}
\put(100,108){\vector(0,1){9}}
\put(150,32){\vector(0,-1){14}}
\put(150,68){\vector(0,1){14}}
\put(200,18){\vector(0,1){14}}
\put(200,82){\vector(0,-1){14}}

\put(80,-15){\vector(-2,1){10}}
\put(80,15){\vector(-2,-1){10}}
\put(80,85){\vector(-2,1){10}}
\put(80,115){\vector(-2,-1){10}}
\put(120,35){\vector(2,1){10}}
\put(120,65){\vector(2,-1){10}}
\put(-169,-7.5)
{
\put(50,-25){\framebox(38,15){\small $x^{(4)}_1(0)$}}
\put(50,0){\dashbox{3}(38,15){\small $*$}}
\put(50,25){\framebox(38,15){\small $x^{(4)}_3(0)$}}
\put(50,50){\dashbox{3}(38,15){\small $*$}}
\put(50,75){\framebox(38,15){\small $x^{(4)}_5(0)$}}
\put(50,100){\dashbox{3}(38,15){\small $*$}}
\put(50,125){\framebox(38,15){\small $x^{(4)}_7(0)$}}
\put(100,-25){\dashbox{3}(38,15){\small $*$}}
\put(100,0){\framebox(38,15){\small $x^{(3)}_2(0)$}}
\put(100,25){\dashbox{3}(38,15){\small $*$}}
\put(100,50){\framebox(38,15){\small $x^{(3)}_4(0)$}}
\put(100,75){\dashbox{3}(38,15){\small $*$}}
\put(100,100){\framebox(38,15){\small $x^{(3)}_6(0)$}}
\put(100,125){\dashbox{3}(38,15){\small $*$}}
}
\put(-50,-8){\vector(0,-1){9}}
\put(-50,8){\vector(0,1){9}}
\put(-50,42){\vector(0,-1){9}}
\put(-50,58){\vector(0,1){9}}
\put(-50,92){\vector(0,-1){9}}
\put(-50,108){\vector(0,1){9}}
\put(-100,-17){\vector(0,1){9}}
\put(-100,17){\vector(0,-1){9}}
\put(-100,33){\vector(0,1){9}}
\put(-100,67){\vector(0,-1){9}}
\put(-100,83){\vector(0,1){9}}
\put(-100,117){\vector(0,-1){9}}
\put(-70,-25){$\vector(-1,0){10}$}
\put(-80,0){$\vector(1,0){10}$}
\put(-70,25){$\vector(-1,0){10}$}
\put(-80,50){$\vector(1,0){10}$}
\put(-70,75){$\vector(-1,0){10}$}
\put(-80,100){$\vector(1,0){10}$}
\put(-70,125){$\vector(-1,0){10}$}
\dottedline(-120,-40)(220,-40)
\put(-135,-42){$\updownarrow$}
\put(-150,-42){$\mu^{\bullet}_-$}
}
\end{picture}
\caption{(Continues to Figure \ref{fig:labelxF2})
Label of cluster variables $x_{\mathbf{i}}(u)$
by $\mathcal{I}_{\ell+}$  for
$F_4$, $\ell=4$.
The middle column in the right quiver (marked by $\diamond$)
is identified with the right column in the left quiver.
}
\label{fig:labelxF1}
\end{figure}
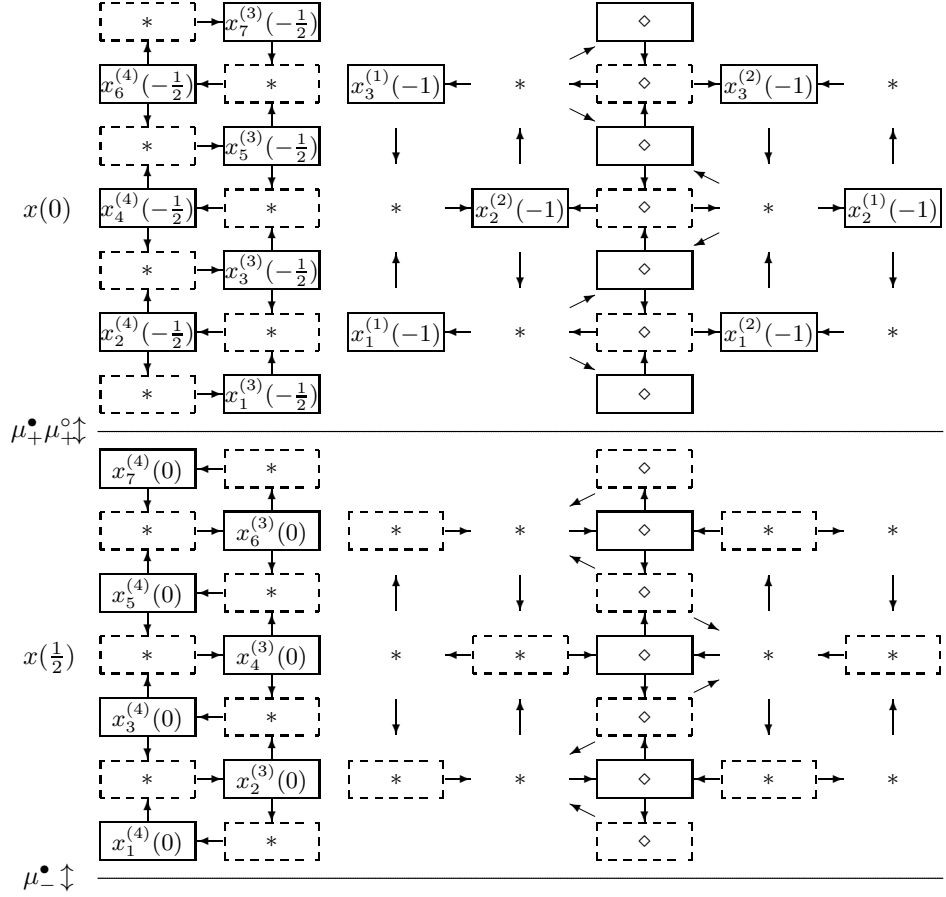

\begin{figure}
\setlength{\unitlength}{0.94pt}
\begin{picture}(370,370)(-150,-230)
\put(-150,46.5){$x(1)$}
\put(-19,-7.5)
{
\put(0,0){\makebox(38,15){\small $*$}}
\put(50,0){\framebox(38,15){\small $x^{(2)}_1(0)$}}
\put(150,0){\makebox(38,15){\small $*$}}
\put(200,0){\framebox(38,15){\small $x^{(1)}_1(0)$}}
\put(0,50){\framebox(38,15){\small $x^{(1)}_2(0)$}}
\put(50,50){\makebox(38,15){\small $*$}}
\put(150,50){\framebox(38,15){\small $x^{(2)}_2(0)$}}
\put(200,50){\makebox(38,15){\small $*$}}
\put(0,100){\makebox(38,15){\small $*$}}
\put(50,100){\framebox(38,15){\small $x^{(2)}_3(0)$}}
\put(150,100){\makebox(38,15){\small $*$}}
\put(200,100){\framebox(38,15){\small $x^{(1)}_3(0)$}}
\put(100,-25){\framebox(38,15){\small $\diamond$}}
\put(100,0){\dashbox{3}(38,15){\small $\diamond$}}
\put(100,25){\framebox(38,15){\small $\diamond$}}
\put(100,50){\dashbox{3}(38,15){\small $\diamond$}}
\put(100,75){\framebox(38,15){\small $\diamond$}}
\put(100,100){\dashbox{3}(38,15){\small $\diamond$}}
\put(100,125){\framebox(38,15){\small $\diamond$}}
}
%
\put(20,0){$\vector(1,0){10}$}
\put(80,0){$\vector(-1,0){10}$}
\put(120,0){$\vector(1,0){10}$}
\put(170,0){$\vector(1,0){10}$}
\put(30,50){$\vector(-1,0){10}$}
\put(80,50){$\vector(-1,0){10}$}
\put(120,50){$\vector(1,0){10}$}
\put(180,50){$\vector(-1,0){10}$}
\put(20,100){$\vector(1,0){10}$}
\put(80,100){$\vector(-1,0){10}$}
\put(120,100){$\vector(1,0){10}$}
\put(170,100){$\vector(1,0){10}$}
\put(0,32){\vector(0,-1){14}}
\put(0,68){\vector(0,1){14}}
\put(50,18){\vector(0,1){14}}
\put(50,82){\vector(0,-1){14}}
\put(100,-17){\vector(0,1){9}}
\put(100,17){\vector(0,-1){9}}
\put(100,33){\vector(0,1){9}}
\put(100,67){\vector(0,-1){9}}
\put(100,83){\vector(0,1){9}}
\put(100,117){\vector(0,-1){9}}
\put(150,32){\vector(0,-1){14}}
\put(150,68){\vector(0,1){14}}
\put(200,18){\vector(0,1){14}}
\put(200,82){\vector(0,-1){14}}
\put(130,-10){\vector(-2,-1){10}}
\put(130,10){\vector(-2,1){10}}
\put(130,90){\vector(-2,-1){10}}
\put(130,110){\vector(-2,1){10}}
\put(70,40){\vector(2,-1){10}}
\put(70,60){\vector(2,1){10}}

\put(-169,-7.5)
{
\put(50,-25){\dashbox{3}(38,15){\small $*$}}
\put(50,0){\framebox(38,15){\small $x^{(4)}_2(\frac{1}{2})$}}
\put(50,25){\dashbox{3}(38,15){\small $*$}}
\put(50,50){\framebox(38,15){\small $x^{(4)}_4(\frac{1}{2})$}}
\put(50,75){\dashbox{3}(38,15){\small $*$}}
\put(50,100){\framebox(38,15){\small $x^{(4)}_6(\frac{1}{2})$}}
\put(50,125){\dashbox{3}(38,15){\small $*$}}
\put(100,-25){\framebox(38,15){\small $x^{(3)}_1(\frac{1}{2})$}}
\put(100,0){\dashbox{3}(38,15){\small $*$}}
\put(100,25){\framebox(38,15){\small $x^{(3)}_3(\frac{1}{2})$}}
\put(100,50){\dashbox{3}(38,15){\small $*$}}
\put(100,75){\framebox(38,15){\small $x^{(3)}_5(\frac{1}{2})$}}
\put(100,100){\dashbox{3}(38,15){\small $*$}}
\put(100,125){\framebox(38,15){\small $x^{(3)}_7(\frac{1}{2})$}}
}
\put(-50,-17){\vector(0,1){9}}
\put(-50,17){\vector(0,-1){9}}
\put(-50,33){\vector(0,1){9}}
\put(-50,67){\vector(0,-1){9}}
\put(-50,83){\vector(0,1){9}}
\put(-50,117){\vector(0,-1){9}}
\put(-100,-8){\vector(0,-1){9}}
\put(-100,8){\vector(0,1){9}}
\put(-100,42){\vector(0,-1){9}}
\put(-100,58){\vector(0,1){9}}
\put(-100,92){\vector(0,-1){9}}
\put(-100,108){\vector(0,1){9}}
\put(-80,-25){$\vector(1,0){10}$}
\put(-70,0){$\vector(-1,0){10}$}
\put(-80,25){$\vector(1,0){10}$}
\put(-70,50){$\vector(-1,0){10}$}
\put(-80,75){$\vector(1,0){10}$}
\put(-70,100){$\vector(-1,0){10}$}
\put(-80,125){$\vector(1,0){10}$}
\dottedline(-120,-40)(220,-40)
\put(-130,-42){$\updownarrow$}
\put(-155,-42){$\mu^{\bullet}_+\mu^{\circ}_-$}
\put(0,-180)
{
\put(-150,46.5){$x(\frac{3}{2})$}
\put(-19,-7.5)
{
\put(0,0){\makebox(38,15){\small $*$}}
\put(50,0){\dashbox{3}(38,15){\small $*$}}
\put(150,0){\makebox(38,15){\small $*$}}
\put(200,0){\dashbox{3}(38,15){\small $*$}}
\put(0,50){\dashbox{3}(38,15){\small $*$}}
\put(50,50){\makebox(38,15){\small $*$}}
\put(150,50){\dashbox{3}(38,15){\small $*$}}
\put(200,50){\makebox(38,15){\small $*$}}
\put(0,100){\makebox(38,15){\small $*$}}
\put(50,100){\dashbox{3}(38,15){\small $*$}}
\put(150,100){\makebox(38,15){\small $*$}}
\put(200,100){\dashbox{3}(38,15){\small $*$}}
\put(100,-25){\dashbox{3}(38,15){\small $\diamond$}}
\put(100,0){\framebox(38,15){\small $\diamond$}}
\put(100,25){\dashbox{3}(38,15){\small $\diamond$}}
\put(100,50){\framebox(38,15){\small $\diamond$}}
\put(100,75){\dashbox{3}(38,15){\small $\diamond$}}
\put(100,100){\framebox(38,15){\small $\diamond$}}
\put(100,125){\dashbox{3}(38,15){\small $\diamond$}}
}
%
\put(30,0){$\vector(-1,0){10}$}
\put(70,0){$\vector(1,0){10}$}
\put(130,0){$\vector(-1,0){10}$}
\put(180,0){$\vector(-1,0){10}$}
\put(20,50){$\vector(1,0){10}$}
\put(70,50){$\vector(1,0){10}$}
\put(130,50){$\vector(-1,0){10}$}
\put(170,50){$\vector(1,0){10}$}
\put(30,100){$\vector(-1,0){10}$}
\put(70,100){$\vector(1,0){10}$}
\put(130,100){$\vector(-1,0){10}$}
\put(180,100){$\vector(-1,0){10}$}
\put(0,18){\vector(0,1){14}}
\put(0,82){\vector(0,-1){14}}
\put(50,32){\vector(0,-1){14}}
\put(50,68){\vector(0,1){14}}
\put(100,-8){\vector(0,-1){9}}
\put(100,8){\vector(0,1){9}}
\put(100,42){\vector(0,-1){9}}
\put(100,58){\vector(0,1){9}}
\put(100,92){\vector(0,-1){9}}
\put(100,108){\vector(0,1){9}}
\put(150,18){\vector(0,1){14}}
\put(150,82){\vector(0,-1){14}}
\put(200,32){\vector(0,-1){14}}
\put(200,68){\vector(0,1){14}}

\put(120,-15){\vector(2,1){10}}
\put(120,15){\vector(2,-1){10}}
\put(120,85){\vector(2,1){10}}
\put(120,115){\vector(2,-1){10}}
\put(80,35){\vector(-2,1){10}}
\put(80,65){\vector(-2,-1){10}}
\put(-169,-7.5)
{
\put(50,-25){\framebox(38,15){\small $x^{(4)}_1(1)$}}
\put(50,0){\dashbox{3}(38,15){\small $*$}}
\put(50,25){\framebox(38,15){\small $x^{(4)}_3(1)$}}
\put(50,50){\dashbox{3}(38,15){\small $*$}}
\put(50,75){\framebox(38,15){\small $x^{(4)}_5(1)$}}
\put(50,100){\dashbox{3}(38,15){\small $*$}}
\put(50,125){\framebox(38,15){\small $x^{(4)}_7(1)$}}
\put(100,-25){\dashbox{3}(38,15){\small $*$}}
\put(100,0){\framebox(38,15){\small $x^{(3)}_2(1)$}}
\put(100,25){\dashbox{3}(38,15){\small $*$}}
\put(100,50){\framebox(38,15){\small $x^{(3)}_4(1)$}}
\put(100,75){\dashbox{3}(38,15){\small $*$}}
\put(100,100){\framebox(38,15){\small $x^{(3)}_6(1)$}}
\put(100,125){\dashbox{3}(38,15){\small $*$}}
}
\put(-50,-8){\vector(0,-1){9}}
\put(-50,8){\vector(0,1){9}}
\put(-50,42){\vector(0,-1){9}}
\put(-50,58){\vector(0,1){9}}
\put(-50,92){\vector(0,-1){9}}
\put(-50,108){\vector(0,1){9}}
\put(-100,-17){\vector(0,1){9}}
\put(-100,17){\vector(0,-1){9}}
\put(-100,33){\vector(0,1){9}}
\put(-100,67){\vector(0,-1){9}}
\put(-100,83){\vector(0,1){9}}
\put(-100,117){\vector(0,-1){9}}
\put(-70,-25){$\vector(-1,0){10}$}
\put(-80,0){$\vector(1,0){10}$}
\put(-70,25){$\vector(-1,0){10}$}
\put(-80,50){$\vector(1,0){10}$}
\put(-70,75){$\vector(-1,0){10}$}
\put(-80,100){$\vector(1,0){10}$}
\put(-70,125){$\vector(-1,0){10}$}
\dottedline(-120,-40)(220,-40)
\put(-135,-42){$\updownarrow$}
\put(-150,-42){$\mu^{\bullet}_-$}
}
\end{picture}
\caption{(Continues from Figure \ref{fig:labelxF1}).
}
\label{fig:labelxF2}
\end{figure}

\subsection{Cluster algebra and alternative labels}

Let $B_{\ell}(F_4)$
be the corresponding skew-symmetric matrix 
to the quiver $Q_{\ell}(F_4)$.
In the rest of the section,
we set the matrix $B=(B_{\mathbf{i}\mathbf{j}})_{\mathbf{i},
\mathbf{j}\in \mathbf{I}}=B_{\ell}(F_4)$
unless otherwise mentioned.

Let $\mathcal{A}(B,x,y)$ 
be the cluster algebra
 with coefficients
in the universal
semifield
$\mathbb{Q}_{\mathrm{sf}}(y)$,
and the  coefficient group $\mathcal{G}(B,y)$
associated with $\mathcal{A}(B,x,y)$.

In view of Lemma \ref{lem:FQmut}
we set $x(0)=x$, $y(0)=y$ and define 
clusters $x(u)=(x_{\mathbf{i}}(u))_{\mathbf{i}\in \mathbf{I}}$
 ($u\in \frac{1}{2}\mathbb{Z}$)
 and coefficient tuples $y(u)=(y_\mathbf{i}(u))_{\mathbf{i}\in \mathbf{I}}$
 ($u\in \frac{1}{2}\mathbb{Z}$)
by the sequence of mutations  \eqref{eq:Cmutseq}.

For a pair $(\mathbf{i},u)\in
 \mathbf{I}\times \frac{1}{2}\mathbb{Z}$,
we set the same parity condition $\mathbf{p}_+$
and $\mathbf{p}_-$ as \eqref{eq:Cpp}.
We have \eqref{eq:Cppm},
and each $(\mathbf{i},u):\mathbf{p}_+$
is a mutation point of \eqref{eq:Cmutseq} in the forward
direction of $u$,
and each  $(\mathbf{i},u):\mathbf{p}_-$
is so in the backward direction of $u$ as before.

\begin{lemma}
 Below $\equiv$ means the equivalence modulo
$2\mathbb{Z}$.
\par
(i)
The map
\begin{align}
\begin{matrix}
g: &\mathcal{I}_{\ell+}&\rightarrow & \{ (\mathbf{i},u): \mathbf{p}_+
 \}\hfill \\
&(a,m,u-\frac{1}{t_a})&\mapsto &
\begin{cases}
((a,m),u)& \mbox{\rm $a=1,2$;
$a+m+u\equiv 0$}\\
((7-a,m),u)& \mbox{\rm  $a=1,2$;
$a+m+u\equiv 1$}\\
((a,m),u)& \mbox{\rm  $a=3,4$}\\
\end{cases}
\end{matrix}
\end{align}
is a bijection.
\par
(ii)
The map
\begin{align}
\begin{matrix}
g': &\mathcal{I}'_{\ell+}&\rightarrow & \{ (\mathbf{i},u): \mathbf{p}_+
\}\hfill \\
&(a,m,u)&\mapsto &
\begin{cases}
((a,m),u)& \mbox{\rm $a=1,2$;
$a+m+u\equiv 0$}\\
((7-a,m),u)& \mbox{\rm  $a=1,2$;
$a+m+u\equiv 1$}\\
((a,m),u)& a=3,4\\
\end{cases}
\end{matrix}
\end{align}
is a bijection.
\end{lemma}

We introduce alternative labels
$x_{\mathbf{i}}(u)=x^{(a)}_m(u-1/t_a)$
($(a,m,u-1/t_a)\in \mathcal{I}_{\ell+}$)
for $(\mathbf{i},u)=g((a,m,u-1/t_a))$
and
$y_{\mathbf{i}}(u)=y^{(a)}_m(u)$
($(a,m,u)\in \mathcal{I}'_{\ell+}$)
for $(\mathbf{i},u)=g'((a,m,u))$, respectively.
See Figures \ref{fig:labelxF1}--\ref{fig:labelxF2}.

\subsection{T-system and cluster algebra}

The result in this subsection is completely parallel
to the $B_r$ and $C_r$ cases.

\begin{lemma}
\label{lem:Fx2}
The family $\{x^{(a)}_m(u)
\mid (a,m,u)\in \mathcal{I}_{\ell+}\}$
satisfies the system of relations
 \eqref{eq:Cx2}
 with $G(b,k,v;a,m,u)$
for $\mathbb{T}_{\ell}(F_4)$.
In particular,
the family $\{ [x^{(a)}_m(u)]_{\mathbf{1}}
\mid (a,m,u)\in \mathcal{I}_{\ell+}\}$
satisfies the T-system $\mathbb{T}_{\ell}(F_4)$
in $\mathcal{A}(B,x)$
by replacing $T^{(a)}_m(u)$ with $[x^{(a)}_m(u)]_{\mathbf{1}}$.
\end{lemma}

The T-subalgebra $\mathcal{A}_T(B,x)$ is defined as
Definition \ref{def:Tsub}.

\begin{theorem}
\label{thm:FTiso}
The ring $\EuScript{T}^{\circ}_{\ell}(F_4)_+$ is isomorphic to
$\mathcal{A}_T(B,x)$ by the correspondence
$T^{(a)}_m(u)\mapsto [x^{(a)}_m(u)]_{\mathbf{1}}$.
\end{theorem}

\subsection{Y-system and cluster algebra}

The result in this subsection is completely parallel
to the $B_r$ and $C_r$ cases.

\begin{lemma}
\label{lem:Fy2}
The family $\{ y^{(a)}_m(u)
\mid (a,m,u)\in \mathcal{I}'_{\ell+}\}$
satisfies the Y-system $\mathbb{Y}_{\ell}(F_4)$
by replacing $Y^{(a)}_m(u)$ with $y^{(a)}_m(u)$.
\end{lemma}

The Y-subgroup $\mathcal{G}_Y(B,y)$ is defined as
Definition \ref{def:Ysub}.

\begin{theorem}
\label{thm:FYiso}
The group $\EuScript{Y}^{\circ}_{\ell}(F_4)_+$ is isomorphic to
$\mathcal{G}_Y(B,y)$ by the correspondence
$Y^{(a)}_m(u)\mapsto y^{(a)}_m(u)$
and $1+Y^{(a)}_m(u)\mapsto 1+y^{(a)}_m(u)$.
\end{theorem}

\subsection{Tropical Y-system at level 2}

By direct computations, the following properties
are verified.

\begin{proposition}
\label{prop:Flev2}
 For 
$[\mathcal{G}_Y(B,y)]_{\mathbf{T}}$
with $B=B_{2}(F_4)$, the following facts hold.
\par
(i) Let $u$ be in the region $0\le u < 2$.
For any $(\mathbf{i},u):\mathbf{p}_+$,
the  monomial $[y_{\mathbf{i}}(u)]_{\mathbf{T}}$
is positive.
\par
(ii) Let $u$ be in the region $-h^{\vee}\le u < 0$.
\begin{itemize}
\item[\em (a)]
 Let $\mathbf{i}=(1,1),(2,1),(5,1),(6,1),(3,2)$,
or $(4,2)$.
For any $(\mathbf{i},u):\mathbf{p}_+$,
the  monomial $[y_{\mathbf{i}}(u)]_{\mathbf{T}}$
is negative.
\item[\em (b)]
Let $\mathbf{i}=(3,1), (3,3), (4,1)$, or $(4,3)$.
For any $(\mathbf{i},u):\mathbf{p}_+$,
the  monomial $[y_{\mathbf{i}}(u)]_{\mathbf{T}}$
is negative for
 $
u=-\frac{1}{2},-1,-\frac{3}{2},-3,-\frac{7}{2},
-4,
-\frac{11}{2},
-6,
 -\frac{13}{2}, -8 ,-\frac{17}{2},
\\ 
-9$
and positive for $u=-2,-\frac{5}{2},-\frac{9}{2},
-5, -7, -\frac{15}{2}$.
\end{itemize}
\par
(iii)
$y_{ii'}(2)=y_{ii'}^{-1}$ if $i=1,2,5,6$ and
$y_{i,4-i'}^{-1}$ if $i=3,4$.
\par
(iv) $y_{ii'}(-h^{\vee})=
y_{7-i,i'}^{-1}$ if $i=1,2,5,6$
and $y_{ii'}^{-1}$ if $i=3,4$.
\end{proposition}

Also we have a description of the `core part'
of  $[y_{\mathbf{i}}(u)]_{\mathbf{T}}$
in the region $-h^{\vee}\leq  u <0$,
corresponding to the $D$ part for $C_r$,
in terms of the root system of $E_6$.
We use the following
index of the Dynkin diagram $E_6$.
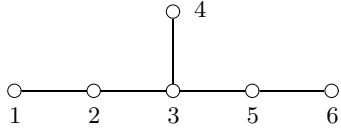
\begin{align*}
\begin{picture}(120,55)(-30,0)
\put(0,15)
{
\put(-30,0){\circle{5}}
\put(0,0){\circle{5}}
\put(30,0){\circle{5}}
\put(60,0){\circle{5}}
\put(90,0){\circle{5}}
\put(30,30){\circle{5}}
\put(-3,0){\line(-1,0){24}}
\put(27,0){\line(-1,0){24}}
\put(57,0){\line(-1,0){24}}
\put(87,0){\line(-1,0){24}}
\put(30,3){\line(0,1){24}}
\put(-4,8)
{
\put(-28,-20){\small $1$}
\put(2,-20){\small $2$}
\put(32,-20){\small $3$}
\put(62,-20){\small $5$}
\put(92,-20){\small $6$}
\put(42,20){\small $4$}
}
}
\end{picture}
\end{align*}
Let 
$\Pi=\{\alpha_1,\dots,\alpha_6\}$, $-\Pi$,
 $\Phi_+$ be  the set of the simple roots,
the negative simple roots, the positive roots, respectively,
of type $E_{6}$.
Let $\sigma_i$ be the piecewise-linear
analogue  of the simple reflection $s_i$,
acting on the set
of the almost positive roots
$\Phi_{\geq -1}=\Phi_{+}\sqcup (-\Pi)$.
We write $\sum_{i} m_i \alpha_i \in \Phi_+$
as $[1^{m_1}, 2^{m_2}, \ldots, 6^{m_6}]$;
furthermore, $[1^0,2^1,3^1,4^1,5^0,6^0]$,
for example, is abbreviated  as $[2,3,4]$.

We define $\sigma$ as the composition
\begin{align}
\label{eq:Fsigma}
\sigma=\sigma_3(\sigma_4\sigma_2\sigma_6)
\sigma_3(\sigma_4\sigma_1\sigma_5).
\end{align}

\begin{lemma}
\label{lem:Forbit}
We have the orbits  by $\sigma$
\begin{align}
\begin{split}
&
-\alpha_1
\ \rightarrow\ 
[1,2,3]
\ \rightarrow\ 
[2,3,4,5,6]
\ \rightarrow\ 
[1,2,3^2,4,5]
\ \rightarrow\ 
[5,6]
\ \rightarrow\ 
-\alpha_6,
\\
&
-\alpha_2
\ \rightarrow\ 
[2,3]
\ \rightarrow\ 
[1,2^2,3^2,4,5,6]
\ \rightarrow\ 
[1,2^2,3^3,4^2,5^2,6]\\
&
\hskip150pt
\ \rightarrow\ 
[1,2,3^2,4,5^2,6]
\ \rightarrow\ 
[5]
\ \rightarrow\ 
-\alpha_5,
\\
&
-\alpha_3
\ \rightarrow\ 
[2,3,4]
\ \rightarrow\ 
[1,2,3^2,4,5,6]
\ \rightarrow\ 
[2,3^2,4,5^2,6]
\ \rightarrow\ 
[1,2,3,5]
\ \rightarrow\ 
-\alpha_3,
\\
&
\phantom{-}
\hskip4.5pt
\alpha_3
\ \rightarrow\ 
[2,3,5,6]
\ \rightarrow\ 
[1,2^2,3^2,4,5]
\ \rightarrow\ 
[1,2,3,4,5,6]
\ \rightarrow\ 
[3,4,5]
\ \rightarrow\ 
\alpha_3,
\\
&
-\alpha_4
\ \rightarrow\ 
[2]
\ \rightarrow\ 
[1,2,3,4]
\ \rightarrow\ 
[3,4,5,6]
\ \rightarrow\ 
[3,5]
\ \rightarrow\ 
-\alpha_4,
\\
&
\phantom{-}
\hskip4.5pt
\alpha_4
\ \rightarrow\ 
[3,4]
\ \rightarrow\ 
[3,5,6]
\ \rightarrow\ 
[2,3,5]
\ \rightarrow\ 
[1,2]
\ \rightarrow\ 
\alpha_4,
\\
&
-\alpha_5
\ \rightarrow\ 
[2,3^2,4,5,6]
\ \rightarrow\ 
[1,2^2,3^3,4,5^2,6]
\ \rightarrow\ 
[1,2^2,3^2,4,5^2,6]\\
&
\hskip200pt
\ \rightarrow\ 
[1,2,3,4,5]
\ \rightarrow\ 
-\alpha_2,
\\
&
-\alpha_6
\ \rightarrow\ 
[6]
\ \rightarrow\ 
[2,3^2,4,5]
\ \rightarrow\ 
[1,2,3,5,6]
\ \rightarrow\ 
[2,3,4,5]
\ \rightarrow\ 
[1]
\ \rightarrow\ 
-\alpha_1.
\end{split}
\end{align}
In particular, these
elements in $\Phi_+$ exhaust the set $\Phi_+$,
thereby providing the orbit decomposition
of $\Phi_+$ by $\sigma$.
\end{lemma}

For $-h^{\vee}\leq u< 0$, define
\begin{align}
\label{eq:Falpha}
&\alpha_{i}(u)=
\begin{cases}
\sigma^{-u/2}(-\alpha_i)
& \mbox{\rm $i=1,4,5$; $u\equiv 0$},\\
\sigma^{-(u-1)/2}(-\alpha_i)
& \mbox{\rm $i=2,6$; $u\equiv -1$},\\
\sigma^{-(u-1)/2}(\alpha_4)
& \mbox{\rm $i=4$; $u\equiv -1$},\\
\sigma^{-(2u-1)/4}(-\alpha_3)
& \mbox{\rm $i=3$; $u\equiv-\frac{3}{2}$},\\
\sigma^{-(2u+1)/4}(\alpha_3)
& \mbox{\rm $i=3$; $u\equiv-\frac{1}{2}$},\\
\end{cases}
\end{align}
where $\equiv$ is mod $2\mathbb{Z}$.
By Lemma \ref{lem:Forbit} they are (all the) positive roots
of $E_{6}$.

For a monomial $m$ in $y=(y_{\mathbf{i}})_{\mathbf{i}\in
\mathbf{I}}$,
let $\pi_A(m)$ denote the specialization
with $y_{31}=y_{33}=y_{41}=y_{43}=1$.
For simplicity, we set
$y_{i1}=y_i$ ($i=1,2,5,6$), $y_{i2}=y_i$ ($i=3,4$),
and also,
$y_{i1}(u)=y_{i}(u)$
($i=1,2,5,6$), $y_{i2}(u)=y_i(u)$ ($i=3,4$).
We define the vectors $\mathbf{t}_{i}(u)
=(t_{i}(u)_k)_{k=1}^{6}$
by
\begin{align}
\pi_A([y_{i}(u)]_{\mathbf{T}})
=
\prod_{k=1}^{6}
 y_{k}^{t_{i}(u)_k}.
\end{align}
We also identify each vector $\mathbf{t}_{i}(u)$
with $\alpha=\sum_{k=1}^{6}
t_{i}(u)_k \alpha_k \in \mathbb{Z}\Pi$.

\begin{proposition}
\label{prop:Ftvec}
Let $-h^{\vee}\leq u< 0$.
Then, we have
\begin{align}
\label{eq:Ftvec1}
\mathbf{t}_{i}(u)=-\alpha_i(u)
\end{align}
for $(i,u)$ in \eqref{eq:Falpha}.
\end{proposition}

\subsection{Tropical Y-systems at higher levels}

Due to the factorization property,
we obtain the following.
\begin{proposition}
\label{prop:Flevh}
Let $\ell> 2$ be an integer.
 For 
$[\mathcal{G}_Y(B,y)]_{\mathbf{T}}$
with $B=B_{\ell}(F_4)$, the following facts hold.
\par
(i) Let $u$ be in the region $0\le u < \ell$.
For any $(\mathbf{i},u):\mathbf{p}_+$,
the  monomial $[y_{\mathbf{i}}(u)]_{\mathbf{T}}$
is positive.
\par
(ii) Let $u$ be in the region $-h^{\vee}\le u < 0$.
\begin{itemize}
\item[\em (a)]
 Let $\mathbf{i}\in \mathbf{I}^{\circ}$
or $\mathbf{i}=(3,i'),(4,i')$ $(i'\in 2\mathbb{N})$.
For any $(\mathbf{i},u):\mathbf{p}_+$,
the  monomial $[y_{\mathbf{i}}(u)]_{\mathbf{T}}$
is negative.
\item[\em (b)]
 Let $\mathbf{i}=(3,i'), (4,i')$
 $(i'\not\in 2\mathbb{N})$.
For any $(\mathbf{i},u):\mathbf{p}_+$,
the  monomial $[y_{\mathbf{i}}(u)]_{\mathbf{T}}$
is negative for
 $
u=-\frac{1}{2},-1,-\frac{3}{2},-3,-\frac{7}{2},
-4,
-\frac{11}{2},
-6,
 -\frac{13}{2}, -8 ,-\frac{17}{2},
-9$
and positive for $u=-2,-\frac{5}{2},-\frac{9}{2},
-5, -7, -\frac{15}{2}$.
\end{itemize}
\par
(iii)
$y_{ii'}(\ell)=y_{i,\ell-i'}^{-1}$ if $i=1,2,5,6$ and
$y_{i,2\ell-i'}^{-1}$ if $i=3,4$.
\par
(iv) $y_{ii'}(-h^{\vee})=
y_{7-i,i'}^{-1}$ if $i=1,2,5,6$
and $y_{ii'}^{-1}$ if $i=3,4$.
\end{proposition}

We obtain corollaries of
Propositions \ref{prop:Flev2} and  \ref{prop:Flevh}.

\begin{theorem}
\label{thm:FtYperiod}
For $[\mathcal{G}_Y(B,y)]_{\mathbf{T}}$
the following relations hold.
\par
(i) Half periodicity: 
$[y_{\mathbf{i}}(u+h^{\vee}+\ell)]_{\mathbf{T}}
=[y_{\boldsymbol{\omega}(\mathbf{i})}(u)]_{\mathbf{T}}$.
\par
(ii) 
 Full periodicity: 
$[y_{\mathbf{i}}(u+2(h^{\vee}+\ell))]_{\mathbf{T}}
=[y_{\mathbf{i}}(u)]_{\mathbf{T}}$.
\end{theorem}

\begin{theorem}
\label{thm:Flevhd}
For $[\mathcal{G}_Y(B,y)]_{\mathbf{T}}$,
let $N_+$ and $N_-$ denote the
total numbers of the positive and negative monomials,
respectively,
among $[y_{\mathbf{i}}(u)]_{\mathbf{T}}$
for $(\mathbf{i},u):\mathbf{p}_+$
in the region $0\leq u < 2(h^{\vee}+\ell)$.
Then, we have
\begin{align}
N_+=4\ell(3\ell+1),
\quad
N_-=24(4\ell-3).
\end{align}
\end{theorem}

\subsection{Periodicities and dilogarithm identities}

Applying \cite[Theorem 5.1]{IIKKN} to
Theorem  \ref{thm:FtYperiod},
we obtain the periodicities:

\begin{theorem}
\label{thm:Fxperiod}
For $\mathcal{A}(B,x,y)$,
the following relations hold.
\par
(i) Half periodicity: 
$x_{\mathbf{i}}(u+h^{\vee}+\ell)
=x_{\boldsymbol{\omega}(\mathbf{i})}(u)
$.
\par
(ii) 
 Full periodicity: 
$x_{\mathbf{i}}(u+2(h^{\vee}+\ell))
=x_{\mathbf{i}}(u)
$.
\end{theorem}

\begin{theorem}
\label{thm:Fyperiod}
For $\mathcal{G}(B,y)$,
the following relations hold.
\par
(i) Half periodicity: 
$y_{\mathbf{i}}(u+h^{\vee}+\ell)
=y_{\boldsymbol{\omega}(\mathbf{i})}(u)
$.
\par
(ii) 
 Full periodicity: 
$y_{\mathbf{i}}(u+2(h^{\vee}+\ell))
=y_{\mathbf{i}}(u)
$.
\end{theorem}

Then,
 Theorems \ref{thm:Tperiod} and \ref{thm:Yperiod}
 for $F_4$ follow from
Theorems
\ref{thm:FTiso}, \ref{thm:FYiso},
\ref{thm:Fxperiod}, and \ref{thm:Fyperiod}.
Furthermore, Theorem \ref{thm:DI2}  for $F_4$ is  obtained from
the above periodicities  and Theorem \ref{thm:Flevhd} as
in the $B_r$ case \cite[Section 6]{IIKKN}.

\section{Type $G_2$}

The $G_2$ case is mostly parallel to the former cases,
but slightly different because the number $t$
in \eqref{eq:t1} is three.
Again, the properties of the tropical Y-system at level 2
(Proposition \ref{prop:Glev2}) are crucial and
specific to $G_2$.

\subsection{Parity decompositions of T and Y-systems}

For a triplet $(a,m,u)\in \mathcal{I}_{\ell}$,
we reset the parity conditions $\mathbf{P}_{+}$ and
$\mathbf{P}_{-}$ by
\begin{align}
\label{eq:GPcond}
\mathbf{P}_{+}:& \ \mbox{$a+m+3u$ is even},\\
\mathbf{P}_{-}:& \ \mbox{$a+m+3u$ is odd}.
\end{align}
Then, we have
$\EuScript{T}^{\circ}_{\ell}(G_2)_+
\simeq
\EuScript{T}^{\circ}_{\ell}(G_2)_-
$
by $T^{(a)}_m(u)\mapsto T^{(a)}_m(u+\frac{1}{3})$ and
\begin{align}
\EuScript{T}^{\circ}_{\ell}(G_2)
\simeq
\EuScript{T}^{\circ}_{\ell}(G_2)_+
\otimes_{\mathbb{Z}}
\EuScript{T}^{\circ}_{\ell}(G_2)_-.
\end{align}

For a triplet $(a,m,u)\in \mathcal{I}_{\ell}$ ,
we reset the parity conditions $\mathbf{P}'_{+}$ and
$\mathbf{P}'_{-}$ by
\begin{align}
\label{eq:GP'cond}
\mathbf{P}'_{+}:& \ \mbox{$a+m+3u$ is odd},\\
\mathbf{P}'_{-}:& \ \mbox{$a+m+3u$ is even}.
\end{align}
We have
\begin{align}
(a,m,u):\mathbf{P}'_+
\ \Longleftrightarrow \
\textstyle (a,m,u\pm \frac{1}{t_a}):\mathbf{P}_+.
\end{align}
Also, we have
$\EuScript{Y}^{\circ}_{\ell}(G_2)_+
\simeq
\EuScript{Y}^{\circ}_{\ell}(G_2)_-
$
by $Y^{(a)}_m(u)\mapsto Y^{(a)}_m(u+\frac{1}{3})$,
$1+Y^{(a)}_m(u)\mapsto 1+Y^{(a)}_m(u+\frac{1}{3})$,
 and
\begin{align}
\EuScript{Y}^{\circ}_{\ell}(G_2)
\simeq
\EuScript{Y}^{\circ}_{\ell}(G_2)_+
\times
\EuScript{Y}^{\circ}_{\ell}(G_2)_-.
\end{align}

\subsection{Quiver $Q_{\ell}(G_2)$}

\begin{figure}
\begin{picture}(275,155)(-50,0)
%
\put(0,0)
{
\put(0,30){\circle{5}}
\put(0,75){\circle{5}}
\put(0,120){\circle{5}}
\put(30,0){\circle*{5}}
\put(30,15){\circle*{5}}
\put(30,30){\circle*{5}}
\put(30,45){\circle*{5}}
\put(30,60){\circle*{5}}
\put(30,75){\circle*{5}}
\put(30,90){\circle*{5}}
\put(30,105){\circle*{5}}
\put(30,120){\circle*{5}}
\put(30,135){\circle*{5}}
\put(30,150){\circle*{5}}
%
\put(30,3){\vector(0,1){9}}
\put(30,27){\vector(0,-1){9}}
\put(30,33){\vector(0,1){9}}
\put(30,57){\vector(0,-1){9}}
\put(30,63){\vector(0,1){9}}
\put(30,87){\vector(0,-1){9}}
\put(30,93){\vector(0,1){9}}
\put(30,117){\vector(0,-1){9}}
\put(30,123){\vector(0,1){9}}
\put(30,147){\vector(0,-1){9}}
\put(0,72){\vector(0,-1){39}}
\put(0,78){\vector(0,1){39}}
\put(3,28){\vector(1,-1){24}}
\put(27,17){\vector(-2,1){23}}
\put(3,30){\vector(1,0){24}}
\put(27,43){\vector(-2,-1){23}}
\put(3,32){\vector(1,1){24}}
\put(27,75){\vector(-1,0){24}}
\put(3,118){\vector(1,-1){24}}
\put(27,107){\vector(-2,1){23}}
\put(3,120){\vector(1,0){24}}
\put(27,133){\vector(-2,-1){23}}
\put(3,122){\vector(1,1){24}}
\put(3,-2)
{
\put(-17,30){\small IV}
\put(-17,75){\small I}
\put(-17,120){\small IV}
}
\put(4,-2)
{
\put(30,0){\small $+$}
\put(30,15){\small $-$}
\put(30,30){\small $+$}
\put(30,45){\small $-$}
\put(30,60){\small $+$}
\put(30,75){\small $-$}
\put(30,90){\small $+$}
\put(30,105){\small $-$}
\put(30,120){\small $+$}
\put(30,135){\small $-$}
\put(30,150){\small $+$}
}
}
\put(80,0)
{
\put(0,30){\circle{5}}
\put(0,75){\circle{5}}
\put(0,120){\circle{5}}
\put(30,0){\circle*{5}}
\put(30,15){\circle*{5}}
\put(30,30){\circle*{5}}
\put(30,45){\circle*{5}}
\put(30,60){\circle*{5}}
\put(30,75){\circle*{5}}
\put(30,90){\circle*{5}}
\put(30,105){\circle*{5}}
\put(30,120){\circle*{5}}
\put(30,135){\circle*{5}}
\put(30,150){\circle*{5}}
%
\put(30,3){\vector(0,1){9}}
\put(30,27){\vector(0,-1){9}}
\put(30,33){\vector(0,1){9}}
\put(30,57){\vector(0,-1){9}}
\put(30,63){\vector(0,1){9}}
\put(30,87){\vector(0,-1){9}}
\put(30,93){\vector(0,1){9}}
\put(30,117){\vector(0,-1){9}}
\put(30,123){\vector(0,1){9}}
\put(30,147){\vector(0,-1){9}}
\put(0,33){\vector(0,1){39}}
\put(0,117){\vector(0,-1){39}}
\put(27,17){\vector(-2,1){24}}
\put(3,30){\vector(1,0){24}}
\put(27,43){\vector(-2,-1){24}}
\put(3,73){\vector(2,-1){24}}
\put(27,75){\vector(-1,0){23}}
\put(3,77){\vector(2,1){24}}
\put(27,107){\vector(-2,1){24}}
\put(3,120){\vector(1,0){24}}
\put(27,133){\vector(-2,-1){24}}
\put(3,-2)
{
\put(-17,30){\small II}
\put(-17,75){\small V}
\put(-17,120){\small II}
}
\put(4,-2)
{
\put(30,0){\small $+$}
\put(30,15){\small $-$}
\put(30,30){\small $+$}
\put(30,45){\small $-$}
\put(30,60){\small $+$}
\put(30,75){\small $-$}
\put(30,90){\small $+$}
\put(30,105){\small $-$}
\put(30,120){\small $+$}
\put(30,135){\small $-$}
\put(30,150){\small $+$}
}
}
\put(160,0)
{
\put(0,30){\circle{5}}
\put(0,75){\circle{5}}
\put(0,120){\circle{5}}
\put(30,0){\circle*{5}}
\put(30,15){\circle*{5}}
\put(30,30){\circle*{5}}
\put(30,45){\circle*{5}}
\put(30,60){\circle*{5}}
\put(30,75){\circle*{5}}
\put(30,90){\circle*{5}}
\put(30,105){\circle*{5}}
\put(30,120){\circle*{5}}
\put(30,135){\circle*{5}}
\put(30,150){\circle*{5}}
%
\put(30,3){\vector(0,1){9}}
\put(30,27){\vector(0,-1){9}}
\put(30,33){\vector(0,1){9}}
\put(30,57){\vector(0,-1){9}}
\put(30,63){\vector(0,1){9}}
\put(30,87){\vector(0,-1){9}}
\put(30,93){\vector(0,1){9}}
\put(30,117){\vector(0,-1){9}}
\put(30,123){\vector(0,1){9}}
\put(30,147){\vector(0,-1){9}}
\put(0,72){\vector(0,-1){39}}
\put(0,78){\vector(0,1){39}}
\put(3,30){\vector(1,0){24}}
\put(27,47){\vector(-1,1){24}}
\put(3,73){\vector(2,-1){24}}
\put(27,75){\vector(-1,0){23}}
\put(3,77){\vector(2,1){24}}
\put(27,103){\vector(-1,-1){24}}
\put(3,120){\vector(1,0){24}}
\put(3,-2)
{
\put(-17,30){\small VI}
\put(-17,75){\small III}
\put(-17,120){\small VI}
}
\put(4,-2)
{
\put(30,0){\small $+$}
\put(30,15){\small $-$}
\put(30,30){\small $+$}
\put(30,45){\small $-$}
\put(30,60){\small $+$}
\put(30,75){\small $-$}
\put(30,90){\small $+$}
\put(30,105){\small $-$}
\put(30,120){\small $+$}
\put(30,135){\small $-$}
\put(30,150){\small $+$}
}
}
\put(-50,73){${\scriptstyle\ell -1}\left\{ \makebox(0,50){}\right.$}
\put(210,73){$\left. \makebox(0,80){}\right\}{\scriptstyle3\ell -1}$}
\end{picture}
\begin{picture}(275,210)(-50,0)
%
\put(0,0)
{
\put(0,30){\circle{5}}
\put(0,75){\circle{5}}
\put(0,120){\circle{5}}
\put(0,165){\circle{5}}
\put(30,0){\circle*{5}}
\put(30,15){\circle*{5}}
\put(30,30){\circle*{5}}
\put(30,45){\circle*{5}}
\put(30,60){\circle*{5}}
\put(30,75){\circle*{5}}
\put(30,90){\circle*{5}}
\put(30,105){\circle*{5}}
\put(30,120){\circle*{5}}
\put(30,135){\circle*{5}}
\put(30,150){\circle*{5}}
\put(30,165){\circle*{5}}
\put(30,180){\circle*{5}}
\put(30,195){\circle*{5}}
%
\put(30,3){\vector(0,1){9}}
\put(30,27){\vector(0,-1){9}}
\put(30,33){\vector(0,1){9}}
\put(30,57){\vector(0,-1){9}}
\put(30,63){\vector(0,1){9}}
\put(30,87){\vector(0,-1){9}}
\put(30,93){\vector(0,1){9}}
\put(30,117){\vector(0,-1){9}}
\put(30,123){\vector(0,1){9}}
\put(30,147){\vector(0,-1){9}}
\put(30,153){\vector(0,1){9}}
\put(30,177){\vector(0,-1){9}}
\put(30,183){\vector(0,1){9}}
\put(0,72){\vector(0,-1){39}}
\put(0,78){\vector(0,1){39}}
\put(0,162){\vector(0,-1){39}}
\put(3,28){\vector(1,-1){24}}
\put(27,17){\vector(-2,1){23}}
\put(3,30){\vector(1,0){24}}
\put(27,43){\vector(-2,-1){23}}
\put(3,32){\vector(1,1){24}}
\put(27,75){\vector(-1,0){24}}
\put(3,118){\vector(1,-1){24}}
\put(27,107){\vector(-2,1){23}}
\put(3,120){\vector(1,0){24}}
\put(27,133){\vector(-2,-1){23}}
\put(3,122){\vector(1,1){24}}
\put(27,165){\vector(-1,0){24}}
\put(3,-2)
{
\put(-17,30){\small IV}
\put(-17,75){\small I}
\put(-17,120){\small IV}
\put(-17,165){\small I}
}
\put(4,-2)
{
\put(30,0){\small $+$}
\put(30,15){\small $-$}
\put(30,30){\small $+$}
\put(30,45){\small $-$}
\put(30,60){\small $+$}
\put(30,75){\small $-$}
\put(30,90){\small $+$}
\put(30,105){\small $-$}
\put(30,120){\small $+$}
\put(30,135){\small $-$}
\put(30,150){\small $+$}
\put(30,165){\small $-$}
\put(30,180){\small $+$}
\put(30,195){\small $-$}
}
}
\put(80,0)
{
\put(0,30){\circle{5}}
\put(0,75){\circle{5}}
\put(0,120){\circle{5}}
\put(0,165){\circle{5}}
\put(30,0){\circle*{5}}
\put(30,15){\circle*{5}}
\put(30,30){\circle*{5}}
\put(30,45){\circle*{5}}
\put(30,60){\circle*{5}}
\put(30,75){\circle*{5}}
\put(30,90){\circle*{5}}
\put(30,105){\circle*{5}}
\put(30,120){\circle*{5}}
\put(30,135){\circle*{5}}
\put(30,150){\circle*{5}}
\put(30,165){\circle*{5}}
\put(30,180){\circle*{5}}
\put(30,195){\circle*{5}}
%
\put(30,3){\vector(0,1){9}}
\put(30,27){\vector(0,-1){9}}
\put(30,33){\vector(0,1){9}}
\put(30,57){\vector(0,-1){9}}
\put(30,63){\vector(0,1){9}}
\put(30,87){\vector(0,-1){9}}
\put(30,93){\vector(0,1){9}}
\put(30,117){\vector(0,-1){9}}
\put(30,123){\vector(0,1){9}}
\put(30,147){\vector(0,-1){9}}
\put(30,153){\vector(0,1){9}}
\put(30,177){\vector(0,-1){9}}
\put(30,183){\vector(0,1){9}}
\put(0,33){\vector(0,1){39}}
\put(0,117){\vector(0,-1){39}}
\put(0,123){\vector(0,1){39}}
\put(27,17){\vector(-2,1){24}}
\put(3,30){\vector(1,0){24}}
\put(27,43){\vector(-2,-1){24}}
\put(3,73){\vector(2,-1){24}}
\put(27,75){\vector(-1,0){23}}
\put(3,77){\vector(2,1){24}}
\put(27,107){\vector(-2,1){24}}
\put(3,120){\vector(1,0){24}}
\put(27,133){\vector(-2,-1){24}}
\put(3,163){\vector(2,-1){24}}
\put(27,165){\vector(-1,0){23}}
\put(3,167){\vector(2,1){24}}
\put(3,-2)
{
\put(-17,30){\small II}
\put(-17,75){\small V}
\put(-17,120){\small II}
\put(-17,165){\small V}
}
\put(4,-2)
{
\put(30,0){\small $+$}
\put(30,15){\small $-$}
\put(30,30){\small $+$}
\put(30,45){\small $-$}
\put(30,60){\small $+$}
\put(30,75){\small $-$}
\put(30,90){\small $+$}
\put(30,105){\small $-$}
\put(30,120){\small $+$}
\put(30,135){\small $-$}
\put(30,150){\small $+$}
\put(30,165){\small $-$}
\put(30,180){\small $+$}
\put(30,195){\small $-$}
}
}
\put(160,0)
{
\put(0,30){\circle{5}}
\put(0,75){\circle{5}}
\put(0,120){\circle{5}}
\put(0,165){\circle{5}}
\put(30,0){\circle*{5}}
\put(30,15){\circle*{5}}
\put(30,30){\circle*{5}}
\put(30,45){\circle*{5}}
\put(30,60){\circle*{5}}
\put(30,75){\circle*{5}}
\put(30,90){\circle*{5}}
\put(30,105){\circle*{5}}
\put(30,120){\circle*{5}}
\put(30,135){\circle*{5}}
\put(30,150){\circle*{5}}
\put(30,165){\circle*{5}}
\put(30,180){\circle*{5}}
\put(30,195){\circle*{5}}
%
\put(30,3){\vector(0,1){9}}
\put(30,27){\vector(0,-1){9}}
\put(30,33){\vector(0,1){9}}
\put(30,57){\vector(0,-1){9}}
\put(30,63){\vector(0,1){9}}
\put(30,87){\vector(0,-1){9}}
\put(30,93){\vector(0,1){9}}
\put(30,117){\vector(0,-1){9}}
\put(30,123){\vector(0,1){9}}
\put(30,147){\vector(0,-1){9}}
\put(30,153){\vector(0,1){9}}
\put(30,177){\vector(0,-1){9}}
\put(30,183){\vector(0,1){9}}
\put(0,72){\vector(0,-1){39}}
\put(0,78){\vector(0,1){39}}
\put(0,162){\vector(0,-1){39}}
\put(3,30){\vector(1,0){24}}
\put(27,47){\vector(-1,1){24}}
\put(3,73){\vector(2,-1){24}}
\put(27,75){\vector(-1,0){23}}
\put(3,77){\vector(2,1){24}}
\put(27,103){\vector(-1,-1){24}}
\put(3,120){\vector(1,0){24}}
\put(27,137){\vector(-1,1){24}}
\put(3,163){\vector(2,-1){24}}
\put(27,165){\vector(-1,0){23}}
\put(3,167){\vector(2,1){24}}
\put(27,193){\vector(-1,-1){24}}
\put(3,-2)
{
\put(-17,30){\small VI}
\put(-17,75){\small III}
\put(-17,120){\small VI}
\put(-17,165){\small III}
}
\put(4,-2)
{
\put(30,0){\small $+$}
\put(30,15){\small $-$}
\put(30,30){\small $+$}
\put(30,45){\small $-$}
\put(30,60){\small $+$}
\put(30,75){\small $-$}
\put(30,90){\small $+$}
\put(30,105){\small $-$}
\put(30,120){\small $+$}
\put(30,135){\small $-$}
\put(30,150){\small $+$}
\put(30,165){\small $-$}
\put(30,180){\small $+$}
\put(30,195){\small $-$}
}
}

\put(-50,95.5){${\scriptstyle\ell -1}\left\{ \makebox(0,73){}\right.$}
\put(210,95.5){$\left. \makebox(0,103){}\right\}{\scriptstyle3\ell -1}$}
\end{picture}
\caption{The quiver $Q_{\ell}(G_2)$  for even $\ell$
(upper) and for odd $\ell$ (lower),
where we identify the right columns in the three quivers.}
\label{fig:quiverG}
\end{figure}
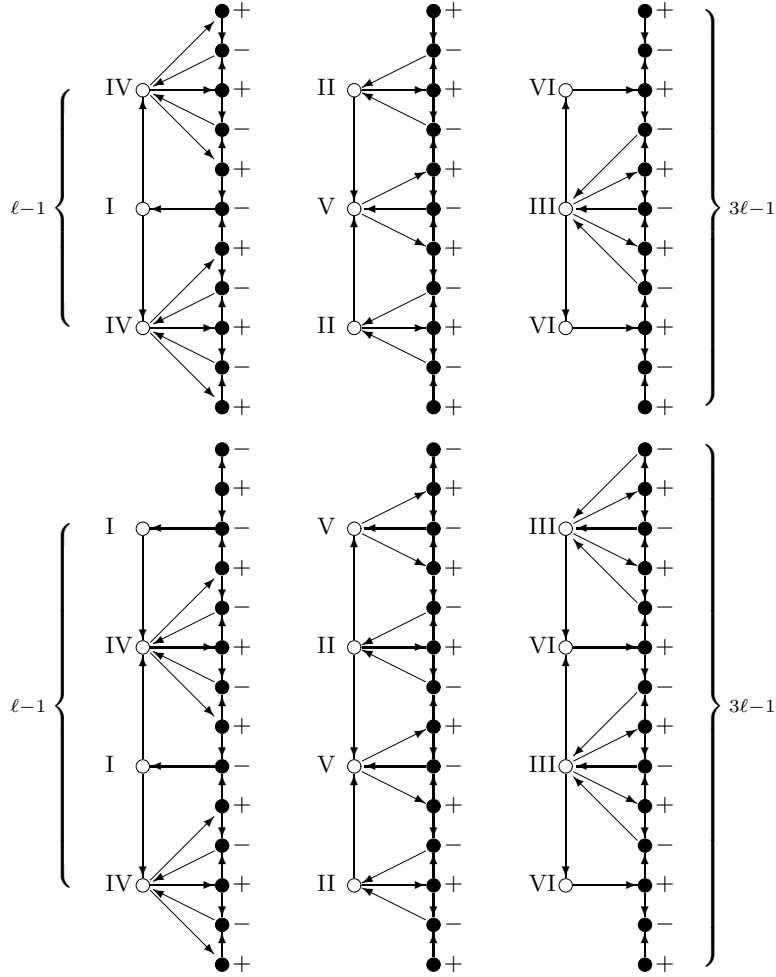

With type $G_2$ and $\ell\geq 2$ we associate
 the quiver $Q_{\ell}(G_2)$
by Figure
\ref{fig:quiverG},
where the right columns in the three quivers are
identified.
Also we assign the empty or filled circle $\circ$/$\bullet$
to each vertex;
furthermore, we assign the sign +/$-$ to each vertex
of property $\bullet$, and one of the numbers
I,\dots, VI to each vertex of property $\circ$.

Let us choose  the index set $\mathbf{I}$
of the vertices of $Q_{\ell}(G_2)$
so that $\mathbf{i}=(i,i')\in \mathbf{I}$ represents
the vertex 
at the $i'$th row (from the bottom)
of the left column in the $i$th quiver (from the left)
for $i=1,2,3$,
 and the one of the right column in any quiver
for $i=4$.
Thus, $i=1,\dots,4$, and $i'=1,\dots,\ell-1$ if $i\neq 4$
and $i'=1,\dots,3\ell-1$ if $i=4$.

For a permutation $s$ of $\{1,2,3\}$,
let $\boldsymbol\nu_{s}$ be the permutation of
$\mathbf{I}$ such that
$\boldsymbol\nu_{s}(i,i')=(s(i),i')$ for $i\neq 4$
and $(4,i')$ for $i=4$.
Let $\boldsymbol{\omega}$ be the involution acting on $\mathbf{I}$
by the up-down reflection.
Let $\boldsymbol\nu_{s}(Q_{\ell}(G_2))$
and $\boldsymbol\omega(Q_{\ell}(G_2))$ denote the
quivers induced from $Q_{\ell}(G_2)$
 by
$\boldsymbol{\nu}_s$ and $\boldsymbol{\omega}$,
respectively.

\begin{lemma}
\label{lem:GQmut}
Let $Q=Q_{\ell}(G_2)$.
\par
(i)
We have a periodic sequence of mutations of quivers
\begin{align}
\label{eq:GB2}
\begin{matrix}
Q
& 
\displaystyle
\mathop{\longleftrightarrow}^{\mu^{\bullet}_+
\mu^{\circ}_{\mathrm{I}}}
&\boldsymbol{\nu}_{(23)}(Q)^{\mathrm{op}}
&
\displaystyle
\mathop{\longleftrightarrow}^{\mu^{\bullet}_-
\mu^{\circ}_{\mathrm{II}}}
&
\boldsymbol{\nu}_{(312)}(Q)
&
\displaystyle
\mathop{\longleftrightarrow}^{\mu^{\bullet}_+
\mu^{\circ}_{\mathrm{III}}}
&
\boldsymbol{\nu}_{(13)}(Q)^{\mathrm{op}}
\\
&
\displaystyle
\mathop{\longleftrightarrow}^{\mu^{\bullet}_-
\mu^{\circ}_{\mathrm{IV}}}
&
\boldsymbol{\nu}_{(231)}(Q)
&
\displaystyle
\mathop{\longleftrightarrow}^{\mu^{\bullet}_+
\mu^{\circ}_{\mathrm{V}}}
&
\boldsymbol{\nu}_{(12)}(Q)^{\mathrm{op}}
&
\displaystyle
\mathop{\longleftrightarrow}^{\mu^{\bullet}_-
\mu^{\circ}_{\mathrm{VI}}}
&
Q.
\end{matrix}
\end{align}
\par
(ii)
 $\boldsymbol{\omega}(Q)=Q$ if $h^{\vee}+\ell$ is even,
and
$\boldsymbol{\omega}(Q)
=\boldsymbol{\nu}_{(13)}(Q)^{\mathrm{op}}$ if $h^{\vee}+\ell$ is odd.
\end{lemma}
See Figures \ref{fig:labelxG1}--\ref{fig:labelxG3}
for an example.

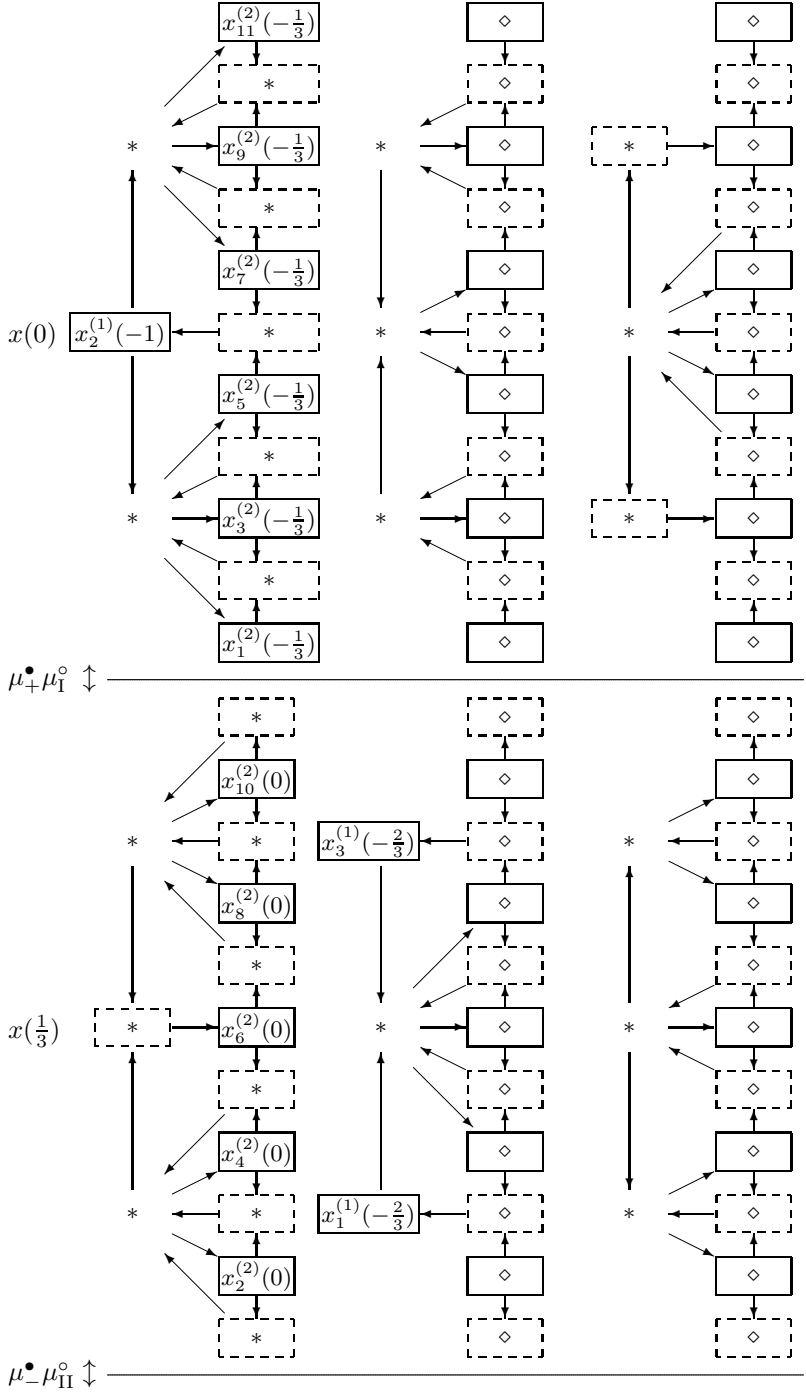
\begin{figure}
\setlength{\unitlength}{0.94pt}
\begin{picture}(500,570)(0,-305)
\put(0,121){$x(0)$}
%
%
\put(50,0)
{
\put(-15,-7.5)
{
\put(0,50){\makebox(30,15){\small $*$}}
\put(-10,125){\framebox(40,15){\small $x^{(1)}_2(-1)$}}
\put(0,200){\makebox(30,15){\small $*$}}
\put(50,0){\framebox(40,15){\small $x^{(2)}_1(-\frac{1}{3})$}}
\put(50,25){\dashbox{3}(40,15){\small $*$}}
\put(50,50){\framebox(40,15){\small $x^{(2)}_3(-\frac{1}{3})$}}
\put(50,75){\dashbox{3}(40,15){\small $*$}}
\put(50,100){\framebox(40,15){\small $x^{(2)}_5(-\frac{1}{3})$}}
\put(50,125){\dashbox{3}(40,15){\small $*$}}
\put(50,150){\framebox(40,15){\small $x^{(2)}_7(-\frac{1}{3})$}}
\put(50,175){\dashbox{3}(40,15){\small $*$}}
\put(50,200){\framebox(40,15){\small $x^{(2)}_9(-\frac{1}{3})$}}
\put(50,225){\dashbox{3}(40,15){\small $*$}}
\put(50,250){\framebox(40,15){\small $x^{(2)}_{11}(-\frac{1}{3})$}}
}
%
\put(13,34){$\vector(1,-1){24}$}
\put(34,33){$\vector(-2,1){18}$}
\put(16,50){$\vector(1,0){18}$}
\put(34,67){$\vector(-2,-1){18}$}
\put(13,66){$\vector(1,1){24}$}
\put(34,125){$\vector(-1,0){18}$}
\put(13,184){$\vector(1,-1){24}$}
\put(34,183){$\vector(-2,1){18}$}
\put(16,200){$\vector(1,0){18}$}
\put(34,217){$\vector(-2,-1){18}$}
\put(13,216){$\vector(1,1){24}$}
\put(0,115){\vector(0,-1){55}}
\put(0,135){\vector(0,1){55}}
\put(50,8){\vector(0,1){9}}
\put(50,42){\vector(0,-1){9}}
\put(50,58){\vector(0,1){9}}
\put(50,92){\vector(0,-1){9}}
\put(50,108){\vector(0,1){9}}
\put(50,142){\vector(0,-1){9}}
\put(50,158){\vector(0,1){9}}
\put(50,192){\vector(0,-1){9}}
\put(50,208){\vector(0,1){9}}
\put(50,242){\vector(0,-1){9}}
}
%
%
\put(150,0)
{
\put(-15,-7.5)
{
\put(0, 50){\makebox(30,15){\small $*$}}
\put(0,125){\makebox(30,15){\small $*$}}
\put(0,200){\makebox(30,15){\small $*$}}
\put(50, 0){\framebox(30,15){\small $\diamond$}}
\put(50,25){\dashbox{3}(30,15){\small $\diamond$}}
\put(50,50){\framebox(30,15){\small $\diamond$}}
\put(50,75){\dashbox{3}(30,15){\small $\diamond$}}
\put(50,100){\framebox(30,15){\small $\diamond$}}
\put(50,125){\dashbox{3}(30,15){\small $\diamond$}}
\put(50,150){\framebox(30,15){\small $\diamond$}}
\put(50,175){\dashbox{3}(30,15){\small $\diamond$}}
\put(50,200){\framebox(30,15){\small $\diamond$}}
\put(50,225){\dashbox{3}(30,15){\small $\diamond$}}
\put(50,250){\framebox(30,15){\small $\diamond$}}
}
%
\put(34,33){$\vector(-2,1){18}$}
\put(16,50){$\vector(1,0){18}$}
\put(34,67){$\vector(-2,-1){18}$}
%
\put(16,117){$\vector(2,-1){18}$}
\put(34,125){$\vector(-1,0){18}$}
\put(16,133){$\vector(2,1){18}$}
%
\put(34,183){$\vector(-2,1){18}$}
\put(16,200){$\vector(1,0){18}$}
\put(34,217){$\vector(-2,-1){18}$}
\put(0,60){\vector(0,1){55}}
\put(0,190){\vector(0,-1){55}}
\put(50,8){\vector(0,1){9}}
\put(50,42){\vector(0,-1){9}}
\put(50,58){\vector(0,1){9}}
\put(50,92){\vector(0,-1){9}}
\put(50,108){\vector(0,1){9}}
\put(50,142){\vector(0,-1){9}}
\put(50,158){\vector(0,1){9}}
\put(50,192){\vector(0,-1){9}}
\put(50,208){\vector(0,1){9}}
\put(50,242){\vector(0,-1){9}}
} 
%
%
\put(250,0)
{
\put(-15,-7.5)
{
\put(0, 50){\dashbox{3}(30,15){\small $*$}}
\put(0,125){\makebox(30,15){\small $*$}}
\put(0,200){\dashbox{3}(30,15){\small $*$}}
\put(50, 0){\framebox(30,15){\small $\diamond$}}
\put(50,25){\dashbox{3}(30,15){\small $\diamond$}}
\put(50,50){\framebox(30,15){\small $\diamond$}}
\put(50,75){\dashbox{3}(30,15){\small $\diamond$}}
\put(50,100){\framebox(30,15){\small $\diamond$}}
\put(50,125){\dashbox{3}(30,15){\small $\diamond$}}
\put(50,150){\framebox(30,15){\small $\diamond$}}
\put(50,175){\dashbox{3}(30,15){\small $\diamond$}}
\put(50,200){\framebox(30,15){\small $\diamond$}}
\put(50,225){\dashbox{3}(30,15){\small $\diamond$}}
\put(50,250){\framebox(30,15){\small $\diamond$}}
}
%
\put(16,50){$\vector(1,0){18}$}
%
\put(37,85){$\vector(-1,1){24}$}
\put(16,117){$\vector(2,-1){18}$}
\put(34,125){$\vector(-1,0){18}$}
\put(16,133){$\vector(2,1){18}$}
\put(37,165){$\vector(-1,-1){24}$}
%
\put(16,200){$\vector(1,0){18}$}
\put(0,115){\vector(0,-1){55}}
\put(0,135){\vector(0,1){55}}
\put(50,8){\vector(0,1){9}}
\put(50,42){\vector(0,-1){9}}
\put(50,58){\vector(0,1){9}}
\put(50,92){\vector(0,-1){9}}
\put(50,108){\vector(0,1){9}}
\put(50,142){\vector(0,-1){9}}
\put(50,158){\vector(0,1){9}}
\put(50,192){\vector(0,-1){9}}
\put(50,208){\vector(0,1){9}}
\put(50,242){\vector(0,-1){9}}
} 
%
\dottedline(40,-15)(320,-15)
\put(30,-17){$\updownarrow$}
\put(0,-17){$\mu^{\bullet}_+ \mu^{\circ}_{\mathrm{I}}$}
\put(0,-280)
{
\put(0,121){$x(\frac{1}{3})$}
%
%
\put(50,0)
{
\put(-15,-7.5)
{
\put(0,50){\makebox(30,15){\small $*$}}
\put(0,125){\dashbox{3}(30,15){\small $*$}}
\put(0,200){\makebox(30,15){\small $*$}}
\put(50,0){\dashbox{3}(30,15){\small $*$}}
\put(50,25){\framebox(30,15){\small $x^{(2)}_2(0)$}}
\put(50,50){\dashbox{3}(30,15){\small $*$}}
\put(50,75){\framebox(30,15){\small $x^{(2)}_4(0)$}}
\put(50,100){\dashbox{3}(30,15){\small $*$}}
\put(50,125){\framebox(30,15){\small $x^{(2)}_6(0)$}}
\put(50,150){\dashbox{3}(30,15){\small $*$}}
\put(50,175){\framebox(30,15){\small $x^{(2)}_8(0)$}}
\put(50,200){\dashbox{3}(30,15){\small $*$}}
\put(50,225){\framebox(30,15){\small $x^{(2)}_{10}(0)$}}
\put(50,250){\dashbox{3}(30,15){\small $*$}}
}
%
\put(37,10){$\vector(-1,1){24}$}
\put(16,42){$\vector(2,-1){18}$}
\put(34,50){$\vector(-1,0){18}$}
\put(16,58){$\vector(2,1){18}$}
\put(37,90){$\vector(-1,-1){24}$}
\put(16,125){$\vector(1,0){18}$}
\put(37,160){$\vector(-1,1){24}$}
\put(16,192){$\vector(2,-1){18}$}
\put(34,200){$\vector(-1,0){18}$}
\put(16,208){$\vector(2,1){18}$}
\put(37,240){$\vector(-1,-1){24}$}
\put(0,60){\vector(0,1){55}}
\put(0,190){\vector(0,-1){55}}
\put(50,17){\vector(0,-1){9}}
\put(50,33){\vector(0,1){9}}
\put(50,67){\vector(0,-1){9}}
\put(50,83){\vector(0,1){9}}
\put(50,117){\vector(0,-1){9}}
\put(50,133){\vector(0,1){9}}
\put(50,167){\vector(0,-1){9}}
\put(50,183){\vector(0,1){9}}
\put(50,217){\vector(0,-1){9}}
\put(50,233){\vector(0,1){9}}
}
%
%
\put(150,0)
{
\put(-15,-7.5)
{
\put(-10, 50){\framebox(40,15){\small $x^{(1)}_1(-\frac{2}{3})$}}
\put(0,125){\makebox(30,15){\small $*$}}
\put(-10,200){\framebox(40,15){\small $x^{(1)}_3(-\frac{2}{3})$}}
\put(50, 0){\dashbox{3}(30,15){\small $\diamond$}}
\put(50,25){\framebox(30,15){\small $\diamond$}}
\put(50,50){\dashbox{3}(30,15){\small $\diamond$}}
\put(50,75){\framebox(30,15){\small $\diamond$}}
\put(50,100){\dashbox{3}(30,15){\small $\diamond$}}
\put(50,125){\framebox(30,15){\small $\diamond$}}
\put(50,150){\dashbox{3}(30,15){\small $\diamond$}}
\put(50,175){\framebox(30,15){\small $\diamond$}}
\put(50,200){\dashbox{3}(30,15){\small $\diamond$}}
\put(50,225){\framebox(30,15){\small $\diamond$}}
\put(50,250){\dashbox{3}(30,15){\small $\diamond$}}
}
%
\put(34,50){$\vector(-1,0){18}$}
\put(13,109){$\vector(1,-1){24}$}
\put(34,108){$\vector(-2,1){18}$}
\put(16,125){$\vector(1,0){18}$}
\put(34,142){$\vector(-2,-1){18}$}
\put(13,141){$\vector(1,1){24}$}
\put(34,200){$\vector(-1,0){18}$}
\put(0,60){\vector(0,1){55}}
\put(0,190){\vector(0,-1){55}}
\put(50,17){\vector(0,-1){9}}
\put(50,33){\vector(0,1){9}}
\put(50,67){\vector(0,-1){9}}
\put(50,83){\vector(0,1){9}}
\put(50,117){\vector(0,-1){9}}
\put(50,133){\vector(0,1){9}}
\put(50,167){\vector(0,-1){9}}
\put(50,183){\vector(0,1){9}}
\put(50,217){\vector(0,-1){9}}
\put(50,233){\vector(0,1){9}}
} 
%
%
\put(250,0)
{
\put(-15,-7.5)
{
\put(0, 50){\makebox(30,15){\small $*$}}
\put(0,125){\makebox(30,15){\small $*$}}
\put(0,200){\makebox(30,15){\small $*$}}
\put(50, 0){\dashbox{3}(30,15){\small $\diamond$}}
\put(50,25){\framebox(30,15){\small $\diamond$}}
\put(50,50){\dashbox{3}(30,15){\small $\diamond$}}
\put(50,75){\framebox(30,15){\small $\diamond$}}
\put(50,100){\dashbox{3}(30,15){\small $\diamond$}}
\put(50,125){\framebox(30,15){\small $\diamond$}}
\put(50,150){\dashbox{3}(30,15){\small $\diamond$}}
\put(50,175){\framebox(30,15){\small $\diamond$}}
\put(50,200){\dashbox{3}(30,15){\small $\diamond$}}
\put(50,225){\framebox(30,15){\small $\diamond$}}
\put(50,250){\dashbox{3}(30,15){\small $\diamond$}}
}
%
\put(16,42){$\vector(2,-1){18}$}
\put(34,50){$\vector(-1,0){18}$}
\put(16,58){$\vector(2,1){18}$}
\put(34,108){$\vector(-2,1){18}$}
\put(16,125){$\vector(1,0){18}$}
\put(34,142){$\vector(-2,-1){18}$}
%
\put(16,192){$\vector(2,-1){18}$}
\put(34,200){$\vector(-1,0){18}$}
\put(16,208){$\vector(2,1){18}$}
\put(0,115){\vector(0,-1){55}}
\put(0,135){\vector(0,1){55}}
\put(50,17){\vector(0,-1){9}}
\put(50,33){\vector(0,1){9}}
\put(50,67){\vector(0,-1){9}}
\put(50,83){\vector(0,1){9}}
\put(50,117){\vector(0,-1){9}}
\put(50,133){\vector(0,1){9}}
\put(50,167){\vector(0,-1){9}}
\put(50,183){\vector(0,1){9}}
\put(50,217){\vector(0,-1){9}}
\put(50,233){\vector(0,1){9}}
} 
%
\dottedline(40,-15)(320,-15)
\put(30,-17){$\updownarrow$}
\put(0,-17){$\mu^{\bullet}_{-}\mu^{\circ}_{\mathrm{II}}$}
}
\end{picture}
\caption{(Continues to Figures \ref{fig:labelxG2}
and \ref{fig:labelxG3})
Label of cluster variables $x_{\mathbf{i}}(u)$
by $\mathcal{I}_{\ell+}$  for
$G_2$, $\ell=4$.
The right columns in the middle
and right quivers (marked by $\diamond$)
are identified with the right column in the left quiver.
}
\label{fig:labelxG1}
\end{figure}


\begin{figure}
\setlength{\unitlength}{0.94pt}
\begin{picture}(500,570)(0,-305)
\put(0,121){$x(\frac{2}{3})$}
%
%
\put(50,0)
{
\put(-15,-7.5)
{
\put(0, 50){\makebox(30,15){\small $*$}}
\put(0,125){\makebox(30,15){\small $*$}}
\put(0,200){\makebox(30,15){\small $*$}}
\put(50,0){\framebox(30,15){\small $x^{(2)}_1(\frac{1}{3})$}}
\put(50,25){\dashbox{3}(30,15){\small $*$}}
\put(50,50){\framebox(30,15){\small $x^{(2)}_3(\frac{1}{3})$}}
\put(50,75){\dashbox{3}(30,15){\small $*$}}
\put(50,100){\framebox(30,15){\small $x^{(2)}_5(\frac{1}{3})$}}
\put(50,125){\dashbox{3}(30,15){\small $*$}}
\put(50,150){\framebox(30,15){\small $x^{(2)}_7(\frac{1}{3})$}}
\put(50,175){\dashbox{3}(30,15){\small $*$}}
\put(50,200){\framebox(30,15){\small $x^{(2)}_9(\frac{1}{3})$}}
\put(50,225){\dashbox{3}(30,15){\small $*$}}
\put(50,250){\framebox(30,15){\small $x^{(2)}_{11}(\frac{1}{3})$}}
}
%
\put(34,33){$\vector(-2,1){18}$}
\put(16,50){$\vector(1,0){18}$}
\put(34,67){$\vector(-2,-1){18}$}
%
\put(16,117){$\vector(2,-1){18}$}
\put(34,125){$\vector(-1,0){18}$}
\put(16,133){$\vector(2,1){18}$}
%
\put(34,183){$\vector(-2,1){18}$}
\put(16,200){$\vector(1,0){18}$}
\put(34,217){$\vector(-2,-1){18}$}
\put(0,60){\vector(0,1){55}}
\put(0,190){\vector(0,-1){55}}
\put(50,8){\vector(0,1){9}}
\put(50,42){\vector(0,-1){9}}
\put(50,58){\vector(0,1){9}}
\put(50,92){\vector(0,-1){9}}
\put(50,108){\vector(0,1){9}}
\put(50,142){\vector(0,-1){9}}
\put(50,158){\vector(0,1){9}}
\put(50,192){\vector(0,-1){9}}
\put(50,208){\vector(0,1){9}}
\put(50,242){\vector(0,-1){9}}
} 
%
%
\put(150,0)
{
\put(-15,-7.5)
{
\put(0, 50){\dashbox{3}(30,15){\small $*$}}
\put(0,125){\makebox(30,15){\small $*$}}
\put(0,200){\dashbox{3}(30,15){\small $*$}}
\put(50, 0){\framebox(30,15){\small $\diamond$}}
\put(50,25){\dashbox{3}(30,15){\small $\diamond$}}
\put(50,50){\framebox(30,15){\small $\diamond$}}
\put(50,75){\dashbox{3}(30,15){\small $\diamond$}}
\put(50,100){\framebox(30,15){\small $\diamond$}}
\put(50,125){\dashbox{3}(30,15){\small $\diamond$}}
\put(50,150){\framebox(30,15){\small $\diamond$}}
\put(50,175){\dashbox{3}(30,15){\small $\diamond$}}
\put(50,200){\framebox(30,15){\small $\diamond$}}
\put(50,225){\dashbox{3}(30,15){\small $\diamond$}}
\put(50,250){\framebox(30,15){\small $\diamond$}}
}
%
\put(16,50){$\vector(1,0){18}$}
%
\put(37,85){$\vector(-1,1){24}$}
\put(16,117){$\vector(2,-1){18}$}
\put(34,125){$\vector(-1,0){18}$}
\put(16,133){$\vector(2,1){18}$}
\put(37,165){$\vector(-1,-1){24}$}
%
\put(16,200){$\vector(1,0){18}$}
\put(0,115){\vector(0,-1){55}}
\put(0,135){\vector(0,1){55}}
\put(50,8){\vector(0,1){9}}
\put(50,42){\vector(0,-1){9}}
\put(50,58){\vector(0,1){9}}
\put(50,92){\vector(0,-1){9}}
\put(50,108){\vector(0,1){9}}
\put(50,142){\vector(0,-1){9}}
\put(50,158){\vector(0,1){9}}
\put(50,192){\vector(0,-1){9}}
\put(50,208){\vector(0,1){9}}
\put(50,242){\vector(0,-1){9}}
} 
%
%
%
\put(250,0)
{
\put(-15,-7.5)
{
\put(0,50){\makebox(30,15){\small $*$}}
\put(-10,125){\framebox(40,15){\small $x^{(1)}_2(-\frac{1}{3})$}}
\put(0,200){\makebox(30,15){\small $*$}}
\put(50, 0){\framebox(30,15){\small $\diamond$}}
\put(50,25){\dashbox{3}(30,15){\small $\diamond$}}
\put(50,50){\framebox(30,15){\small $\diamond$}}
\put(50,75){\dashbox{3}(30,15){\small $\diamond$}}
\put(50,100){\framebox(30,15){\small $\diamond$}}
\put(50,125){\dashbox{3}(30,15){\small $\diamond$}}
\put(50,150){\framebox(30,15){\small $\diamond$}}
\put(50,175){\dashbox{3}(30,15){\small $\diamond$}}
\put(50,200){\framebox(30,15){\small $\diamond$}}
\put(50,225){\dashbox{3}(30,15){\small $\diamond$}}
\put(50,250){\framebox(30,15){\small $\diamond$}}
}
%
\put(13,34){$\vector(1,-1){24}$}
\put(34,33){$\vector(-2,1){18}$}
\put(16,50){$\vector(1,0){18}$}
\put(34,67){$\vector(-2,-1){18}$}
\put(13,66){$\vector(1,1){24}$}
\put(34,125){$\vector(-1,0){18}$}
\put(13,184){$\vector(1,-1){24}$}
\put(34,183){$\vector(-2,1){18}$}
\put(16,200){$\vector(1,0){18}$}
\put(34,217){$\vector(-2,-1){18}$}
\put(13,216){$\vector(1,1){24}$}
\put(0,115){\vector(0,-1){55}}
\put(0,135){\vector(0,1){55}}
\put(50,8){\vector(0,1){9}}
\put(50,42){\vector(0,-1){9}}
\put(50,58){\vector(0,1){9}}
\put(50,92){\vector(0,-1){9}}
\put(50,108){\vector(0,1){9}}
\put(50,142){\vector(0,-1){9}}
\put(50,158){\vector(0,1){9}}
\put(50,192){\vector(0,-1){9}}
\put(50,208){\vector(0,1){9}}
\put(50,242){\vector(0,-1){9}}
}

\dottedline(40,-15)(320,-15)
\put(30,-17){$\updownarrow$}
\put(0,-17){$\mu^{\bullet}_+ \mu^{\circ}_{\mathrm{III}}$}
\put(0,-280)
{
\put(0,121){$x(1)$}
%
%
\put(50,0)
{
\put(-15,-7.5)
{
\put(0, 50){\framebox(30,15){\small $x^{(1)}_1(0)$}}
\put(0,125){\makebox(30,15){\small $*$}}
\put(0,200){\framebox(30,15){\small $x^{(1)}_3(0)$}}
\put(50,0){\dashbox{3}(30,15){\small $*$}}
\put(50,25){\framebox(30,15){\small $x^{(2)}_2(\frac{2}{3})$}}
\put(50,50){\dashbox{3}(30,15){\small $*$}}
\put(50,75){\framebox(30,15){\small $x^{(2)}_4(\frac{2}{3})$}}
\put(50,100){\dashbox{3}(30,15){\small $*$}}
\put(50,125){\framebox(30,15){\small $x^{(2)}_6(\frac{2}{3})$}}
\put(50,150){\dashbox{3}(30,15){\small $*$}}
\put(50,175){\framebox(30,15){\small $x^{(2)}_8(\frac{2}{3})$}}
\put(50,200){\dashbox{3}(30,15){\small $*$}}
\put(50,225){\framebox(30,15){\small $x^{(2)}_{10}(\frac{2}{3})$}}
\put(50,250){\dashbox{3}(30,15){\small $*$}}
}
%
\put(34,50){$\vector(-1,0){18}$}
\put(13,109){$\vector(1,-1){24}$}
\put(34,108){$\vector(-2,1){18}$}
\put(16,125){$\vector(1,0){18}$}
\put(34,142){$\vector(-2,-1){18}$}
\put(13,141){$\vector(1,1){24}$}
\put(34,200){$\vector(-1,0){18}$}
\put(0,60){\vector(0,1){55}}
\put(0,190){\vector(0,-1){55}}
\put(50,17){\vector(0,-1){9}}
\put(50,33){\vector(0,1){9}}
\put(50,67){\vector(0,-1){9}}
\put(50,83){\vector(0,1){9}}
\put(50,117){\vector(0,-1){9}}
\put(50,133){\vector(0,1){9}}
\put(50,167){\vector(0,-1){9}}
\put(50,183){\vector(0,1){9}}
\put(50,217){\vector(0,-1){9}}
\put(50,233){\vector(0,1){9}}
} 
%
%
\put(150,0)
{
\put(-15,-7.5)
{
\put(0, 50){\makebox(30,15){\small $*$}}
\put(0,125){\makebox(30,15){\small $*$}}
\put(0,200){\makebox(30,15){\small $*$}}
\put(50, 0){\dashbox{3}(30,15){\small $\diamond$}}
\put(50,25){\framebox(30,15){\small $\diamond$}}
\put(50,50){\dashbox{3}(30,15){\small $\diamond$}}
\put(50,75){\framebox(30,15){\small $\diamond$}}
\put(50,100){\dashbox{3}(30,15){\small $\diamond$}}
\put(50,125){\framebox(30,15){\small $\diamond$}}
\put(50,150){\dashbox{3}(30,15){\small $\diamond$}}
\put(50,175){\framebox(30,15){\small $\diamond$}}
\put(50,200){\dashbox{3}(30,15){\small $\diamond$}}
\put(50,225){\framebox(30,15){\small $\diamond$}}
\put(50,250){\dashbox{3}(30,15){\small $\diamond$}}
}
%
\put(16,42){$\vector(2,-1){18}$}
\put(34,50){$\vector(-1,0){18}$}
\put(16,58){$\vector(2,1){18}$}
\put(34,108){$\vector(-2,1){18}$}
\put(16,125){$\vector(1,0){18}$}
\put(34,142){$\vector(-2,-1){18}$}
%
\put(16,192){$\vector(2,-1){18}$}
\put(34,200){$\vector(-1,0){18}$}
\put(16,208){$\vector(2,1){18}$}
\put(0,115){\vector(0,-1){55}}
\put(0,135){\vector(0,1){55}}
\put(50,17){\vector(0,-1){9}}
\put(50,33){\vector(0,1){9}}
\put(50,67){\vector(0,-1){9}}
\put(50,83){\vector(0,1){9}}
\put(50,117){\vector(0,-1){9}}
\put(50,133){\vector(0,1){9}}
\put(50,167){\vector(0,-1){9}}
\put(50,183){\vector(0,1){9}}
\put(50,217){\vector(0,-1){9}}
\put(50,233){\vector(0,1){9}}
} 
%
%
%
\put(250,0)
{
\put(-15,-7.5)
{
\put(0,50){\makebox(30,15){\small $*$}}
\put(0,125){\dashbox{3}(30,15){\small $*$}}
\put(0,200){\makebox(30,15){\small $*$}}
\put(50, 0){\dashbox{3}(30,15){\small $\diamond$}}
\put(50,25){\framebox(30,15){\small $\diamond$}}
\put(50,50){\dashbox{3}(30,15){\small $\diamond$}}
\put(50,75){\framebox(30,15){\small $\diamond$}}
\put(50,100){\dashbox{3}(30,15){\small $\diamond$}}
\put(50,125){\framebox(30,15){\small $\diamond$}}
\put(50,150){\dashbox{3}(30,15){\small $\diamond$}}
\put(50,175){\framebox(30,15){\small $\diamond$}}
\put(50,200){\dashbox{3}(30,15){\small $\diamond$}}
\put(50,225){\framebox(30,15){\small $\diamond$}}
\put(50,250){\dashbox{3}(30,15){\small $\diamond$}}
}
%
\put(37,10){$\vector(-1,1){24}$}
\put(16,42){$\vector(2,-1){18}$}
\put(34,50){$\vector(-1,0){18}$}
\put(16,58){$\vector(2,1){18}$}
\put(37,90){$\vector(-1,-1){24}$}
\put(16,125){$\vector(1,0){18}$}
\put(37,160){$\vector(-1,1){24}$}
\put(16,192){$\vector(2,-1){18}$}
\put(34,200){$\vector(-1,0){18}$}
\put(16,208){$\vector(2,1){18}$}
\put(37,240){$\vector(-1,-1){24}$}
\put(0,60){\vector(0,1){55}}
\put(0,190){\vector(0,-1){55}}
\put(50,17){\vector(0,-1){9}}
\put(50,33){\vector(0,1){9}}
\put(50,67){\vector(0,-1){9}}
\put(50,83){\vector(0,1){9}}
\put(50,117){\vector(0,-1){9}}
\put(50,133){\vector(0,1){9}}
\put(50,167){\vector(0,-1){9}}
\put(50,183){\vector(0,1){9}}
\put(50,217){\vector(0,-1){9}}
\put(50,233){\vector(0,1){9}}
}
\dottedline(40,-15)(320,-15)
\put(30,-17){$\updownarrow$}
\put(0,-17){$\mu^{\bullet}_{-}\mu^{\circ}_{\mathrm{IV}}$}
}
\end{picture}
\caption{(Continues from Figure \ref{fig:labelxG1})}
\label{fig:labelxG2}
\end{figure}


\begin{figure}
\setlength{\unitlength}{0.94pt}
\begin{picture}(500,570)(0,-305)
\put(0,121){$x(\frac{4}{3})$}
%
%
\put(50,0)
{
\put(-15,-7.5)
{
\put(0, 50){\dashbox{3}(30,15){\small $*$}}
\put(0,125){\makebox(30,15){\small $*$}}
\put(0,200){\dashbox{3}(30,15){\small $*$}}
\put(50,0){\framebox(30,15){\small $x^{(2)}_1(1)$}}
\put(50,25){\dashbox{3}(30,15){\small $*$}}
\put(50,50){\framebox(30,15){\small $x^{(2)}_3(1)$}}
\put(50,75){\dashbox{3}(30,15){\small $*$}}
\put(50,100){\framebox(30,15){\small $x^{(2)}_5(1)$}}
\put(50,125){\dashbox{3}(30,15){\small $*$}}
\put(50,150){\framebox(30,15){\small $x^{(2)}_7(1)$}}
\put(50,175){\dashbox{3}(30,15){\small $*$}}
\put(50,200){\framebox(30,15){\small $x^{(2)}_9(1)$}}
\put(50,225){\dashbox{3}(30,15){\small $*$}}
\put(50,250){\framebox(30,15){\small $x^{(2)}_{11}(1)$}}
}
%
\put(16,50){$\vector(1,0){18}$}
%
\put(37,85){$\vector(-1,1){24}$}
\put(16,117){$\vector(2,-1){18}$}
\put(34,125){$\vector(-1,0){18}$}
\put(16,133){$\vector(2,1){18}$}
\put(37,165){$\vector(-1,-1){24}$}
%
\put(16,200){$\vector(1,0){18}$}
\put(0,115){\vector(0,-1){55}}
\put(0,135){\vector(0,1){55}}
\put(50,8){\vector(0,1){9}}
\put(50,42){\vector(0,-1){9}}
\put(50,58){\vector(0,1){9}}
\put(50,92){\vector(0,-1){9}}
\put(50,108){\vector(0,1){9}}
\put(50,142){\vector(0,-1){9}}
\put(50,158){\vector(0,1){9}}
\put(50,192){\vector(0,-1){9}}
\put(50,208){\vector(0,1){9}}
\put(50,242){\vector(0,-1){9}}
} 
%
%
%
\put(150,0)
{
\put(-15,-7.5)
{
\put(0,50){\makebox(30,15){\small $*$}}
\put(0,125){\framebox(30,15){\small $x^{(1)}_2(\frac{1}{3})$}}
\put(0,200){\makebox(30,15){\small $*$}}
\put(50, 0){\framebox(30,15){\small $\diamond$}}
\put(50,25){\dashbox{3}(30,15){\small $\diamond$}}
\put(50,50){\framebox(30,15){\small $\diamond$}}
\put(50,75){\dashbox{3}(30,15){\small $\diamond$}}
\put(50,100){\framebox(30,15){\small $\diamond$}}
\put(50,125){\dashbox{3}(30,15){\small $\diamond$}}
\put(50,150){\framebox(30,15){\small $\diamond$}}
\put(50,175){\dashbox{3}(30,15){\small $\diamond$}}
\put(50,200){\framebox(30,15){\small $\diamond$}}
\put(50,225){\dashbox{3}(30,15){\small $\diamond$}}
\put(50,250){\framebox(30,15){\small $\diamond$}}
}
%
\put(13,34){$\vector(1,-1){24}$}
\put(34,33){$\vector(-2,1){18}$}
\put(16,50){$\vector(1,0){18}$}
\put(34,67){$\vector(-2,-1){18}$}
\put(13,66){$\vector(1,1){24}$}
\put(34,125){$\vector(-1,0){18}$}
\put(13,184){$\vector(1,-1){24}$}
\put(34,183){$\vector(-2,1){18}$}
\put(16,200){$\vector(1,0){18}$}
\put(34,217){$\vector(-2,-1){18}$}
\put(13,216){$\vector(1,1){24}$}
\put(0,115){\vector(0,-1){55}}
\put(0,135){\vector(0,1){55}}
\put(50,8){\vector(0,1){9}}
\put(50,42){\vector(0,-1){9}}
\put(50,58){\vector(0,1){9}}
\put(50,92){\vector(0,-1){9}}
\put(50,108){\vector(0,1){9}}
\put(50,142){\vector(0,-1){9}}
\put(50,158){\vector(0,1){9}}
\put(50,192){\vector(0,-1){9}}
\put(50,208){\vector(0,1){9}}
\put(50,242){\vector(0,-1){9}}
}
%
%
\put(250,0)
{
\put(-15,-7.5)
{
\put(0, 50){\makebox(30,15){\small $*$}}
\put(0,125){\makebox(30,15){\small $*$}}
\put(0,200){\makebox(30,15){\small $*$}}
\put(50, 0){\framebox(30,15){\small $\diamond$}}
\put(50,25){\dashbox{3}(30,15){\small $\diamond$}}
\put(50,50){\framebox(30,15){\small $\diamond$}}
\put(50,75){\dashbox{3}(30,15){\small $\diamond$}}
\put(50,100){\framebox(30,15){\small $\diamond$}}
\put(50,125){\dashbox{3}(30,15){\small $\diamond$}}
\put(50,150){\framebox(30,15){\small $\diamond$}}
\put(50,175){\dashbox{3}(30,15){\small $\diamond$}}
\put(50,200){\framebox(30,15){\small $\diamond$}}
\put(50,225){\dashbox{3}(30,15){\small $\diamond$}}
\put(50,250){\framebox(30,15){\small $\diamond$}}
}
%
\put(34,33){$\vector(-2,1){18}$}
\put(16,50){$\vector(1,0){18}$}
\put(34,67){$\vector(-2,-1){18}$}
%
\put(16,117){$\vector(2,-1){18}$}
\put(34,125){$\vector(-1,0){18}$}
\put(16,133){$\vector(2,1){18}$}
%
\put(34,183){$\vector(-2,1){18}$}
\put(16,200){$\vector(1,0){18}$}
\put(34,217){$\vector(-2,-1){18}$}
\put(0,60){\vector(0,1){55}}
\put(0,190){\vector(0,-1){55}}
\put(50,8){\vector(0,1){9}}
\put(50,42){\vector(0,-1){9}}
\put(50,58){\vector(0,1){9}}
\put(50,92){\vector(0,-1){9}}
\put(50,108){\vector(0,1){9}}
\put(50,142){\vector(0,-1){9}}
\put(50,158){\vector(0,1){9}}
\put(50,192){\vector(0,-1){9}}
\put(50,208){\vector(0,1){9}}
\put(50,242){\vector(0,-1){9}}
} 

\dottedline(40,-15)(320,-15)
\put(30,-17){$\updownarrow$}
\put(0,-17){$\mu^{\bullet}_+ \mu^{\circ}_{\mathrm{V}}$}
\put(0,-280)
{
\put(0,121){$x(\frac{5}{3})$}
%
%
\put(50,0)
{
\put(-15,-7.5)
{
\put(0, 50){\makebox(30,15){\small $*$}}
\put(0,125){\makebox(30,15){\small $*$}}
\put(0,200){\makebox(30,15){\small $*$}}
\put(50,0){\dashbox{3}(30,15){\small $*$}}
\put(50,25){\framebox(30,15){\small $x^{(2)}_2(\frac{4}{3})$}}
\put(50,50){\dashbox{3}(30,15){\small $*$}}
\put(50,75){\framebox(30,15){\small $x^{(2)}_4(\frac{4}{3})$}}
\put(50,100){\dashbox{3}(30,15){\small $*$}}
\put(50,125){\framebox(30,15){\small $x^{(2)}_6(\frac{4}{3})$}}
\put(50,150){\dashbox{3}(30,15){\small $*$}}
\put(50,175){\framebox(30,15){\small $x^{(2)}_8(\frac{4}{3})$}}
\put(50,200){\dashbox{3}(30,15){\small $*$}}
\put(50,225){\framebox(30,15){\small $x^{(2)}_{10}(\frac{4}{3})$}}
\put(50,250){\dashbox{3}(30,15){\small $*$}}
}
%
\put(16,42){$\vector(2,-1){18}$}
\put(34,50){$\vector(-1,0){18}$}
\put(16,58){$\vector(2,1){18}$}
\put(34,108){$\vector(-2,1){18}$}
\put(16,125){$\vector(1,0){18}$}
\put(34,142){$\vector(-2,-1){18}$}
%
\put(16,192){$\vector(2,-1){18}$}
\put(34,200){$\vector(-1,0){18}$}
\put(16,208){$\vector(2,1){18}$}
\put(0,115){\vector(0,-1){55}}
\put(0,135){\vector(0,1){55}}
\put(50,17){\vector(0,-1){9}}
\put(50,33){\vector(0,1){9}}
\put(50,67){\vector(0,-1){9}}
\put(50,83){\vector(0,1){9}}
\put(50,117){\vector(0,-1){9}}
\put(50,133){\vector(0,1){9}}
\put(50,167){\vector(0,-1){9}}
\put(50,183){\vector(0,1){9}}
\put(50,217){\vector(0,-1){9}}
\put(50,233){\vector(0,1){9}}
} 

%
%
\put(150,0)
{
\put(-15,-7.5)
{
\put(0,50){\makebox(30,15){\small $*$}}
\put(0,125){\dashbox{3}(30,15){\small $*$}}
\put(0,200){\makebox(30,15){\small $*$}}

\put(50, 0){\dashbox{3}(30,15){\small $\diamond$}}
\put(50,25){\framebox(30,15){\small $\diamond$}}
\put(50,50){\dashbox{3}(30,15){\small $\diamond$}}
\put(50,75){\framebox(30,15){\small $\diamond$}}
\put(50,100){\dashbox{3}(30,15){\small $\diamond$}}
\put(50,125){\framebox(30,15){\small $\diamond$}}
\put(50,150){\dashbox{3}(30,15){\small $\diamond$}}
\put(50,175){\framebox(30,15){\small $\diamond$}}
\put(50,200){\dashbox{3}(30,15){\small $\diamond$}}
\put(50,225){\framebox(30,15){\small $\diamond$}}
\put(50,250){\dashbox{3}(30,15){\small $\diamond$}}
}
%
\put(37,10){$\vector(-1,1){24}$}
\put(16,42){$\vector(2,-1){18}$}
\put(34,50){$\vector(-1,0){18}$}
\put(16,58){$\vector(2,1){18}$}
\put(37,90){$\vector(-1,-1){24}$}
\put(16,125){$\vector(1,0){18}$}
\put(37,160){$\vector(-1,1){24}$}
\put(16,192){$\vector(2,-1){18}$}
\put(34,200){$\vector(-1,0){18}$}
\put(16,208){$\vector(2,1){18}$}
\put(37,240){$\vector(-1,-1){24}$}
\put(0,60){\vector(0,1){55}}
\put(0,190){\vector(0,-1){55}}
\put(50,17){\vector(0,-1){9}}
\put(50,33){\vector(0,1){9}}
\put(50,67){\vector(0,-1){9}}
\put(50,83){\vector(0,1){9}}
\put(50,117){\vector(0,-1){9}}
\put(50,133){\vector(0,1){9}}
\put(50,167){\vector(0,-1){9}}
\put(50,183){\vector(0,1){9}}
\put(50,217){\vector(0,-1){9}}
\put(50,233){\vector(0,1){9}}
}
%
%
\put(250,0)
{
\put(-15,-7.5)
{
\put(0, 50){\framebox(30,15){\small $x^{(1)}_1(\frac{2}{3})$}}
\put(0,125){\makebox(30,15){\small $*$}}
\put(0,200){\framebox(30,15){\small $x^{(1)}_3(\frac{2}{3})$}}
\put(50, 0){\dashbox{3}(30,15){\small $\diamond$}}
\put(50,25){\framebox(30,15){\small $\diamond$}}
\put(50,50){\dashbox{3}(30,15){\small $\diamond$}}
\put(50,75){\framebox(30,15){\small $\diamond$}}
\put(50,100){\dashbox{3}(30,15){\small $\diamond$}}
\put(50,125){\framebox(30,15){\small $\diamond$}}
\put(50,150){\dashbox{3}(30,15){\small $\diamond$}}
\put(50,175){\framebox(30,15){\small $\diamond$}}
\put(50,200){\dashbox{3}(30,15){\small $\diamond$}}
\put(50,225){\framebox(30,15){\small $\diamond$}}
\put(50,250){\dashbox{3}(30,15){\small $\diamond$}}
}
%
\put(34,50){$\vector(-1,0){18}$}
\put(13,109){$\vector(1,-1){24}$}
\put(34,108){$\vector(-2,1){18}$}
\put(16,125){$\vector(1,0){18}$}
\put(34,142){$\vector(-2,-1){18}$}
\put(13,141){$\vector(1,1){24}$}
\put(34,200){$\vector(-1,0){18}$}
\put(0,60){\vector(0,1){55}}
\put(0,190){\vector(0,-1){55}}
\put(50,17){\vector(0,-1){9}}
\put(50,33){\vector(0,1){9}}
\put(50,67){\vector(0,-1){9}}
\put(50,83){\vector(0,1){9}}
\put(50,117){\vector(0,-1){9}}
\put(50,133){\vector(0,1){9}}
\put(50,167){\vector(0,-1){9}}
\put(50,183){\vector(0,1){9}}
\put(50,217){\vector(0,-1){9}}
\put(50,233){\vector(0,1){9}}
} 
%
\dottedline(40,-15)(320,-15)
\put(30,-17){$\updownarrow$}
\put(0,-17){$\mu^{\bullet}_{-}\mu^{\circ}_{\mathrm{VI}}$}
}
\end{picture}
\caption{(Continues from Figure \ref{fig:labelxG2})}
\label{fig:labelxG3}
\end{figure}

\subsection{Cluster algebra and alternative labels}
Let $B_{\ell}(G_2)$
be the corresponding skew-symmetric matrix 
to the quiver $Q_{\ell}(G_2)$.
In the rest of the section,
we set the matrix $B=(B_{\mathbf{i}\mathbf{j}})_{\mathbf{i},
\mathbf{j}\in \mathbf{I}}=B_{\ell}(G_2)$
unless otherwise mentioned.

Let $\mathcal{A}(B,x,y)$ 
be the cluster algebra
 with coefficients
in the universal
semifield
$\mathbb{Q}_{\mathrm{sf}}(y)$,
and $\mathcal{G}(B,y)$ be the  coefficient group 
associated with $\mathcal{A}(B,x,y)$.

In view of Lemma \ref{lem:GQmut}
we set $x(0)=x$, $y(0)=y$ and define 
clusters $x(u)=(x_{\mathbf{i}}(u))_{\mathbf{i}\in \mathbf{I}}$
 ($u\in \frac{1}{3}\mathbb{Z}$)
 and coefficient tuples $y(u)=(y_\mathbf{i}(u))_{\mathbf{i}\in \mathbf{I}}$
 ($u\in \frac{1}{3}\mathbb{Z}$)
by the sequence of mutations
\begin{align}
\label{eq:Gmutseq}
\begin{matrix}
\cdots
&
\displaystyle
\mathop{\longleftrightarrow}^{\mu^{\bullet}_-
\mu^{\circ}_{\mathrm{VI}}}
&
(B,x(0),y(0))
& 
\displaystyle
\mathop{\longleftrightarrow}^{\mu^{\bullet}_+
\mu^{\circ}_{\mathrm{I}}}
&(-\boldsymbol{\nu}_{(23)}(B),
x(\frac{1}{3}),
y(\frac{1}{3}))\\
&
\displaystyle
\mathop{\longleftrightarrow}^{\mu^{\bullet}_-
\mu^{\circ}_{\mathrm{II}}}
&
(\boldsymbol{\nu}_{(312)}(B),
x(\frac{2}{3}),
y(\frac{2}{3}))
&
\displaystyle
\mathop{\longleftrightarrow}^{\mu^{\bullet}_+
\mu^{\circ}_{\mathrm{III}}}
&
(-\boldsymbol{\nu}_{(13)}(B),
x(1),
y(1))
\\
&
\displaystyle
\mathop{\longleftrightarrow}^{\mu^{\bullet}_-
\mu^{\circ}_{\mathrm{IV}}}
&
(\boldsymbol{\nu}_{(231)}(B),
x(\frac{4}{3}),
y(\frac{4}{3}))
&
\displaystyle
\mathop{\longleftrightarrow}^{\mu^{\bullet}_+
\mu^{\circ}_{\mathrm{V}}}
&
(-\boldsymbol{\nu}_{(12)}(B),
x(\frac{5}{3}),
y(\frac{5}{3}))
&
\displaystyle
\mathop{\longleftrightarrow}^{\mu^{\bullet}_-
\mu^{\circ}_{\mathrm{VI}}}
&
\cdots.
\end{matrix}
\end{align}
where $\boldsymbol{\nu}_s(B)=B'$ is defined
by $B'_{\boldsymbol{\nu}_s(\mathbf{i})
\boldsymbol{\nu}_s(\mathbf{j})}=B_{\mathbf{i}\mathbf{j}}$.

For a pair $(\mathbf{i},u)\in
 \mathbf{I}\times \frac{1}{3}\mathbb{Z}$,
we set the parity condition $\mathbf{p}_+$
and $\mathbf{p}_-$ by
\begin{align}
\mathbf{p}_+:&
\begin{cases}
 \mathbf{i}\in \mathbf{I}^{\circ}_{\mathrm{I}}
\sqcup \mathbf{I}^{\bullet}_+
& u\equiv 0\\
 \mathbf{i}\in \mathbf{I}^{\circ}_{\mathrm{II}}
\sqcup \mathbf{I}^{\bullet}_-
& u\equiv \frac{1}{3}
\\
 \mathbf{i}\in \mathbf{I}^{\circ}_{\mathrm{III}}
\sqcup \mathbf{I}^{\bullet}_+
& u\equiv \frac{2}{3}\\
 \mathbf{i}\in \mathbf{I}^{\circ}_{\mathrm{IV}}
\sqcup \mathbf{I}^{\bullet}_-
& u\equiv 1
\\
 \mathbf{i}\in \mathbf{I}^{\circ}_{\mathrm{V}}
\sqcup \mathbf{I}^{\bullet}_+
& u\equiv \frac{4}{3}\\
 \mathbf{i}\in \mathbf{I}^{\circ}_{\mathrm{VI}}
\sqcup \mathbf{I}^{\bullet}_-
& u\equiv \frac{5}{3},
\\
\end{cases}
\qquad
\mathbf{p}_-:
\begin{cases}
 \mathbf{i}\in \mathbf{I}^{\circ}_{\mathrm{VI}}
\sqcup \mathbf{I}^{\bullet}_-
& u\equiv 0\\
 \mathbf{i}\in \mathbf{I}^{\circ}_{\mathrm{I}}
\sqcup \mathbf{I}^{\bullet}_+
& u\equiv \frac{1}{3}
\\
 \mathbf{i}\in \mathbf{I}^{\circ}_{\mathrm{II}}
\sqcup \mathbf{I}^{\bullet}_-
& u\equiv \frac{2}{3}\\
 \mathbf{i}\in \mathbf{I}^{\circ}_{\mathrm{III}}
\sqcup \mathbf{I}^{\bullet}_+
& u\equiv 1
\\
 \mathbf{i}\in \mathbf{I}^{\circ}_{\mathrm{IV}}
\sqcup \mathbf{I}^{\bullet}_-
& u\equiv \frac{4}{3}\\
 \mathbf{i}\in \mathbf{I}^{\circ}_{\mathrm{V}}
\sqcup \mathbf{I}^{\bullet}_+
& u\equiv \frac{5}{3},
\\
\end{cases}
\end{align}
where $\equiv$ is modulo $2\mathbb{Z}$.
We have
\begin{align}
(\mathbf{i},u):\mathbf{p}_+
\quad
\Longleftrightarrow
\quad
(\mathbf{i},u+\frac{1}{3}):\mathbf{p}_-.
\end{align}
Each $(\mathbf{i},u):\mathbf{p}_+$
is a mutation point of \eqref{eq:Gmutseq} in the forward
direction of $u$,
and each  $(\mathbf{i},u):\mathbf{p}_-$
is so in the backward direction of $u$.

\begin{lemma}
 Below $\equiv$ means the equivalence modulo
$2\mathbb{Z}$.
\par
(i)
The map
\begin{align}
\begin{matrix}
g: &\mathcal{I}_{\ell+}&\rightarrow & \{ (\mathbf{i},u): \mathbf{p}_+
\}\hfill \\
&(a,m,u-\frac{1}{t_a})&\mapsto &
\begin{cases}
((1,m),u)& \mbox{\rm $a= 1$;
$m+u\equiv 0$}\\
((2,m),u)
& \mbox{\rm $a= 1$;
$m+u\equiv \frac{4}{3}$}\\
((3,m),u)&
 \mbox{\rm $a= 1$;
$m+u\equiv \frac{2}{3}$}\\
((4,m),u)& \mbox{\rm  $a=2$}\\
\end{cases}
\end{matrix}
\end{align}
is a bijection.
\par
(ii)
The map
\begin{align}
\begin{matrix}
g': &\mathcal{I}'_{\ell+}&\rightarrow & \{ (\mathbf{i},u): \mathbf{p}_+
\}\hfill \\
&(a,m,u)&\mapsto &
\begin{cases}
((1,m),u)& \mbox{\rm $a= 1$;
$m+u\equiv 0$}\\
((2,m),u)
& \mbox{\rm $a= 1$;
$m+u\equiv \frac{4}{3}$}\\
((3,m),u)&
 \mbox{\rm $a= 1$;
$m+u\equiv \frac{2}{3}$}\\
((4,m),u)& a=2\\
\end{cases}
\end{matrix}
\end{align}
is a bijection.
\end{lemma}

We introduce alternative labels
$x_{\mathbf{i}}(u)=x^{(a)}_m(u-1/t_a)$
($(a,m,u-1/t_a)\in \mathcal{I}_{\ell+}$)
for $(\mathbf{i},u)=g((a,m,u-1/t_a))$
and
$y_{\mathbf{i}}(u)=y^{(a)}_m(u)$
($(a,m,u)\in \mathcal{I}'_{\ell+}$)
for $(\mathbf{i},u)=g'((a,m,u))$,
respectively.
See Figures \ref{fig:labelxG1}--\ref{fig:labelxG3}
for an example.

\subsection{T-system and cluster algebra}


\begin{lemma}
\label{lem:Gx2}
The family $\{x^{(a)}_m(u)
\mid (a,m,u)\in \mathcal{I}_{\ell+}\}$
satisfies the system of relations
 \eqref{eq:Cx2}
 with $G(b,k,v;a,m,u)$
for $\mathbb{T}_{\ell}(G_2)$.
In particular,
the family $\{ [x^{(a)}_m(u)]_{\mathbf{1}}
\mid (a,m,u)\in \mathcal{I}_{\ell+}\}$
satisfies the T-system $\mathbb{T}_{\ell}(G_2)$
in $\mathcal{A}(B,x)$
by replacing $T^{(a)}_m(u)$ with $[x^{(a)}_m(u)]_{\mathbf{1}}$.
\end{lemma}

\begin{definition}
The {\em T-subalgebra
$\mathcal{A}_T(B,x)$
of ${\mathcal{A}}(B,x)$
associated with the sequence \eqref{eq:Gmutseq}}
is the subring of
${\mathcal{A}}(B,x)$
generated by
$[x_{\mathbf{i}}(u)]_{\mathbf{1}}$
($(\mathbf{i},u)\in \mathbf{I}\times \frac{1}{3}\mathbb{Z}$).
\end{definition}

\begin{theorem}
\label{thm:GTiso}
The ring $\EuScript{T}^{\circ}_{\ell}(G_2)_+$ is isomorphic to
$\mathcal{A}_T(B,x)$ by the correspondence
$T^{(a)}_m(u)\mapsto [x^{(a)}_m(u)]_{\mathbf{1}}$.
\end{theorem}

\subsection{Y-system and cluster algebra}

\begin{lemma}
\label{lem:Gy2}
The family $\{ y^{(a)}_m(u)
\mid (a,m,u)\in \mathcal{I}'_{\ell+}\}$
satisfies the Y-system $\mathbb{Y}_{\ell}(G_2)$
by replacing $Y^{(a)}_m(u)$ with $y^{(a)}_m(u)$.
\end{lemma}

\begin{definition}
The {\em Y-subgroup
$\mathcal{G}_Y(B,y)$
of ${\mathcal{G}}(B,y)$
associated with the sequence \eqref{eq:Gmutseq}}
is the subgroup of
${\mathcal{G}}(B,y)$ 
generated by
$y_{\mathbf{i}}(u)$
($(\mathbf{i},u)\in \mathbf{I}\times \frac{1}{3}\mathbb{Z}$)
and $1+y_{\mathbf{i}}(u)$
($(\mathbf{i},u):\mathbf{p}_+$ or $\mathbf{p}_-$).
\end{definition}

\begin{theorem}
\label{thm:GYiso}
The group $\EuScript{Y}^{\circ}_{\ell}(G_2)_+$ is isomorphic to
$\mathcal{G}_Y(B,y)$ by the correspondence
$Y^{(a)}_m(u)\mapsto y^{(a)}_m(u)$
and $1+Y^{(a)}_m(u)\mapsto 1+y^{(a)}_m(u)$.
\end{theorem}

\subsection{Tropical Y-system at level 2}

By direct computations, the following properties
are verified.

\begin{proposition}
\label{prop:Glev2}
 For 
$[\mathcal{G}_Y(B,y)]_{\mathbf{T}}$
with $B=B_{2}(G_2)$, the following facts hold.
\par
(i) Let $u$ be in the region $0\le u < 2$.
For any $(\mathbf{i},u):\mathbf{p}_+$,
the  monomial $[y_{\mathbf{i}}(u)]_{\mathbf{T}}$
is positive.
\par
(ii) Let $u$ be in the region $-h^{\vee}\le u < 0$.
\begin{itemize}
\item[\em (a)]
 Let $\mathbf{i}=(1,1),(2,1),(3,1)$, or $(4,3)$.
For any $(\mathbf{i},u):\mathbf{p}_+$,
the  monomial $[y_{\mathbf{i}}(u)]_{\mathbf{T}}$
is negative.
\item[\em (b)]
 Let $\mathbf{i}=(4,1),(4,2),(4,4)$, or $(4,5)$
For any $(\mathbf{i},u):\mathbf{p}_+$,
the  monomial $[y_{\mathbf{i}}(u)]_{\mathbf{T}}$
is negative for
 $u=-\frac{1}{3},-\frac{2}{3},-2,
-\frac{7}{3},-\frac{11}{3},-4$
and positive for $u=-1,-\frac{4}{3},-\frac{5}{3},
-\frac{8}{3},-3,
-\frac{10}{3}$.
\end{itemize}
\par
(iii)
$y_{ii'}(2)=y_{ii'}^{-1}$ if $i\neq 4$ and
$y_{4,6-i'}^{-1}$ if $i=4$.
\par
(iv) $y_{ii'}(-h^{\vee})=
y_{ii'}^{-1}$.
\end{proposition}

Also we have a description of the core part
of  $[y_{\mathbf{i}}(u)]_{\mathbf{T}}$
in the region $-h^{\vee}\leq  u <0$
in terms of the root system of $D_4$.
We use the following
index of the Dynkin diagram $D_4$.
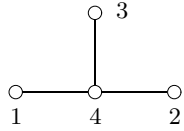
\begin{align*}
\begin{picture}(60,55)(0,0)
\put(0,15)
{
\put(0,0){\circle{5}}
\put(30,0){\circle{5}}
\put(60,0){\circle{5}}
\put(30,30){\circle{5}}
\put(27,0){\line(-1,0){24}}
\put(57,0){\line(-1,0){24}}
\put(30,3){\line(0,1){24}}
\put(-4,8)
{
\put(2,-20){\small $1$}
\put(32,-20){\small $4$}
\put(62,-20){\small $2$}
\put(42,20){\small $3$}
}
}
\end{picture}
\end{align*}
Let $\Pi=\{\alpha_1,\dots,\alpha_4\}$, $-\Pi$,
 $\Phi_+$ be the set of the simple roots,
the negative simple roots, the positive roots, respectively,
of type $D_{4}$.
Let $\sigma_i$ be the piecewise-linear
analogue  of the simple reflection $s_i$,
acting on the set
of the almost positive roots
$\Phi_{\geq -1}=\Phi_{+}\sqcup (-\Pi)$.
We define $\sigma$ as the composition
\begin{align}
\label{eq:Gsigma}
\sigma=\sigma_3\sigma_4\sigma_1\sigma_4\sigma_2\sigma_4.
\end{align}

\begin{lemma}
\label{lem:Gorbit}
We have the orbits by $\sigma$
\begin{align}
\begin{split}
&
-\alpha_1
\ \rightarrow\ 
\alpha_1+\alpha_3+\alpha_4
\ \rightarrow\ 
\alpha_1+\alpha_2+\alpha_4
\ \rightarrow\ 
-\alpha_1,
\\
&
-\alpha_2
\ \rightarrow\ 
\alpha_1+\alpha_2+\alpha_3+\alpha_4
\ \rightarrow\ 
\alpha_2+\alpha_4
\ \rightarrow\ 
-\alpha_2,
\\
&
-\alpha_3
\ \rightarrow\ 
\alpha_3
\ \rightarrow\ 
\alpha_1+\alpha_2+\alpha_3+2\alpha_4
\ \rightarrow\ 
-\alpha_3,
\\
&
-\alpha_4
\ \rightarrow\ 
\alpha_2+\alpha_3+\alpha_4
\ \rightarrow\ 
\alpha_4
\ \rightarrow\ 
\alpha_1
\\
&\phantom{-\alpha_3}\hskip5pt
\ \rightarrow\ 
\alpha_2
\ \rightarrow\ 
\alpha_3+\alpha_4
\ \rightarrow\ 
\alpha_1+\alpha_4
\ \rightarrow\ 
-\alpha_4.
\end{split}
\end{align}
In particular, these
elements in $\Phi_+$ exhaust the set $\Phi_+$,
thereby providing the orbit decomposition
of $\Phi_+$ by $\sigma$.
\end{lemma}

For $-h^{\vee}\leq u< 0$, define
\begin{align}
\label{eq:Galpha}
&\alpha_{i}(u)=
\begin{cases}
\sigma^{-(u-1)/2}(-\alpha_1)
& \mbox{\rm $i=1$; $u=-1,-3$},\\
\sigma^{-(3u-1)/6}(-\alpha_2)
& \mbox{\rm $i=2$; $u=-\frac{5}{3},-\frac{11}{3}$},\\
\sigma^{-(3u-5)/6}(-\alpha_3)
& \mbox{\rm $i=3$; $u=-\frac{1}{3},-\frac{7}{3}$},\\
\sigma^{-(3u+2)/6}(\alpha_3+\alpha_4)
& \mbox{\rm $i=4$; $u=-\frac{2}{3},-\frac{8}{3}$},\\
\sigma^{-(3u+4)/6}(\alpha_1)
& \mbox{\rm $i=4$; $u=-\frac{4}{3},-\frac{10}{3}$},\\
\sigma^{-u/2}(-\alpha_4)
& \mbox{\rm $i=4$; $u=-2,-4$}.\\
\end{cases}
\end{align}
By Lemma \ref{lem:Gorbit} they are (all the) positive roots
of $D_{4}$.

For a monomial $m$ in $y=(y_{\mathbf{i}})_{\mathbf{i}\in
\mathbf{I}}$,
let $\pi_D(m)$ denote the specialization
with $y_{41}=y_{42}=y_{44}=y_{45}=1$.
For simplicity, we set
$y_{i1}=y_i$ ($i\neq 4$), $y_{43}=y_4$,
and also,
$y_{i1}(u)=y_{i}(u)$
($i\neq 4$), $y_{43}(u)=y_4(u)$.
We define the vectors $\mathbf{t}_{i}(u)
=(t_{i}(u)_k)_{k=1}^{4}$
by
\begin{align}
\pi_D([y_{i}(u)]_{\mathbf{T}})
=
\prod_{k=1}^{4}
 y_{k}^{t_{i}(u)_k}.
\end{align}
We also identify each vector $\mathbf{t}_{i}(u)$
with $\alpha=\sum_{k=1}^{4}
t_{i}(u)_k \alpha_k \in \mathbb{Z}\Pi$.

\begin{proposition}
\label{prop:Gtvec}
Let $-h^{\vee}\leq u< 0$.
Then, we have
\begin{align}
\label{eq:Gtvec1}
\mathbf{t}_{i}(u)=-\alpha_i(u),
\end{align}
for $(i,u)$ in \eqref{eq:Galpha}.
\end{proposition}

\subsection{Tropical Y-systems of higher levels}

\begin{proposition}
\label{prop:Glevh}
Let $\ell> 2$ be an integer.
 For 
$[\mathcal{G}_Y(B,y)]_{\mathbf{T}}$
with $B=B_{\ell}(G_2)$, the following facts hold.
\par
\par
(i) Let $u$ be in the region $0\le u < \ell$.
For any $(\mathbf{i},u):\mathbf{p}_+$,
the  monomial $[y_{\mathbf{i}}(u)]_{\mathbf{T}}$
is positive.
\par
(ii) Let $u$ be in the region $-h^{\vee}\le u < 0$.
\begin{itemize}
\item[\em (a)]
 Let $\mathbf{i}\in \mathbf{I}^{\circ}$
or $(4,i')$ $(i'\in 3\mathbb{N})$.
For any $(\mathbf{i},u):\mathbf{p}_+$,
the  monomial $[y_{\mathbf{i}}(u)]_{\mathbf{T}}$
is negative.
\item[\em (b)]
 Let $\mathbf{i}=(4,i')$ $(i'\not\in 3\mathbb{N})$.
For any $(\mathbf{i},u):\mathbf{p}_+$,
the  monomial $[y_{\mathbf{i}}(u)]_{\mathbf{T}}$
is negative for
 $u=-\frac{1}{3},-\frac{2}{3},-2,
-\frac{7}{3},-\frac{11}{3},-4$
and positive for $u=-1,-\frac{4}{3},-\frac{5}{3},
-\frac{8}{3},
\\ 
-3,
-\frac{10}{3}$.
\end{itemize}
\par
(iii)
$y_{ii'}(\ell)=y_{i,\ell-i'}^{-1}$ if $i\neq 4$ and
$y_{4,3\ell-i'}^{-1}$ if $i=4$.
\par
(iv) $y_{ii'}(-h^{\vee})=
y_{ii'}^{-1}$.
\end{proposition}

We obtain corollaries of
Propositions \ref{prop:Glev2} and  \ref{prop:Glevh}.

\begin{theorem}
\label{thm:GtYperiod}
For $[\mathcal{G}_Y(B,y)]_{\mathbf{T}}$,
the following relations hold.
\par
(i) Half periodicity: 
$[y_{\mathbf{i}}(u+h^{\vee}+\ell)]_{\mathbf{T}}
=[y_{\boldsymbol{\omega}(\mathbf{i})}(u)]_{\mathbf{T}}$.
\par
(ii) 
 Full periodicity: 
$[y_{\mathbf{i}}(u+2(h^{\vee}+\ell))]_{\mathbf{T}}
=[y_{\mathbf{i}}(u)]_{\mathbf{T}}$.
\end{theorem}

\begin{theorem}
\label{thm:Glevhd}
For $[\mathcal{G}_Y(B,y)]_{\mathbf{T}}$,
let $N_+$ and $N_-$ denote the
total numbers of the positive and negative monomials,
respectively,
among $[y_{\mathbf{i}}(u)]_{\mathbf{T}}$
for $(\mathbf{i},u):\mathbf{p}_+$
in the region $0\leq u < 2(h^{\vee}+\ell)$.
Then, we have
\begin{align}
N_+=6\ell(2\ell+1),
\quad
N_-=12(3\ell-2).
\end{align}
\end{theorem}

\subsection{Periodicities and dilogarithm identities}

Applying \cite[Theorem 5.1]{IIKKN} to
Theorem  \ref{thm:GtYperiod},
we obtain the periodicities:

\begin{theorem}
\label{thm:Gxperiod}
For $\mathcal{A}(B,x,y)$,
the following relations hold.
\par
(i) Half periodicity: 
$x_{\mathbf{i}}(u+h^{\vee}+\ell)
=x_{\boldsymbol{\omega}(\mathbf{i})}(u)
$.
\par
(ii) 
 Full periodicity: 
$x_{\mathbf{i}}(u+2(h^{\vee}+\ell))
=x_{\mathbf{i}}(u)
$.
\end{theorem}

\begin{theorem}
\label{thm:Gyperiod}
For $\mathcal{G}(B,y)$,
the following relations hold.
\par
(i) Half periodicity: 
$y_{\mathbf{i}}(u+h^{\vee}+\ell)
=y_{\boldsymbol{\omega}(\mathbf{i})}(u)
$.
\par
(ii) 
 Full periodicity: 
$y_{\mathbf{i}}(u+2(h^{\vee}+\ell))
=y_{\mathbf{i}}(u)
$.
\end{theorem}

Then,
 Theorems \ref{thm:Tperiod} and \ref{thm:Yperiod}
 for $G_2$ follow from
Theorems
\ref{thm:GTiso}, \ref{thm:GYiso},
\ref{thm:Gxperiod}, and \ref{thm:Gyperiod}.
Furthermore, Theorem \ref{thm:DI2}  for $G_2$ is  obtained from
the above periodicities  and Theorem \ref{thm:Glevhd} as
in the $B_r$ case \cite[Section 6]{IIKKN}.

\section{Mutation equivalence of quivers}

Recall that two quivers $Q$ and $Q'$ are said to
be {\em mutation equivalent},
and denoted by $Q\sim Q'$ here,
if there is a quiver isomorphism from $Q$ to some
quiver obtained from $Q'$ by successive mutations.

Below we present several mutation equivalent pairs
of the quivers $Q_{\ell}(X_r)$,
though the list is not complete at all.
For simply laced $X_r$,
$Q_{\ell}(X_r)$ is the quiver defined as the square product
$\vec{X}_r\square \vec{A}_{\ell-1}$
in \cite[Section 8]{Ke}.

\begin{proposition}
We have the following mutation equivalences
of quivers.
\begin{align}
\begin{split}
Q_{2}(B_r)&\sim Q_2(D_{2r+1}),\\
Q_{2}(C_3)&\sim Q_3(D_{4}),\\
Q_{2}(F_4)&\sim Q_3(D_{5}),\\
Q_{3}(C_2)&\sim Q_4(A_{3}),\\
Q_{\ell}(G_2)&\sim Q_{\ell}(C_{3}).
\end{split}
\end{align}

\end{proposition}

\end{document}